\documentclass[11pt]{amsart}

\usepackage{epsfig,amsmath,amsfonts,latexsym, mathtools}
\usepackage[abs]{overpic}
\usepackage{color}
 \usepackage{palatino}
 \usepackage{hyperref}
 \usepackage{comment}

\newtheorem{theorem}{Theorem}
\newtheorem{question}[theorem]{Question}
\newtheorem{proposition}[theorem]{Proposition}
\newtheorem{lemma}[theorem]{Lemma}
\newtheorem{corollary}[theorem]{Corollary}
\newtheorem{conjecture}[theorem]{Conjecture}

\theoremstyle{definition}

\newtheorem{observation}[theorem]{Observation}
\newtheorem{example}[theorem]{Example}
\newtheorem{construction}[theorem]{Construction}

\theoremstyle{remark}
\newtheorem{remark}[theorem]{Remark}
\DeclareMathOperator{\Fill}{Fill}
\DeclareMathOperator{\Cob}{Cob}

\def\dfn#1{{\em #1}}
\def\R{\mathbb{R}}
\def\Z{\mathbb{Z}}

\def\n{\mathbf{n}}
\def\m{\mathbf{m}}
\def\BZ{\mathbf{Z}}

\numberwithin{theorem}{section}
\theoremstyle{plain}

\begin{document}

\title{Symplectic fillings and cobordisms of lens spaces} 

\author{John B. Etnyre}

\author{Agniva Roy}

\address{School of Mathematics \\ Georgia Institute
of Technology \\  Atlanta  \\ Georgia}

\email{etnyre@math.gatech.edu}

\email{aroy86@gatech.edu}


\begin{abstract}
We complete the classification of symplectic fillings of tight contact structures on lens spaces. In particular, we show that any symplectic filling $X$ of a virtually overtwisted contact structure on $L(p,q)$ has another symplectic structure that fills the universally tight contact structure on $L(p,q)$. Moreover, we show that the Stein filling of $L(p,q)$ with maximal second homology is given by the plumbing of disk bundles. We also consider the question of constructing symplectic cobordisms between lens spaces and report some partial results.
\end{abstract}

\maketitle

\section{Introduction}
Studying symplectic fillings of contact manifolds has a long history and is useful in constructing symplectic manifolds. There have only been a few results where one can classify all the symplectic fillings of a given contact manifold. We will restrict our attention to symplectic fillings of lens spaces. The first classification result was for the standard tight contact structure on $S^3$. Eliashberg \cite{Eliashberg90b} showed that any symplectic filling of this manifold was a blowup of the standard symplectic $B^4$. This was extended by McDuff \cite{McDuff90} to the universally tight contact structure $\xi_{ut}$ on the lens space $L(p,1)$. She showed that when $p\not=4$ the only minimal symplectic filling (that is, one that is not a blowup) of $(L(p,1),\xi_{ut})$ is given by a disk bundle $E_p$ over $S^2$ with Euler number $-p$. But when $p=4$ there are two minimal symplectic fillings: $E_4$ and a rational homology ball. This classification is up to diffeomorphism, but was later extended by Hind to a classification up to symplectic deformation \cite{Hind03}. 

The next breakthrough was due to Lisca \cite{Lisca08} who classified the symplectic fillings of the universally tight contact structure on any $L(p,q)$. Here the classification is much more complicated and is recalled in Section~\ref{lclass} below. Recall any lens space has a unique (if the orientation of the contact planes is ignored) universally tight contact structure, while in many cases a lens space can have many virtually overtwisted contact structures. The first result about fillings of such contact structures was due to Plamenevskaya and van Horn-Morris \cite{PlamenevskayaVanHorn-Morris2010}. They showed that any virtually overtwisted contact structure on $L(p,1)$ has a unique minimal symplectic filling. They established their result using a powerful technique of Wendl \cite{Wendl10} that reduces the understanding of Stein fillings of a lens space (or any contact structure supported by a planar open book) to questions about factoring the monodromy of a planar open book supporting the contact structure. These tools were further used by Kaloti \cite{Kaloti13pre} to study other contact structures on lens spaces. Most recently Fossati \cite{Fossati19pre} used these techniques together with a recent result of Menke \cite{menke18pre} to classify all symplectic fillings of lens spaces obtained by (integral) surgery on the Hopf link.  

We note that all the classification results above for $L(p,1)$ and Kaloti's work are classifications of minimal symplectic fillings up to symplectic deformation, while the other classification results are up to diffeomorphism. In fact, in Lisca's work, he constructs symplectic $4$--manifolds and shows the boundaries are diffeomorphic to the appropriate lens space with its universally tight contact structure. So there is no sense in which the diffeomorphisms are the identity on the boundary. 

In \cite{Fossati2019pre2} Fossati gave many restrictions on the topology of symplectic fillings of lens spaces. To state them we recall some notation. Given a continued fraction $[-a_1,\ldots, -a_n]$ for $-p/q<-1$ with $a_i\geq 2$, we know that the lens space $L(p,q)$ is described by surgery on the linear chain $C$ of unknots $U_1,\ldots, U_n$ shown in Figure~\ref{chain}, where the unknot $U_i$ has framing $-a_i$. We call $n$ the length of $p/q$, and denote it $l(p/q)$. We also note that the chain $C$ can be Legendrian realized so that Legendrian surgery yields a contact structure on $L(p,q)$ and all tight contact structures can be so realized. (See Section~\ref{koclassification}.) We denote the Stein filling of $(L(p,q),\xi)$ coming from this surgery description by $X_\xi$. 
\begin{figure}[htb]{\tiny
\begin{overpic}
{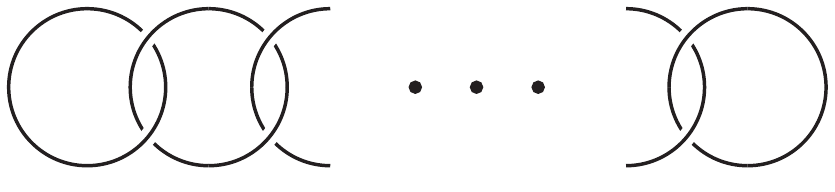}
\put(17, 0){$-a_1$}
\put(53, 0){$-a_2$}
\put(213,0){$-a_n$}
\end{overpic}}
\caption{A chain of framed unknots.}
\label{chain}
\end{figure}
Fossati showed that the Euler characteristic of any minimal symplectic filling of any tight contact structure on $L(p,q)$ is less than or equal to $l(p/q)+1$. He also showed that if a filling has $b_2=l(p/q)$ then it has the same intersection form as the filling in Figure~\ref{chain} and also must be simply-connected. Moreover, it was shown that no virtually overtwisted contact structure can be filled by a rational homology ball; subsequent independent proofs of this fact were given in \cite{EtnyreTosun20pre, GollaStarkston19pre}.  Fossati also bounded the Euler characteristic of a filling in terms of the order of its fundamental group. Finally, Fossati also showed that any symplectic filling of a virtually overtwisted structure on $L(p,q)$ is simply-connected for $p$ prime.

In this paper, we will classify all fillings of virtually overtwisted contact structures on lens spaces using only the work of Menke \cite{menke18pre} and Lisca \cite{Lisca08}. This together with Lisca's work, completes the classification of symplectic fillings of lens spaces. Our main result was also obtained by Christian and Li in \cite{ChristianLi20Pre} who, in addition, find many other nice applications of Menke's result, while we explore more fully the ramifications for fillings of lens spaces. 

We also note, while studying implications of the classification result, we will describe in Section~\ref{alt} an alternate construction of the Stein fillings from Lisca's work \cite{Lisca08} that might be of independent interest. 

\subsection{Consequences of the classification}\label{corollaries}
The statement of the classification result is a bit involved, so we begin by describing some corollaries of the main result. 

We denote the set of minimal fillings, up to diffeomorphism,  of a contact manifold $(M,\xi)$ by $\Fill(M,\xi)$. 
\begin{theorem}\label{maxcollection}
If $X$ is any minimal symplectic filling of a virtually overtwisted contact structure $\xi$ on $L(p,q)$ then there is some other symplectic structure on $X$ that fills the universally tight contact structure $\xi_{ut}$ on $L(p,q)$.  That is 
\[
\Fill(L(p,q),\xi)\subseteq \Fill(L(p,q), \xi_{ut}). 
\]
\end{theorem}
In fact, the subset of fillings of a virtually overtwisted contact structure can be an arbitrarily small (non-empty) subset. 
\begin{theorem}\label{cor2}
Given any integer $k$ there is a lens space $L(p,q)$ for which $\Fill(L(p,q), \xi_{ut})$ has more than $k$ elements but there is a virtually overtwisted contact structure $\xi$ such that $\Fill(L(p,q), \xi)$ has only one element.
\end{theorem}
On the other hand virtually overtwisted contact structures can have many fillings. 
\begin{theorem}\label{cor3}
Given any integer $k$ there is a lens space $L(p,q)$ and virtually overtwisted contact structure $\xi$ for which $\Fill(L(p,q), \xi)$ has more than $k$ elements.
\end{theorem}

We can also characterize the ``maximal" filling of a lens space.
\begin{theorem}\label{cor4}
The second Betti number of any symplectic filling $X$ of $(L(p,q), \xi)$ is less than or equal to the length of $p/q$
\[
b_2(X)\leq l(p/q),
\]
with equality if and only if $X$ is diffeomorphic to $X_\xi$, the filling coming from attaching Stein handles to a Legendrian realization of the chain in Figure~\ref{chain}.
\end{theorem}
\noindent
Fossati proved a similar result up to homeomorphism when $p$ was $2, 4,$ a power of an odd prime or $2$ times a power of an odd prime, \cite{Fossati2019pre2}.

In the other direction we note that fillings with the minimal possible second Betti number, that is rational homology balls, were completely determined by Golla and Starkston in \cite{GollaStarkston19pre}, with an independent proof later given by the first author and Tosun \cite{EtnyreTosun20pre}, and a proof with extra hypothesis was given by Fossati in \cite{Fossati2019pre2} (these extra hypothesis were not needed given Theorem~\ref{maxcollection}). 
As pointed out by the referee, these results can be generalize as follows by using our main result below. 
\begin{theorem}\label{chilower}
Let $\xi$ be the contact structure on $L(p,q)$ obtained by surgery on the Legendrian realization $\mathcal{C}=\{L_1,\ldots, L_m\}$ of the chain in Figure~\ref{chain}. Let $k$ be the number of the $L_i$ that have been stabilized both positive and negatively and let $l$ be the number of inconsistent sub-chains of $\mathcal{C}$ one obtains when these doubly stabilized Legendrian knots are removed. Then the Euler characteristic of any filling $X$ of $L(p,q)$ satisfies 
\[
\chi(X)\geq 1+k+ \lceil l/2\rceil.
\]
\end{theorem}
See the discussion just before Example~\ref{maximal} for the definition of an inconsistent sub-chain, but a key fact is that for an overtwisted contact structure we must have that $k+\lceil l/2\rceil$ is at least $1$. Thus we see that such a contact structure cannot bound a rational homology ball. 

The explicit rational homology ball fillings discussed above were not given in the papers mentioned above, so we note what they are here, and also show that a lens space can have at most one rational homology ball filling (this uniqueness was also proven in \cite{GollaStarkston19pre}). After releasing the first version of this paper, Marco Golla pointed out that the Stein diagrams for the rational homology ball were obtained earlier by Lekili and Maydanskiy in \cite{LekiliMaydanskiy14}. 
\begin{lemma}\label{ratballlemma}
A contact structure on a lens space has a rational homology ball symplectic filling if and only if the contact structure is universally tight and the lens space is diffeomorphic to $L(m^2, mh-1)$, with $h$ relatively prime to $m$ (and which can be taken to be between $1$ and $m-1$). Moreover the filling is unique and shown in Figure~\ref{ratballs}. 
\begin{figure}[htb]{\tiny
\begin{overpic}
{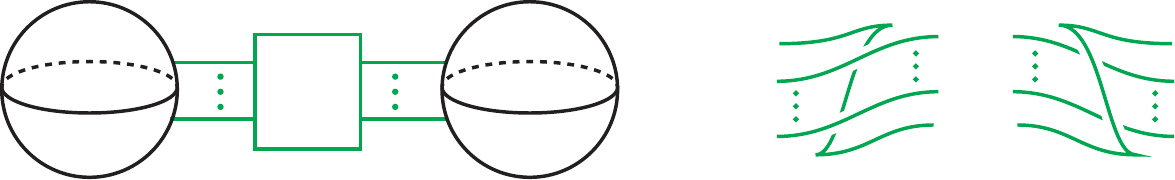}
\end{overpic}}
\caption{On the left are the rational homology balls that symplectically fill lens spaces. There are $m$ strands and the box contains $h$ copies of one of the Legendrian braids on the right. (You cannot use both types of braids.)}
\label{ratballs}
\end{figure}
\end{lemma}

We would like to understand which lens space can be obtained by surgery on a Legendrian knot in $(S^3,\xi_{std})$. To this end we first observe the structure of lens spaces with $b_2=1$ fillings.
\begin{lemma}\label{b21}
If $X$ fills a lens space and has second Betti number $1$ then it is obtained by attaching two $2$--handles to $S^1\times D^3$ along two torus knots in $S^1\times S^2$. If $r/s$ and $a/b$ are the slopes of the torus knots then $X$ is simply connected if and only if $s$ and $b$ are relatively prime. The slope convention is that an $h/m$ torus knot is order $m$ in the first homology of $S^1\times S^2$ and is shown in Figure~\ref{ratballs}. 
\end{lemma}
This result was also obtained by Simone in \cite[Theorem~1.4]{Simone16pre}, though the statement is given in different terms. 
\begin{remark}
It is interesting to note that while there is a unique filling of a lens space by a rational homology ball, if it exists, there can be more than one filling of a lens space with second Betti number $1$. For example, the universally tight contact structure on the lens space $L(36,13)$ can be obtained by surgery on a Legendrian realization of the torus knot $T_{-7,5}$ and also has a filling by a manifold with fundamental group $\Z/2\Z$ and second Betti number $1$. These fillings are obtained by applying Lisca's algorithm for constructing fillings given in Section~\ref{lclass} to the null sequences $(2,3,1,2,3)$ and $(2,2,2,1,4)$. (We note the continued fraction of $36/(36-13)$ is  $[2,3,2,2,4]$.)
\end{remark}

Using this we can begin to understand which lens space can be obtained by Legendrian surgery on a knot in the standard contact $S^3$.
\begin{theorem}\label{berge}
One can obtain a contact structure on the lens space $L(nm+1,m^2)$ from Legendrian surgery on a Legendrian realization of the $(n,-m)$--torus knot with Thurston-Bennequin invariant $-nm$; here $n$ and $m$ are relatively prime positive integers. One may also obtain a contact structure on the lens space $L(3n^2+3n+1, 3n+1)$ from Legendrian surgery on a Legendrian realization of the knot shown in Figure~\ref{bknot} with Thurston-Bennequin invariant $-3n^2-3n$.
\begin{figure}[htb]{\tiny
\begin{overpic}
{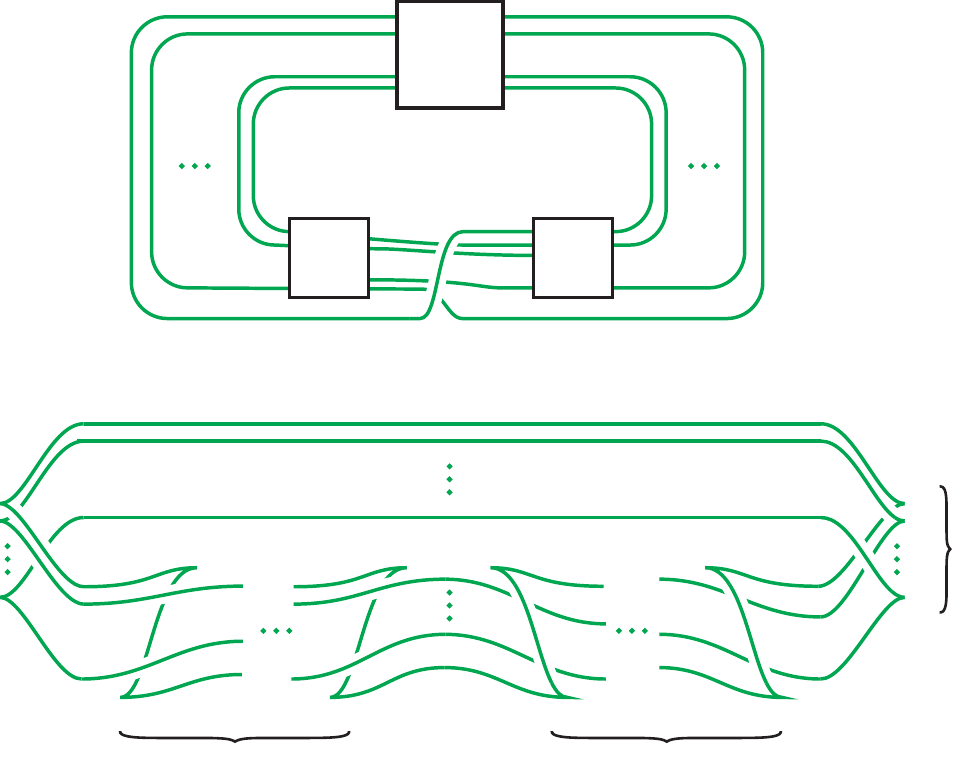}
\put(123, 202){$-1$}
\put(88, 143){$-1$}
\put(160, 143){$-1$}
\put(277, 60){$n+1$}
\put(65, -2){$n$}
\put(184, -2){$n-1$}
\end{overpic}}
\caption{The top diagram is a knot on which $-3n^2-3n-1$ surgery yields a lens space. There are $n+1$ strands going into the top box and the boxes indicate one full left handed twist. The bottom diagram is a Legendrian realization of this knot (the reflection about a veritical line gives another Legendrian realization).}
\label{bknot}
\end{figure}
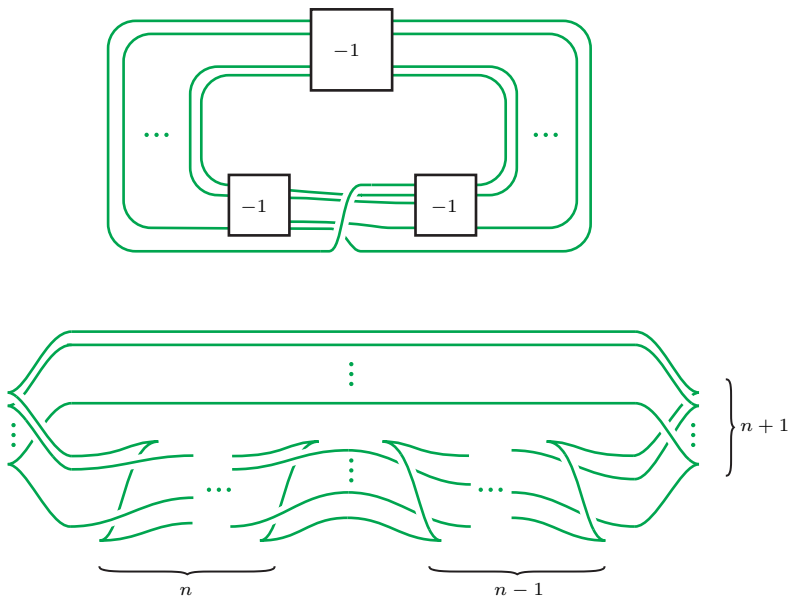
\end{theorem}
\begin{conjecture}\label{bergec}
A contact structure on a lens space is obtained from Legendrian surgery on a knot in the tight contact structure on $S^3$ if and only if it is one of the ones listed in Theorem~\ref{berge} or any tight contact structure on $L(p,1)$ (these come from Legendrian surgery on Legendrian unknots).
\end{conjecture}
See Remark~\ref{onlyex} after the proof of Theorem~\ref{berge} for evidence for this conjecture. 
\begin{remark}
Since the maximal Thurston-Bennequin invariant representatives of the $(n,-m)$-torus knot are classified \cite{EtnyreHonda01b}, we know all the contact structures on $L(nm+1,m^2)$ that come from surgery on such knots. There is no such classification of the knots in Figure~\ref{bknot}, however we show in the proof of Theorem~\ref{berge} that only the universally tight contact structures on $L(3n^2+3n+1, 3n+1)$ come from such surgeries. 
\end{remark}
\begin{remark}
In the proof of Theorem~\ref{berge} it is suggested that the only fillings of tight contact structures coming from surgery on a Legendrian knot in $(S^3,\xi_{std})$ are the ones in the theorem above. If this were true then the following conjecture would be true.
\end{remark}
\begin{conjecture}
 If $L$ is a Legendrian knot in $(S^3,\xi_{std})$ on which Legendrian surgery yields a lens space, then the $4$--manifold obtained from $B^4$ by attaching a $2$--handle along $L$ with framing one less than the Thurston-Bennequin invariant of $L$ is diffeomorphic to the one obtained by attaching a $2$--handle to one of the knots listed in Theorem~\ref{berge}, or to an unknot with framing less than $-1$.
\end{conjecture}
\begin{remark}
If this were true then this is a step towards proving the ``contact Berge conjecture". The contact analog of the Berge conjecture is that Legendrian surgery on a knot in the standard tight contact structure on $S^3$ would yield a lens space if and only if it were a Legendrian realization of a Berge knot with appropriate Thurston-Bennequin invariant. From the above conjecture, the only Berge knots which have Legendrian realizations on which Legendrian surgery yields a lens space are the negative torus knots and the knots in Figure~\ref{bknot} (and the unknot). The above conjecture says that any Legendrian knot on which Legendrian surgery yields a lens space must have the same $4$-dimensional trace {{(up to orientation preserving diffeomorphism)}} as the surgeries on the knots in Theorem~\ref{berge} (including the unknot, which can be thought of as a negative torus knot).  
\end{remark}
\begin{remark}
It is interesting to note that Geiges and Onaran \cite{GeigesOnaran18} showed that any tight contact structure on $L(ns^2-s+1, s^2)$ with $n \geq 2, s \geq 1$ is obtained by Legendrian surgery on some Legendrian torus knot in some (possibly overtwisted) contact structure on $S^3$. This indicates another interesting version of the contact Berge conjecture: what Legendrian knots in $S^3$ with some (possibly overtwisted) contact structure give tight contact structures on lens space via Legendrian surgery?
\end{remark}

We are also able to put some restrictions on the fundamental group of the Stein fillings of lens spaces.
\begin{theorem}\label{pi1}
If $X$ is a filling of $L(p,q)$, then $\pi_{1}(X)$ is a proper quotient of $\mathbb{Z}_{p}$. In particular, if $p$ is prime, $X$ must be simply connected.
\end{theorem}
\noindent
The second part of this theorem was also proven in \cite{Fossati2019pre2} and after the current paper was released, Aceto, McCoy, and Park \cite{AcetoMcCoyPark2020pre} proved a stronger version of this theorem and related the fundamental group of the filling to its second Betti number. We can also observe the following about fillings of specific lens spaces. 
\begin{theorem}\label{morepi1}
Let $\mathcal{C}$ be a Legendrian realization of the chain in Figure~\ref{chain} determining the lens space $L(p,q)$ and $\xi_{\mathcal{C}}$ be the contact structure determined by Legendrian surgery on $\mathcal{C}$. If the first or last unknot in the chain is stabilized both positively and negatively, then the pull back $\xi_{\mathcal{C}}$ to any cover of $L(p,q)$ is overtwisted and any filling of $(L(p,q), \xi_{\mathcal{C}})$ is simply connected. 
\end{theorem} 
See Remark~\ref{middlestab} for further discussions on the hypothesis of this theorem and overtwisted covers.

We now recover Fossati's classification of surgery on the Hopf link \cite{Fossati19pre}, though it is stated in different language and also includes the universally tight case. 
\begin{theorem}\label{cor5}
If $L(p,q)$ is obtained by (integral, negative) surgery on a one or two component chain then all contact structures on $L(p,q)$ bound a Stein structure on the plumbing defined by the surgery. This is the only Stein filling of a contact structure $\xi$ except for the following cases:
\begin{enumerate}
\item $\xi=\xi_{ut}$ on $L(4,1)$, which is surgery on an unknot with framing $-4$, also bounds the rational homology ball shown in Figure~\ref{excpetionalfills}, that has fundamental group $\Z_2$.
\item $\xi=\xi_{ut}$ on $L(8,3)$, which is surgery on a Hopf link with framings $-3$ and $-3$, also bounds a manifold with $b_2=1$ and fundamental group $\Z_2$ shown in Figure~\ref{excpetionalfills}.
\item $\xi=\xi_{ut}$ on $L(9,2)\equiv L(9,5)$, which is surgery on a Hopf link with framings $-5$ and $-2$, also bounds the rational homology ball shown in Figure~\ref{excpetionalfills}, that has fundamental group $\Z_3$.
\item $L(4n-1, 4)\equiv L(4n-1, n)$, which is surgery on a Hopf link with framings $-4$ and $-n$, also bounds Legendrian surgery on a maximal Thurston-Bennequin invariant Legendrian realization of the $(-2n+1,2)$-torus knot shown if Figure~\ref{excpetionalfills} if the Euler class of $\xi$ is $2n-3, 2n-5, \ldots, -2n+3$. The ones with Euler class $\pm(2n-3)$ are universally tight and the others are virtually overtwisted. (When $n>2$, there are other virtually overtwisted structures that only have the plumbing filling.) In particular, this shows that if Legendrian surgery on a knot $K$ yields $L(4n-1,4)$, the trace of the surgery must agree with the trace of surgery on a Legendrian realisation of the $(-2n+1,2)$-torus knot.
\end{enumerate}
\end{theorem}

\begin{figure}[htb]{\tiny
\begin{overpic}
{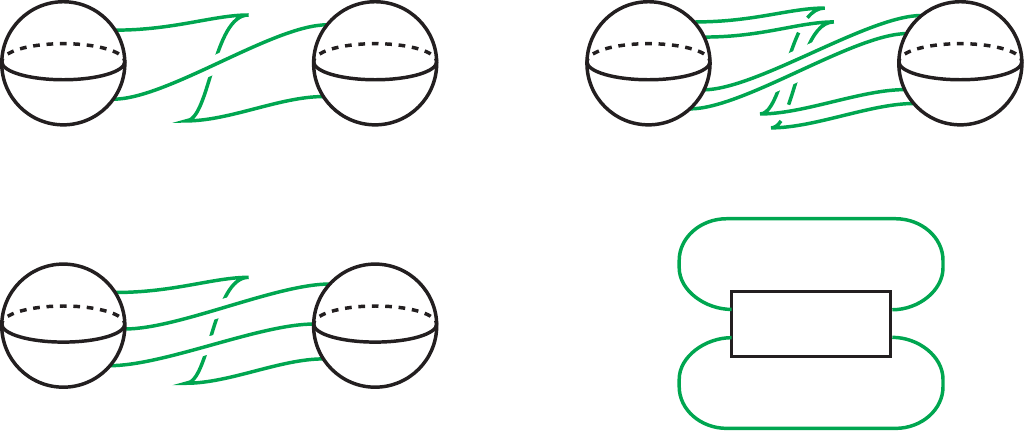}
\put(275, 10){$-4n+1$}
\put(214,27){\normalsize$-2n+1$}
\end{overpic}}
\caption{In the upper row, left to right, we have a Stein filling of $L(4,1)$ and $L(8,3)$. In the bottom row a Stein filling of $L(9,2)\equiv L(9,5)$ and a surgery diagram for the fillings of $L(4n-1,4)\equiv L(4n-1, n)$ (the number in the box represents the number of half twists). The Stein handlebody diagrams fill one of the universally tight contact structures on the stated lens space the other is filled by the diagram obtained by reflection about a vertical line.}
\label{excpetionalfills}
\end{figure}
\noindent

Using our main theorem, Theorem~\ref{list}, one may easily check that an explicit enumeration of the fillings of universally tight contact structures on lens spaces obtained by surgery on a chain of unknots with $n$ or fewer components, easily leads to an enumeration of fillings of virtually overtwisted contact structures on lens spaces which are obtained by surgery on a chain of unknots with $n+1$ or fewer components. We illustrate an example of this strategy by classifying the Stein fillings of contact structures on lens spaces obtained from a three component link.

\begin{theorem}\label{3component}
If $L(p,q)$ is obtained by (integral, negative) surgery on a three component chain of unknots then all contact structures on $L(p,q)$ bound a Stein structure on the plumbing defined by the surgery. On every lens space, there exists at least one contact structure that admits only the plumbing filling. This is the only Stein filling of a contact structure $\xi$ except for the following cases:
\begin{enumerate}
    \item The universally tight structures on 
    \begin{align*}
   L&(10, 7), L(12,7), L(13,5), L(13,9), L(16, 11), L(17,7),\\ 
   & L(21,8), 
   L(22,9), L(30,11), L(33,9),  \text{ and }  L(37,14),
    \end{align*}
all admit 2 fillings. There are virtually overtwisted contact structures on 
\[
 L(17,7), L(22,9), L(30,11), L(33,9),  \text{ and }  L(37,14),
\]
that admit 2 fillings and others that admit just 1.
\item The universally tight structures on 
  \begin{align*}
L(18,&5), L(19,7), L(24,7), L(25,14), L(26,15), L(29,11),  \\ & L(31,9), L(24,19), L(41, 15),\text{ and }   L(72, 19)
\end{align*}
 all admit 3 fillings. There exist virtually overtwisted structures on 
   \begin{align*}
L(17,&7), 
L(18,5), 
L(19,7), 
L(21,8), 
L(22,9), 
L(25, 14), 
L(26, 15), 
L(29,11),\\
&L(30,11),  
L(31,9), 
L(33,19), 
L(34,19), 
L(40,11),  \text{ and }
L(41, 15),
    \end{align*}
    that admit 2 fillings while others admit only 1. There exist virtually overtwisted structures on 
    \[
    L(24, 7), L(29,11), L(31, 9), L(34, 19),  L(40,11), \text{ and } L(72, 19)
    \]
    that admit 3 fillings while others admit 2 and others that admit only 1. 
 \item    
    The universally tight contact structure on 
    \[
    L(56,15)
    \]
    admits 4 fillings. There are virtually overtwisted contact structures on $L(56,15)$ that admit 4 fillings, and others that admit 3 fillings, 2 fillings, and just 1 filing. 
\item The universally tight contact structures on 
   \begin{align*}
L(4(l+3)(m+2)-m-6, (l+3)(m+2)-1)& \text{ with } l\geq 3 \text{ or } l=0, m\geq 3, \\
L(4(k+2)(m+2)-k-m-4, 4(m+2)-1)& \text{ with } m+k\geq 4 \text{ and } k,m\not=2, \\
L(9(m+3)-2, 5(m+2)-1)& \text{ with } m\geq 3, \\
L(19(m+2)-4, 5(m+2)-1)& \text{ with } m\geq 1, \\
L(7l+10, 2l+3) & \text{ with } l\geq 4, \text{ and }\\
L(9l+13, 2l+3), &\text{ with } l\geq 3
    \end{align*}
each admit 2 fillings. There are virtually overtwisted contact structures on these lens space that admit 2 fillings and others that admit only 1.     
\item The universally tight contact structures on the lens spaces 
   \begin{align*}
L(16(l+3)-8, 4(l+3)-1), &\text{ with } l\geq 3, \text{ and }\\
L(15(m+2)-4, 4(m+2)-1), &\text{ with } m\geq 3
    \end{align*}
each admit 3 fillings. There are virtually overtwisted contact structures on these lens spaces that admit 3 fillings and others that admit 2 or just 1 filling. 
\end{enumerate}
\end{theorem}
We discuss further corollaries concerning symplectic cobordisms after discussing the classification of fillings of contact structures on lens spaces. 

\subsection{The classification of symplectic fillings}\label{classsect}
Given a linear chain of framed unknots $C$ as in Figure~\ref{chain}, let $L(C)$ be the lens space obtained from surgery on this chain. Notice that if $C$ is the empty chain, then $L(C)$ is $S^3$. Denote the knots in the chain, from left to right, as $U_{1}, U_{2}, \cdots, U_{n}$.

If we remove $U_k$ from the chain $C$ we will get two other chains $C_s$ and $C_e$ (one might be empty if the removed unknot was on the end of the chain). Notice that in $L(C_s)\#L(C_e)$ the knot $U_k$ is simply the connected sum of cores of Heegaard tori in each of the lens spaces (we call these rational unknots). Moreover, we can recover $L(C)$ from $L(C_s)$ and $L(C_e)$ by taking their connect sum and doing $-a_k$ surgery on $U_k$ (the zero framing on $U_k$ comes from the Seifert disk in $S^3$). In addition, if $X$ and $X'$ are $4$--manifolds with boundary $L(C_s)$ and $L(C_e)$, respectively, then if we attach a $1$--handle to connect $X\sqcup X'$ and a $2$--handle to $U_k$, we get a $4$--manifold bounded by $L(C)$. So we see how to build fillings of $L(C)$ from fillings of $L(C_s)$ and $L(C_e)$. We can also describe this filling of $L(C)$ in terms of a round $1$--handle attachment to $X\sqcup X'$ along rational unknots in $L(C_s)$ and $L(C_e)$, see Section~\ref{menkedecompose} for details. 

\begin{construction}\label{constructionofG}
Now, all the tight contact structures $\xi$ on $L(p,q)$ are obtained by Legendrian surgery on Legendrian realizations $L_i$ of the $U_i$ with Thurston-Bennequin invariant $-a_i+1$. Denote the chain of Legendrian knots by $\mathcal{C}$. We can carry out the discussion above for contact structures on lens spaces. Specifically if we remove $L_k$ from $\mathcal{C}$ then we get two chains $\mathcal{C}_s$ and $\mathcal{C}_e$. If we denote the contact structure on the lens space coming from Legendrian surgery on $\mathcal{C}$ by $\xi_{\mathcal{C}}$ then we see there is a Legendrian knot $L_k$ in $L(\mathcal{C}_s)\#L(\mathcal{C}_e)$ on which Legendrian surgery yields $L(\mathcal{C})$. 
This knot is uniquely determined up to a contactomorphism smoothly isotopic to the identity by Lemma~\ref{weaklysimple} since the knot type is weakly Legendrian simple and its Thurston-Bennequin invariant and rotation number are determined by the fact that Legendrian surgery yields $(L(p,q), \xi_\mathcal{C})$. 

If $(X,\omega)$ is a symplectic filling of $L(\mathcal{C}_s)$ and $(X',\omega')$ is a symplectic filling of $L(\mathcal{C}_e)$ then we can build a symplectic filling of $L(\mathcal{C})$ by attaching a Weinstein $1$--handle to $X\sqcup X'$ and a Weinstein $2$--handle to $L_k$. 

If we denote by $\Fill(L(p,q))$ the set of minimal symplectic fillings of $L(p,q)$ up to diffeomorphism (from Wendl's work \cite{Wendl10} we know that this is also the set of Stein fillings of $L(p,q)$), then the above construction gives a map
\[
G_{\{L_k\}}:\Fill(L(\mathcal{C}_s))\times \Fill(L(\mathcal{C}_e)) \to \Fill(L(\mathcal{C}))
\] 
to the fillings of $\mathcal{C}$.

Similarly if $S=\{L_{i_1}, \ldots, L_{i_k}\}$ is a collection of components of $\mathcal{C}$ then $\mathcal{C}-S$ will be a collection of chains $\mathcal{C}_1,\ldots, \mathcal{C}_{k+1}$ (where some of the $\mathcal{C}_i$ could be empty if some of the $L_{i_j}$ are adjacent, or occur at an end of $\mathcal{C}$). There is clearly a similar map to the one constructed above
\[
G_S:\prod_{i=1}^{k+1} \Fill(\mathcal{C}_i) \to \Fill(\mathcal{C}).
\]
from the minimal symplectic fillings of the sub-chains to the fillings of $L(\mathcal{C})$.
\hfill \qed
\end{construction}

\noindent 
We now need some terminology concerning sub-chains of a Legendrian chain $\mathcal{C}$. We call a chain $\mathcal{C}$ \dfn{consistently stabilized} if all of the knots in $\mathcal{C}$ have only been stabilized positively (or not at all) or have all been stabilized negatively (or not at all). We recall, see Theorem~\ref{classifylpq} and Remark~\ref{classifylpq2} below, that the contact structure $\xi_{\mathcal{C}}$ is universally tight if and only if $\mathcal{C}$ is consistently stabilized. 

We call a chain $\mathcal{C}$ \dfn{nicely stabilized} if each component in the $\mathcal{C}$ has only been stabilized either positively or negatively (or not at all), but different components could have been stabilized differently. See top of Figure~\ref{nicechains}.
\begin{figure}[htb]{\tiny
\begin{overpic}
{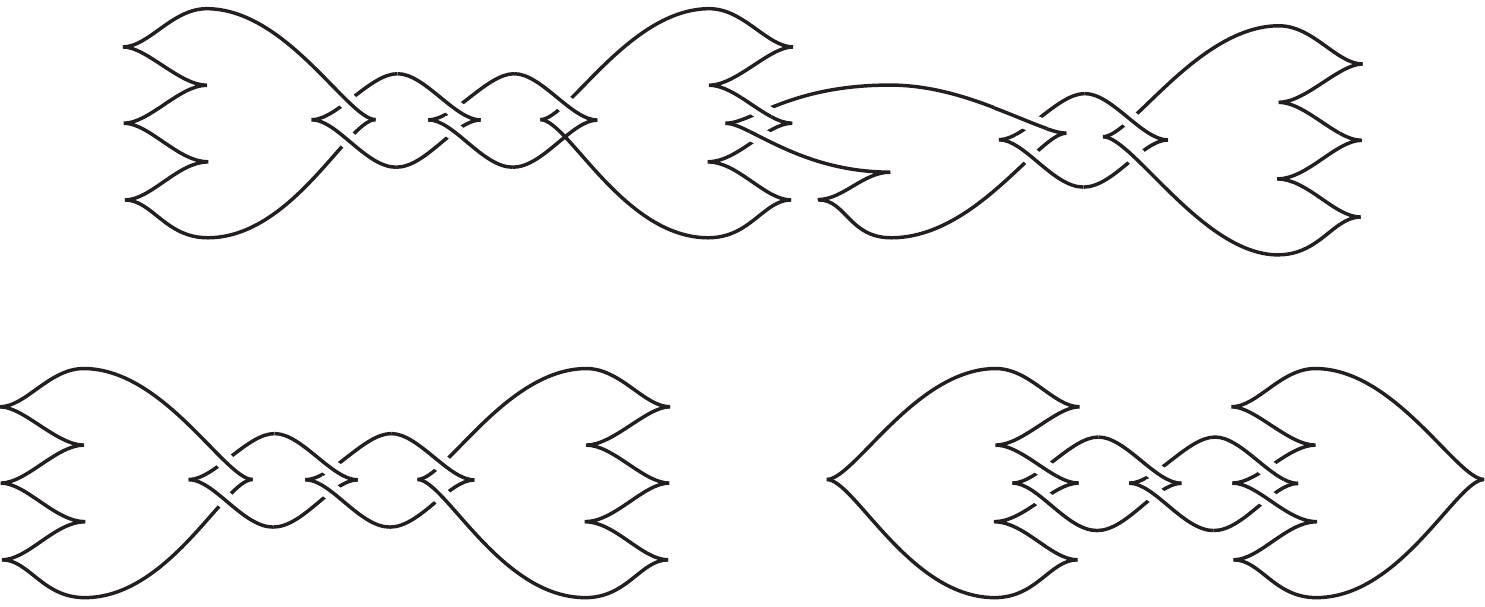}
\put(58, 96){$L_1$}
\put(107, 117){$L_2$}
\put(145, 117){$L_3$}
\put(199, 96){$L_4$}
\put(258, 96){$L_5$}
\put(312, 113){$L_6$}
\put(363, 93){$L_7$}
\end{overpic}}
\caption{A nicely stabilized chain at the top that contains three (overlapping) inconsistent sub-chains. There are two types of inconsistent chains, which are shown at the bottom.} 
\label{nicechains}
\end{figure}
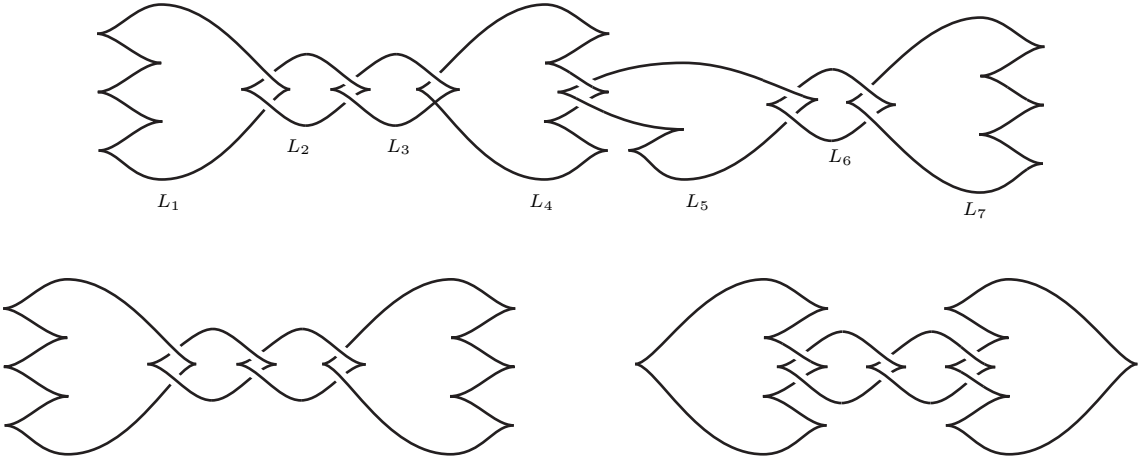

If $\mathcal{C}$ is a nicely stabilized chain then an \dfn{inconsistent sub-chain of $\mathcal{C}$} is a sub-chain $\mathcal{C}_i$ such that the first element and last element in the sub-chain are both stabilized and stabilized in opposite ways, and all the other elements of $\mathcal{C}_i$ are not stabilized at all. See bottom of Figure~\ref{nicechains}. Notice that inconsistent sub-chains can overlap at their end points.

Now given a chain $\mathcal{C}$ of Legendrian unknots, let $\mathcal{D}$ be the collection of knots in $\mathcal{C}$ that have been stabilized both positively and negatively. So each chain in $\mathcal{C}-\mathcal{D}$ is a nicely stabilized chain. Let $\mathcal{S}$ be the set of all unknots in $\mathcal{C}-\mathcal{D}$ that are in inconsistent sub-chains. We call a collection $\mathcal{M}$ of knots in $\mathcal{S}$ \dfn{maximal} if its intersection with each inconsistent sub-chain in $\mathcal{C}-\mathcal{D}$ contains exactly one element.  Notice that if all the inconsistent sub-chains are disjoint, then $\mathcal{M}$ contains as many elements as there are  inconsistent sub-chains, but if the ends of inconsistent sub-chains overlap, as in Figure~\ref{nicechains}, there might be fewer elements in  $\mathcal{M}$. 
\begin{example}\label{maximal}
At the top of Figure~\ref{nicechains} we see a nicely stabilized chain such that $\mathcal{S}$ consists of the chains $\{L_1,L_2,L_3,L_4\}$, $\{L_4, L_5\}$, and $\{L_5, L_6, L_7\}$. Any maximal collection of links in $\mathcal{S}$ must contain $L_4$ or $L_5$ (and notice it cannot contain both) and thus will have two elements while there are three inconsistent sub-chains. 
\end{example}
A key observation is that for any maximal collection $\mathcal{M}$ in $\mathcal{S}$ each chain in $\mathcal{C}-(\mathcal{D}\cup \mathcal{M})$ is a consistently stabilized chain and so the corresponding contact structure is universally tight. 

We can now describe all fillings of contact structures on lens spaces. 
\begin{theorem}\label{list}
Given any chain $\mathcal{C}$ of Legendrian unknots, let $\mathcal{D}$ be the subset of components that have been stabilized both positively and negatively and $\mathcal{S}$ be the set of  unknots in $\mathcal{C}-\mathcal{D}$ that are in inconsistent sub-chains, as discussed above. Then then the fillings of $(L(\mathcal{C}),\xi_{\mathcal{C}})$ are 
\[
\bigcup_{\mathcal{M} \text{ maximal collection in }\mathcal{S}} \text{Image} (G_{\mathcal{D}\cup \mathcal{M}}).
\]
We notice that each chain in $\mathcal{C}-(\mathcal{D}\cup \mathcal{M})$ is consistently stabilized, so corresponds to a universally tight contact structure and hence all of its fillings are determined by Lisca's theorem, Theorem~\ref{lisca}.
\end{theorem}
\begin{remark}
This result was also obtained by Christian and Li in \cite{ChristianLi20Pre}. 
\end{remark}

This theorem is proven in Section~\ref{provelist} by systematically applying a theorem of Menke, Theorem~\ref{menke}, to reduce the virtually overtwisted case to the universally tight case that is handled by Lisca's work, Theorem~\ref{lisca}. Two technical parts that may be of independent interest are an understanding of continued fractions, paths in the Farey graph, and breaking up continued fractions discussed in Sections~\ref{breakingup} and~\ref{pathinfg}; and building fillings of a lens space described by surgery on a chain of unknots from fillings of lens spaces obtained by surgery on sub-chains described in Section~\ref{combiningchains}. Along the way we will also need the Legendrian simplicity of certain knots in the connect sum of lens spaces established in Section~\ref{legsimpofratunknots}.

When trying to list the contact structures on lens spaces it is useful to understand the maps $G_S$. 
\begin{theorem}\label{inject}
Give a chain $\mathcal{C}=\{L_1,\ldots, L_n\}$ of Legendrian unknots and some $1\leq k\leq n$, let $\mathcal{C}_s$ and $\mathcal{C}_e$ be the chains in $\mathcal{C}-\{L_k\}$. The map
\[
G_{\{L_k\}}: \Fill(\mathcal{C}_s) \times \Fill(\mathcal{C}_e)\to \Fill(\mathcal{C})
\]
from the minimal symplectic fillings of the sub-chains to the filling of $L(p,q) = L(\mathcal{C})$ is injective unless $q^2\equiv 1 \! \mod p$. When $q^2\equiv 1 \! \mod p$ the map may still be injective. If not there is at most a two-to-one ambiguity described in Proposition~\ref{combinestrings}. 
\end{theorem}
\begin{remark}
This map is not in general surjective. In fact, in Section~\ref{injectsect} we will see that in Theorem~\ref{list} one does need to consider all the images of $G_{\mathcal{D}\cup \mathcal{M}}$ for all maximal collections in $\mathcal{S}$, as the image for any $\mathcal{M}$ can miss elements in $\Fill(\mathcal{C})$. 
\end{remark}
\begin{remark}
By Theorem~\ref{maxcollection} we see that the minimal symplectic fillings of any contact structure on $L(p,q)$ are a subset of the fillings of the universally tight contact structure. Since Lisca's theorem, Theorem~\ref{lisca}, describes the latter set, Theorem~\ref{list} specifies the former. 
\end{remark}

\subsection{Symplectic cobordisms between lens spaces}\label{cobordisms}
When Stein cobordisms between lens spaces exist, their topology can be controlled using the following theorem. 

\begin{theorem}\label{length}
If there is a Stein cobordism from a tight contact structure $\xi$ on $L(p,q)$ to any contact structure $\xi'$ on $L(p',q')$ then we must have that $l(p'/q')\geq l(p/q)$, where $l(r/s)$ is the length of the continued fractions expansion of $r/s$ as discussed above. Moreover, if there is such a cobordism $X$ with $l(p'/q')= l(p/q)$, then it is an $h$-cobordism, $L(p,q)\cong L(p',q')$, and $\xi$ must be contactomorphic to $\xi'$.
\end{theorem}
\begin{remark}
We expect the cobordism when $l(p'/q')= l(p/q)$ will actually be an $s$-cobordism, but for that one would need to see that the Whitehead torsion of $X$ is trivial. This is automatic in some cases, such at when $p=2, 3, 4$ and $6$, but not in general. Less certain, but it is still reasonable to conjecture is that the cobordism is actually a product. 
\end{remark}
\begin{remark}
Notice that there is no restriction on symplectic cobordisms between tight contact structures on lens spaces, so there is a significant difference between them and Stein cobordisms. To see this recall that any tight contact structure is supported by a planar open book \cite{Schoenenberger05} and hence there is a symplectic cobordism from it to the tight contact structure on $S^3$ obtained by capping off all but one boundary component of the open book \cite{Eliashberg04,Gay02a}. We can now glue this cobordism to any filling of a lens space, minus a Darboux ball, to get a cobordisms from any tight contact structure on a lens space to any other tight contact structure on any other lens space. We thank the referee for pointing out this construction. 

We also note that if one considers overtwisted contact structures, then there are plenty of Stein cobordisms between them that do not satisfy the constraints in the above theorem, \cite{EtnyreHonda02a}. 
\end{remark}

In \cite{Plamenevskaya12}, Plamenevskaya showed that there is no tight contact structure on a lens space that can be obtained from itself by Legendrian surgery on a link.  The above theorem strengthens this result.

We now turn to constructing cobordisms. We will give four constructions of cobordisms, two fairly straightforward and two that are less obvious. 

\noindent{\bf Construction 1: Subchains.} 
Notice that if $\mathcal{C}'$ is a sub-chain of a chain of Legendrian unknots $\mathcal{C}$ then we can construct minimal symplectic cobordisms as follows. Let $\mathcal{C}_s$ and $\mathcal{C}_e$ be the two components of $\mathcal{C}-\mathcal{C}'$ where the first is the left most part of the chain (and one or both of these chains can be empty). Now let $L_s$ be the right most unknot in $\mathcal{C}_s$ and $L_e$ the left most unknot in $\mathcal{C}_e$. Finally let $\mathcal{C}_s'=\mathcal{C}_s-\{L_s\}$ and similarly for $\mathcal{C}_e'$. Then by taking a portion of the symplectization of the contact manifold $L(\mathcal{C}')$ and a filling $X'_s$ and $X'_e$ for $\mathcal{C}_s'$ and $\mathcal{C}_e'$ we can attach two Stein $1$--handles to get a connected manifold and then attach two $2$--handles as in the previous section to get a cobordism from $L(\mathcal{C}')$ to $L(\mathcal{C})$. That is, we have a gluing map, analogous to the one above
\[
G_{L_s,L_e}: \Fill(\mathcal{C}_s')\times \Fill(\mathcal{C}_e')\to \Cob(L(\mathcal{C}'), L(\mathcal{C}))
\]
where $\Cob(L(\mathcal{C}'), L(\mathcal{C}))$ is the set of Stein cobordisms, up to diffeomorphism, from the lens space $L(\mathcal{C}')$ to $L(\mathcal{C})$. 

One can use the above construction to create many non-trivial symplectic cobordisms. For example, one can construct cobordisms with arbitrarily large homology or with non-trivial fundamental group. 

\noindent {\bf Construction 2: Rolled up-diagrams.} We begin with an example. 
In Figure~\ref{nonstandard} we see a cobordism from $L(6,1)$ to $L(14,3)$, since the continued fraction of $-14/3=[-5,-3]$ and for $-6=[-6]$ we see that you can have cobordisms without one chain being a sub-chain of the other. To see the left hand diagram gives surgery on the 2-chain with framings $-5$ and $-3$, just slide the $-6$--framed handle over the $-5$--framed handle. 
\begin{figure}[htb]{\tiny
\begin{overpic}
{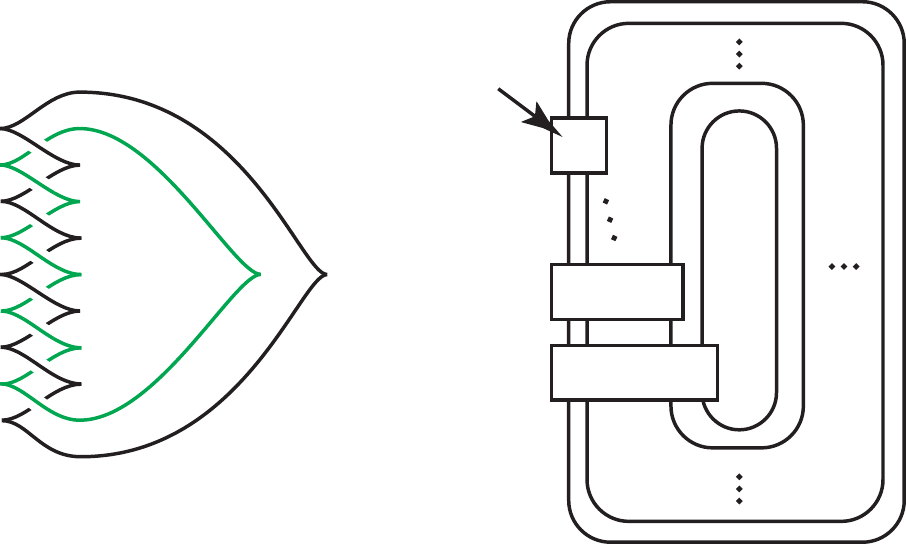}
\put(215, 100){$b_1$}
\put(235, 100){$b_2$}
\put(237, 120){$b_{n-1}$}
\put(265, 120){$b_n$}
\put(168, 48){$-a_1+1$}
\put(163, 71){$-a_2+2$}
\put(118, 134){$-a_{n-1}+2$}
\end{overpic}}
\caption{On the left is a cobordism of lens spaces. Legendrian surgery on the larger curve gives the lens space $L(6,1)$, then a Stein $2$--handle is attached to the other curve to give a cobordism to another lens space. On the right is a ``rolled-up diagram" for the chain $[-a_1,\ldots, -a_n]$. 
} 
\label{nonstandard}
\end{figure}
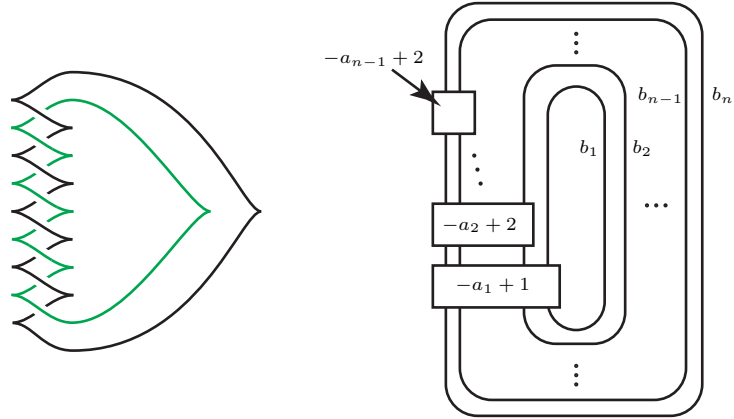
We can understand these cobordism in terms of \dfn{rolled-up diagrams} introduced by Sch\"onenberger in \cite{Schoenenberger05}. One starts with the chain in Figure~\ref{chain}. Now slide the $-a_2$--framed unknot over the $-a_1$--framed unknot, then the $-a_3$ framed unknot over what became of the $-a_2$--framed unknot, and continue until the $-a_n$--framed unknot is slid. This will give the right hand picture in Figure~\ref{nonstandard}, where the framings are
\[
b_i=2(i-1)+ \sum_{j=1}^i -a_j.
\]
You can see the left hand picture in the figure is a rolled up diagram for the $[-5,-3]$ chain. 

One may readily check that if one deletes some combination of the first unknot, the last unknot, or unknots with framing the same as their predecessor, from the rolled up diagram, then one gets another rolled up diagram for some lens space $L(p',q')$  and we can attach $2$--handles to the removed unknots to get a cobordism from $L(p',q')$  to the original one. (Notice that if you delete only the last unknot then this cobordism is the same cobordism you would get from the gluing map above.)

It is easy to figure out the possibilities for the  lens spaces $L(p', q')$. The continued fraction for $p'/q'$ is obtained by doing any combination of the following to $[-a_1,\ldots, -a_n]$:
\begin{enumerate}
\item Remove $-a_1$ and change $-a_2$ to $-a_1-a_2+2$,
\item Remove $-a_n$,
\item Remove a $-a_i$ that is equal to $-2$. 
\end{enumerate}
Also, by noting that one may roll-up the diagram starting with $-a_n$ instead of $-a_1$, then we can also reverse the roles of $-a_1$ and $-a_n$ in the above. 

\noindent{\bf Construction 3: Torus knots.} In \cite{ChakrabortyEtnyreMin20pre} Chakraborty, Min, and the first author gave a precise criterion for when a knot type that sits on a torus has a Legendrian representative for which the contact framing is larger than the torus framing. These are called ``Legendrian large torus knots". We will discuss this criterion in detail in Section~\ref{setion:cobordism} but here we note the consequences and draw surgery diagrams for these knots. These diagrams come from the upcoming paper \cite{BakerEtnyreMinOnaran21pre} by Baker, Min, Onaran, and the first author, where a classification of torus knots in lens space is given. 

\noindent
{\bf Torus framed surgeries.} It is easy to see that if the knot $K$ sits on a separating torus $T$ in $M$, then the result of doing surgery on $K$ with framing coming from $T$ is a reducible manifold $M_1\# M_2$ where the $M_i$ are the result of Dehn filling the components of $M\setminus T$ with slope $K$. If $\xi$ is a contact structure on a lens space $L(p,q)$ that is obtained by Legendrian surgery on a chain as in Figure~\ref{chain} where the $i^{th}$ component has been stabilized both positively and negatively, then we may draw a surgery diagram for $\xi$ as shown in Figure~\ref{reduciblesurgery}. 
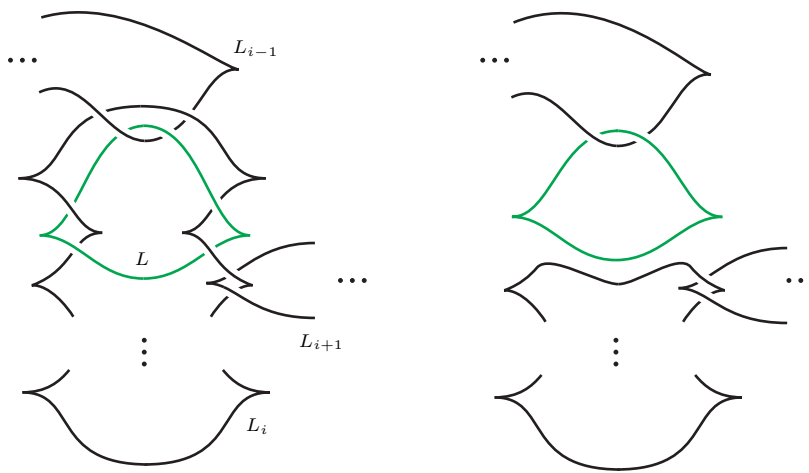
\begin{figure}[htb]{\tiny
\begin{overpic}
{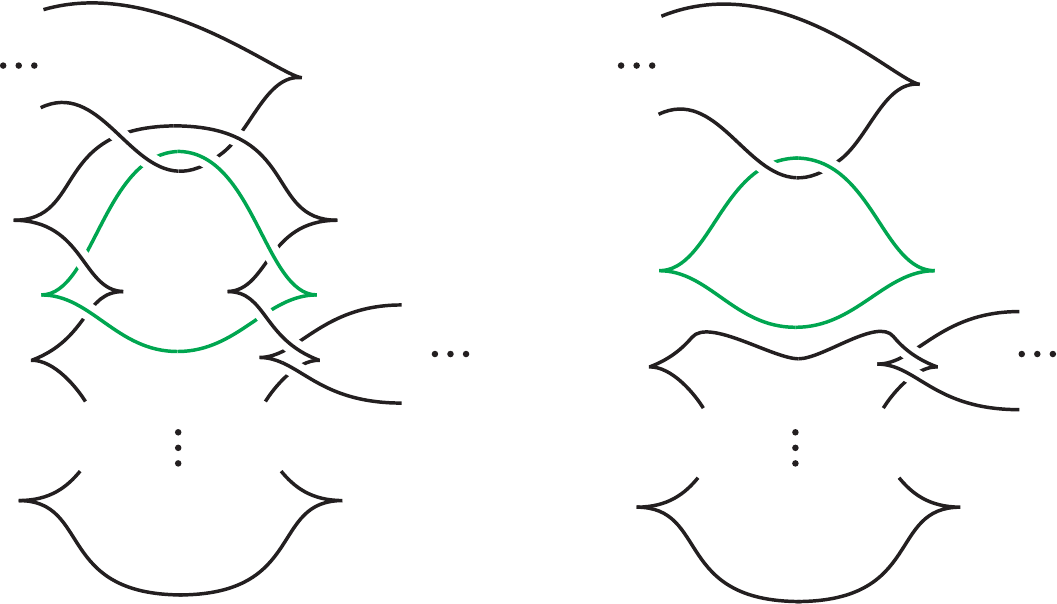}
\put(90, 15){$L_i$}
\put(110, 47){$L_{i+1}$}
\put(85, 158){$L_{i-1}$}
\put(48, 78){$L$}
\end{overpic}}
\caption{On the left is a portion of the chain of Legendrian unknots giving the contact structure $\xi$ together with the Legendrian knot $L$. On the right is the result of Legendrian surgery on $L$ together with a handle slide of $L_i$ over $L$. 
} 
\label{reduciblesurgery}
\end{figure}
The curve $L$ in the diagram can be seen to be a torus knot and the contact framing on $L$ is one larger than the torus framing. So Legendrian surgery on $L$ will result in a the reducible manifold shown at the right hand side of Figure~\ref{reduciblesurgery}. This may be checked by Legendrian handle slides, see \cite{DingGeiges09} and in particular \cite[Section~6.2]{Avdek13}. Just slide $L_i$ in the figure over $L$. Thus attaching a Stein handle to $L$ will give a cobordism from $(L(p,q), \xi)$ to a contact structure on the connected sum of two lens spaces. Attaching another $2$--handle will result in a cobordism from $(L(p,q),\xi)$ to a contact structure on another lens space. 
\begin{example}
Figure~\ref{redex} 
\begin{figure}[htb]{\tiny
\begin{overpic}
{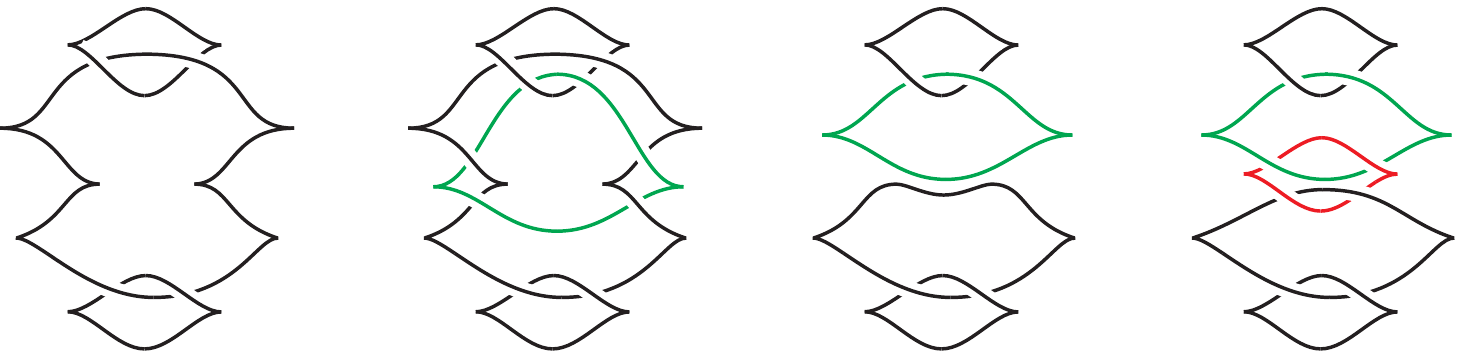}
\end{overpic}}
\caption{On the left is the virtually overtwisted contact structure on $L(12,7)$. In the next figure a Stein $2$--handle is attached to the next unknot. The next figure shows an equivalent surgery diagram. The last figure is the result of another $2$--handle attachment to complete the cobordism to $L(6,5)$.} 
\label{redex}
\end{figure}
gives a Stein cobordism from the virtually overtwisted contact structure on $L(12,7)$ (notice there is only one) to a universally tight contact structure on $L(6,5)$.
\end{example}

\noindent
{\bf Torus plus one framed surgeries.}
There is another type of unexpected cobordism coming from Legendrian large torus knots. Recall, see for example \cite{Rolfsen76}, that if $K$ sits on a torus $T$ in $M$, then the result of doing surgery on $K$ with framing one larger than the framing coming from $T$ is the same as cutting $M$ along $T$ and re-glueing the tori with a negative Dehn twist about $K$. If $\xi$ is a contact structure 
on a lens space $L(p,q)$ that is obtained by Legendrian surgery on a chain as in Figure~\ref{chain} where the $i^{th}$ component has been stabilized at least two times positively and two times negatively, then we may draw a surgery diagram for $\xi$ as shown in Figure~\ref{DTsurgery}. 
\begin{figure}[htb]{\tiny
\begin{overpic}
{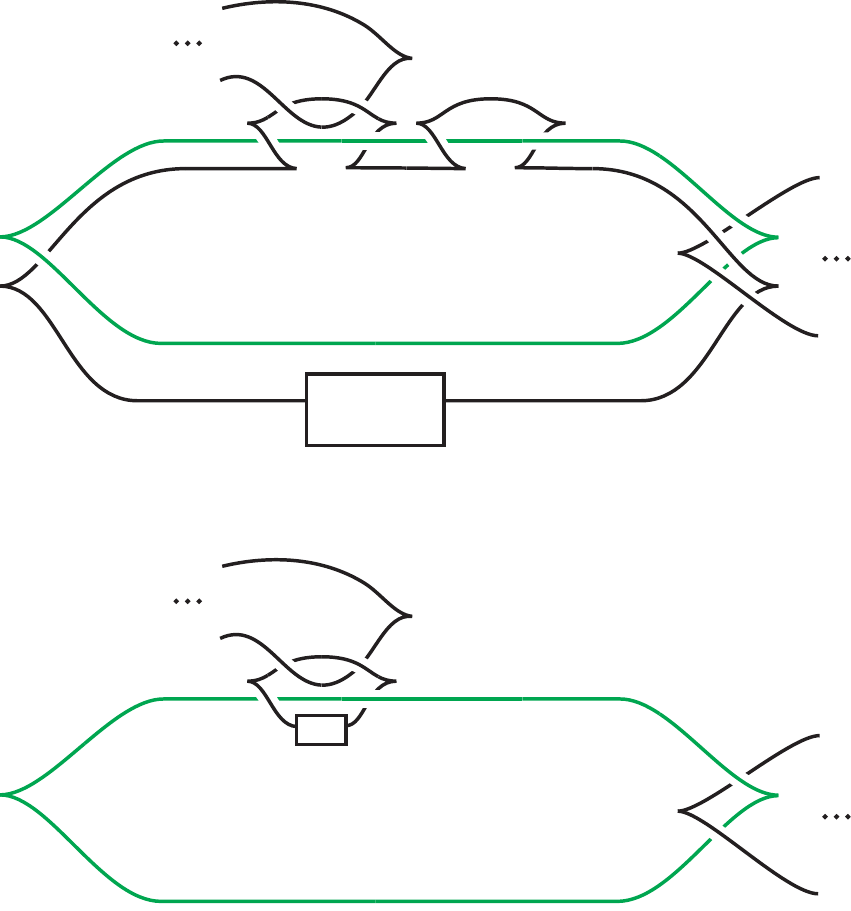}
\put(155, 139){$L_i$}
\put(225, 160){$L_{i+1}$}
\put(115, 255){$L_{i-1}$}
\put(195, 220){$L$}
\end{overpic}}
\caption{The upper diagram shows a portion of the chain of Legendrian unknots giving the contact structure $\xi$ together with the Legendrian knot $L$. The bottom diagram is the result of Legendrian surgery on $L$ together with a handle slide of $L_i$ over $L$. The box indicates the possibility for other stabilizations.
} 
\label{DTsurgery}
\end{figure}
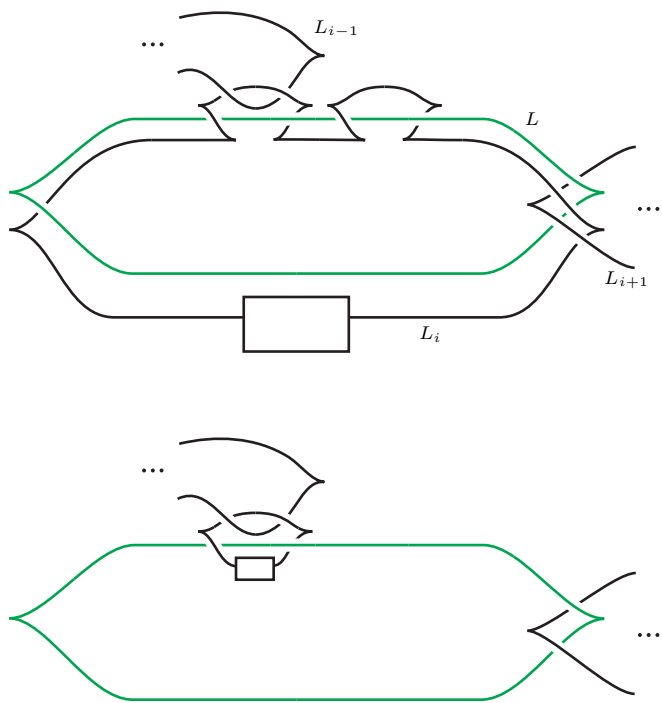
Legendrian surgery produces a new lens space as shown in the bottom of Figure~\ref{DTsurgery}. So attaching a Stein handle to $L$ will give a Stein cobordism from $(L(p,q), \xi)$ to another lens space. We discuss which lens space in Section~\ref{setion:cobordism}, but give an example here.

\begin{example}
There is a symplectic cobordism from the virtually overtwisted contact structure on $L(6,1)$ with Euler class $0$ to the universally tight contact structure on $L(3,2)$. This is easy to see as the contact structure on $L(6,1)$ is given by Legendrian surgery on a Legendrian unknot that is obtained from the maximal Thurston-Bennequin unknot by stabilizing twice positive and twice negatively. Adding a Stein $2$--handle as indicated in Figure~\ref{DTsurgery} will result in a chain of two maximal Thurston-Bennequin unknots, which describes $L(3,2)$. 
\end{example}

Using these cobordisms we can prove the following result.
\begin{theorem}\label{makenice}
There is a Stein cobordism from any tight contact structure on a lens space to a contact structure on a lens space obtained from surgery on a nicely stabilized chain (recall, this means that each unknot in the chain only has positive stabilizations or only negative stabilizations). 
\end{theorem}
Notice that the theorem can give cobordisms from virtually overtwisted contact structures to universally tight ones. The first two constructions cannot do this. Such a cobordism was first pointed out to the authors by Marc Kegel. 

\noindent
{\bf Torus minus one framed surgeries.} We also know, see for example \cite{Rolfsen76}, that if $K$ sits on a torus $T$ in $M$, then the result of doing surgery on $K$ with framing one less than the framing coming from $T$ is the same as cutting $M$ along $T$ and re-glueing the tori with a positive Dehn twist about $K$. We will see in Section~\ref{koclassification} that in any tight contact structure on $L(p,q)$ one may Legendrian realize an $(a,b)$-torus knot with framing agreeing with the torus framing if $-p/q<b/a<0$. Adding a Stein handle to this knot will give a Stein cobordism from $L(p,q)$ to $L(pb^2+p(1-ab), q(1+ab)-pa^2)$. We discuss this further and how to determine the contact structure on the upper boundary in Section~\ref{setion:cobordism}; but we note here, that unlike the other torus knots surgeries, these surgeries will always take virtually overtwisted contact structures to virtually overtwisted ones.

\noindent{\bf Construction 4: Sporadic surgery diagrams.} Marc Kegel showed us the two cobordisms in Figures~\ref{R1} and~\ref{R2}. \begin{figure}[htb]{\tiny
\begin{overpic}
{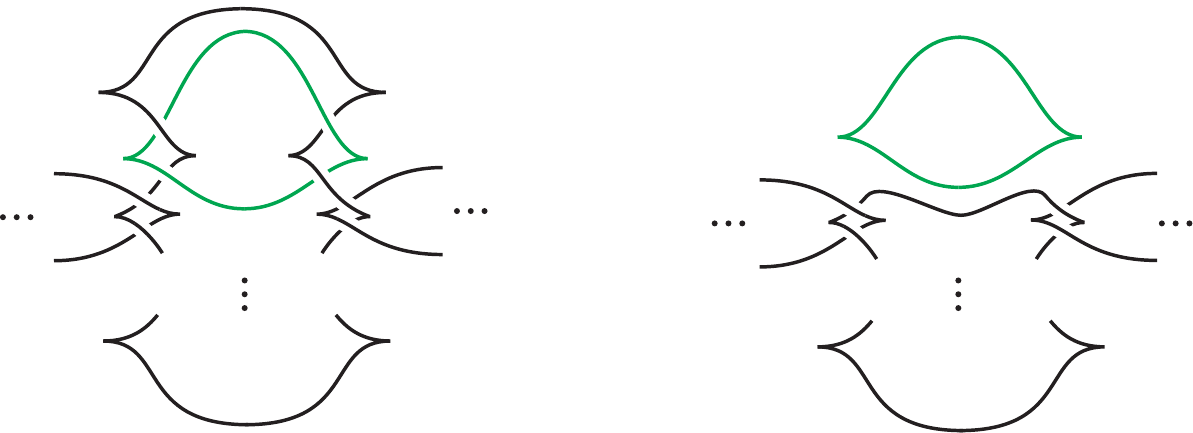}
\put(97, 114){$L_i$}
\put(119, 43){$L_{i+1}$}
\put(15, 42){$L_{i-1}$}
\put(67, 70){$L$}
\end{overpic}}
\caption{The left diagram shows a portion of the chain of Legendrian unknots giving a contact structure on a lens space together with a Legendrian knot $L$. The right diagram is the result of Legendrian surgery on $L$ together with a handle slide of $L_i$ over $L$. 
} 
\label{R1}
\end{figure}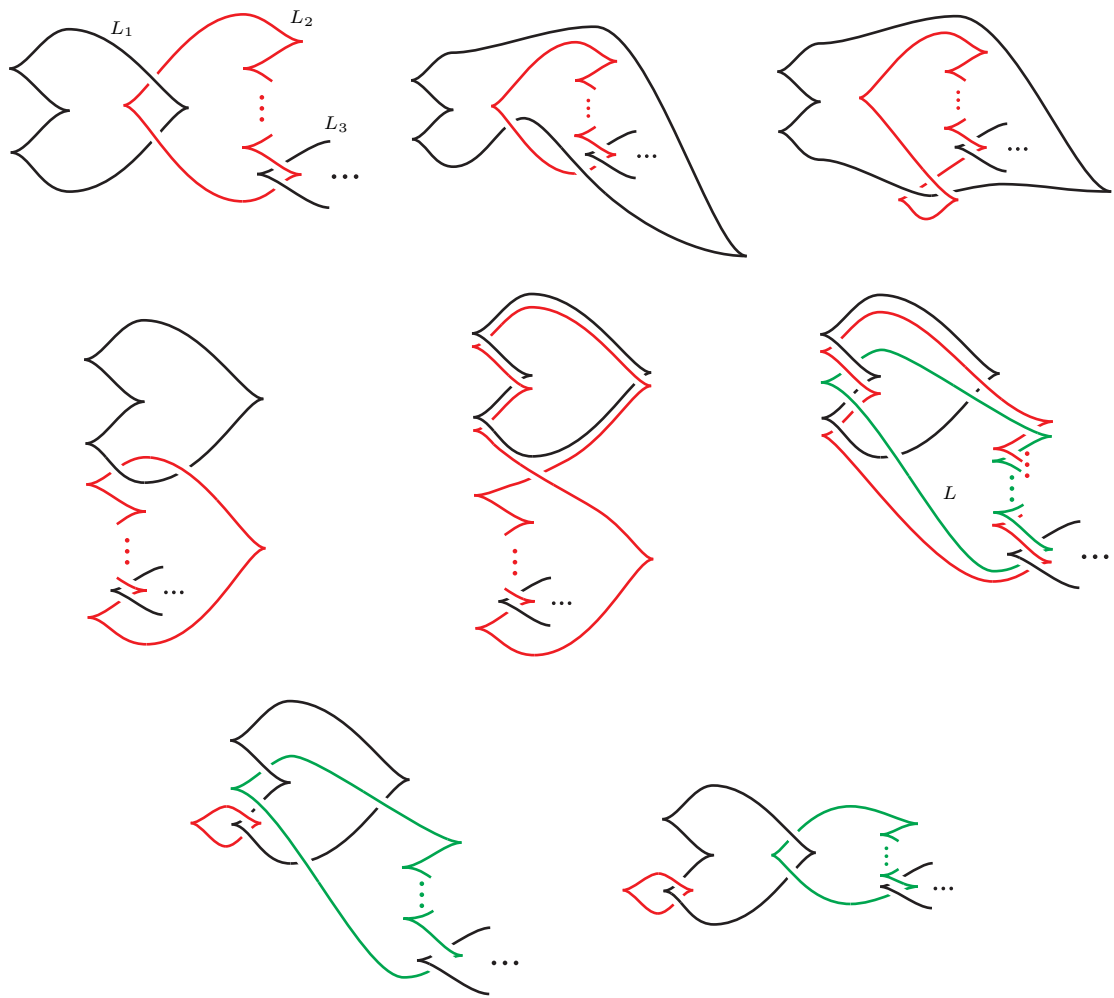
\begin{figure}[htb]{\tiny
\begin{overpic}
{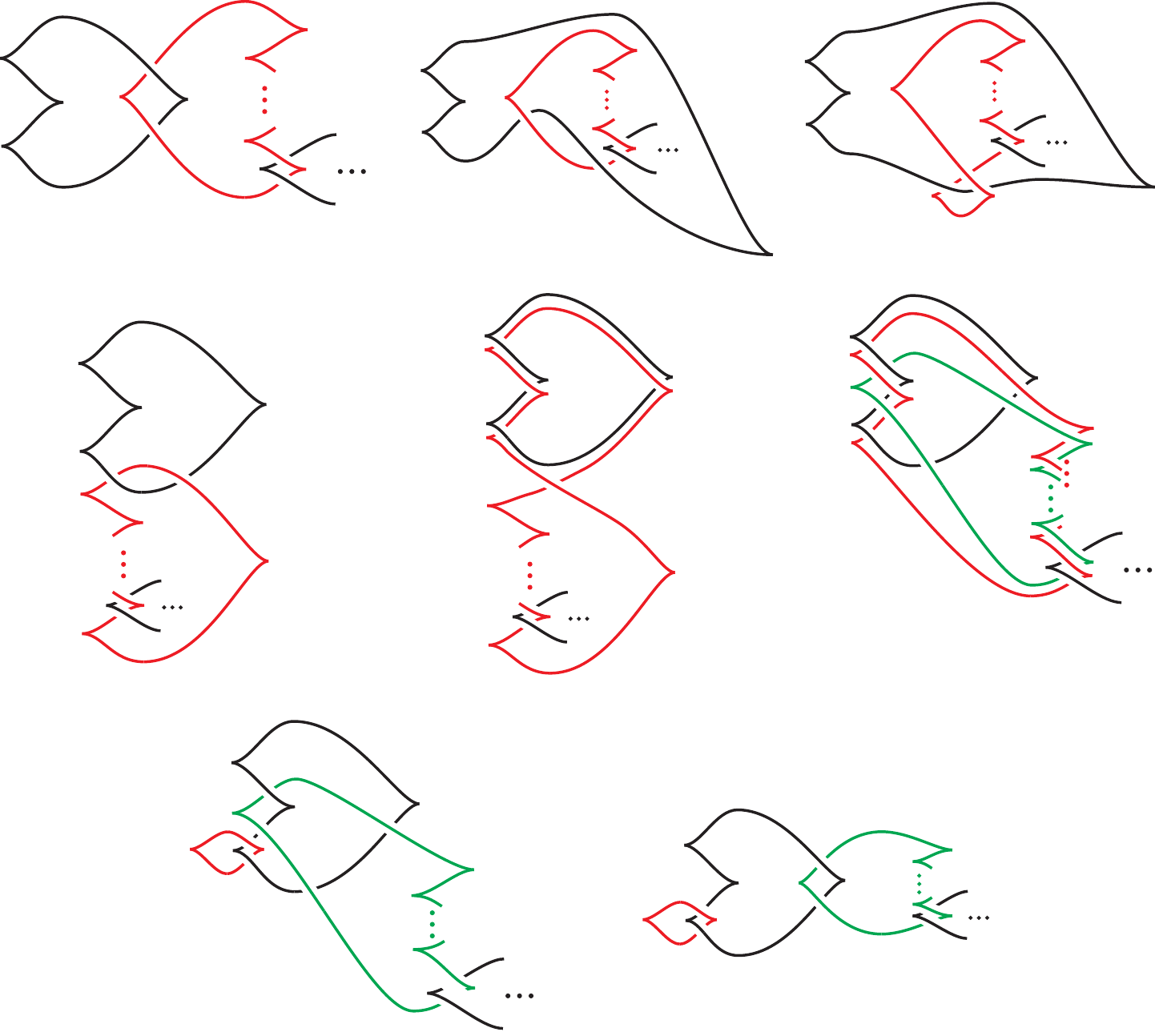}
\put(38, 366){$L_1$}
\put(119, 330){$L_{3}$}
\put(106, 370){$L_{2}$}
\put(353, 190){$L$}
\end{overpic}}
\caption{A Stein cobordism from the lens space described by the chain of Legendrian unknots in the upper left diagram to the one in the lower right digram that is given by attaching a Stein $2$--handle to $L$. Notice that $L$ has one less zig-zag that $L_2$ had in the upper left diagram.} 
\label{R2}
\end{figure}The first always gives a cobordism to a connected sum of lens spaces (to which one can add another $2$--handle to obtain a cobordisms to another lens space). The second cobordism can remove some inconsistent sub-chains from a surgery diagram. We more fully describe the handle attachment in Figure~\ref{R2}: the upper left diagram shows a portion of a chain of Legendrian unknots giving a contact structure on a lens space. In the next diagram a Type~II Legendrian Reidemeister move is performed by pushing the right most cusp of $L_1$ past the lower strand of $L_2$ followed by several Type~II and~III moves to push the upper strand of $L_2$ past the rest of the chain. In the next diagram, we perform a Type~I move to $L_2$ and Type~III to $L_1$ and $L_2$. In the first diagram on Row~2, we use Type~II and~III moves to push the zig-zags of $L_2$ and the rest of the chain outside of $L_1$ and then do a Type~I move to $L_2$. The next diagram is a handle slide of $L_2$ over $L_1$. In the last diagram of Row~2 we see that the zig-zags of $L_2$ and the rest of the chain are pushed to the upper right of $L_2$ and then a Type~I move is performed on $L_2$. We also see the knot $L$ in the lens space. Adding a Stein $1$--handle to $L$ and then sliding $L_2$ over $L$ yields the first figure on Row~3 with is clearly isotopic to the last figure on the row.

At this time we do not have a good geometric understanding of why these cobordisms exist as we do for Construction~3. Though one can see that the Legendrian knot $L$ in Figure~\ref{R1} is a cable of a torus knot, but we do not geometrically see why it has Thurston-Bennequin invariant allowing for a reducible surgery.  We end this discussion by asking if the second cobordism can be generalized. If so the answer to the following question could be YES.
\begin{question}
Given any virtually overtwisted contact structure on a lens space is there a Stein cobordism to a universally tight one?
\end{question}

We have given four ways to construct cobordisms and ask the following. 
\begin{question}
Do all Stein cobordisms between tight contact structures on lens spaces come from these four constructions?
\end{question}
We notice that a positive answer to this question would allow us to strengthen Theorem~\ref{length} to say that any cobordism between a tight contact structure on $L(p,q)$ and another such structure on the same lens space must be a piece of the symplectization of the contact structure.

{\bf Acknowledgements:} We are grateful to  Paolo Aceto, Paolo Lisca, Duncan McCoy, Hyunki Min, Anubhav Mukherjee, and JungHwan Park for useful conversations about the work in this paper. We also thank Marco Golla for giving a lot of valuable feedback on the first draft of the paper, including important references to the literature and a simplification of the proof of Theorem~\ref{length}. In addition, we are very grateful to Ken Baker who was instrumental in our understand of the knots in $S^3$ giving surgeries to lens spaces, Lenny Ng who was instrumental in helping us understand their Legendrian realizations, and Marc Kegel who showed us the surgery diagrams in Construction~4 above (and a surgery diagram that lead to our understanding in Construction~3).  We also thank the referee for many valuable comments and insights on the first drafts of the paper. Both of the authors were partially supported by NSF grant DMS-1906414.

\section{Continued fractions and contact structure on lens spaces}
In this section we will recall background results and establish some preliminary results. Specifically, in Sections~\ref{cf}, \ref{fgraph}, and~\ref{pathinfg} we recall continued fractions, the Farey graph, and paths in the Farey graph, respectively. This is all used in Section~\ref{koclassification} to recall the classification of tight contact structures on lens spaces and solid tori. A technical result about breaking up continued fractions will be established in Section~\ref{breakingup}.

In Section~\ref{legsimpofratunknots} we prove the Legendrian simplicity of rational unknots in lens spaces and establish the same result for connected sums of such knots. Section~\ref{lclass} recalls Lisca's classification of minimal symplectic fillings of universally tight contact structures on lens spaces, Section~\ref{alt} gives an alternate description of these fillings, and Section~\ref{combiningchains} shows how to build fillings of a lens space described by surgery on a chain from fillings of lens spaces described by sub-chains. Finally, in Section~\ref{menkedecompose} we recall Menke's theorem that tells us how to decompose symplectic fillings of virtually overtwisted contact structures on a lens space. 

\subsection{Continued fractions}\label{cf}
Given a collection of integers $a_1,\ldots, a_n$ we denote by $[a_1,\ldots, a_n]$ the continued fraction
\[
a_1-\cfrac{1}{a_2-\cfrac{1}{\cdots -\cfrac{1}{a_n}}}.
\]
The lens space $L(p,q)$ is by definition the result of $(-p/q)$--surgery on the unknot and one may always assume that $-p/q<-1$. There are uniquely determined integers $a_1,\ldots, a_n$ larger than $1$ such that $-p/q = [-a_1, \ldots, -a_n]$. Thus $p/q=[a_1, \ldots, a_n]$. In Lisca's classification of symplectic fillings of universally tight contact structure on lens spaces we will need to consider the continued fractions expansion of ${p}/{(p-q)}$. If we write $p/(p-q)$ in the continue fraction $[b_1,\ldots, b_m]$ with the $b_i>1$ then the $b_i$ and $a_j$ are related by Riemenschneider point diagrams \cite{Riemenschneider74}. Given the expansion of $[a_1,\ldots, a_n]$ we make a diagram consisting of dots. The first row has $a_1-1$ dots, then the $j^{th}$ row has $a_j-1$ dots with the first dot appearing right under the last dot of the previous row.  We now recover the $b_i$ by counting the dots in a column. Specifically $b_i-1$ is the number of dot in the $i^{th}$ column. In Figure~\ref{change} we see the Riemenschneider diagram for $84/19=[5,2,4,3]$ and easily compute $84/65=[2,2,2,4,2,3,2]$.

\begin{figure}[htb]{\tiny
\begin{overpic}
{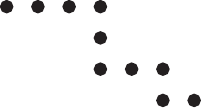}
\end{overpic}}
\caption{The Riemenschneider point diagram for $84/19$.}
\label{change}
\end{figure}

\subsection{The Farey graph}\label{fgraph}
The Farey graph is a useful tool for keeping track of embedded essential curves on the torus and is used in the statement of the classification of tight contact structures on the thickened torus. 
\begin{figure}[htb]{\small
\begin{overpic}
{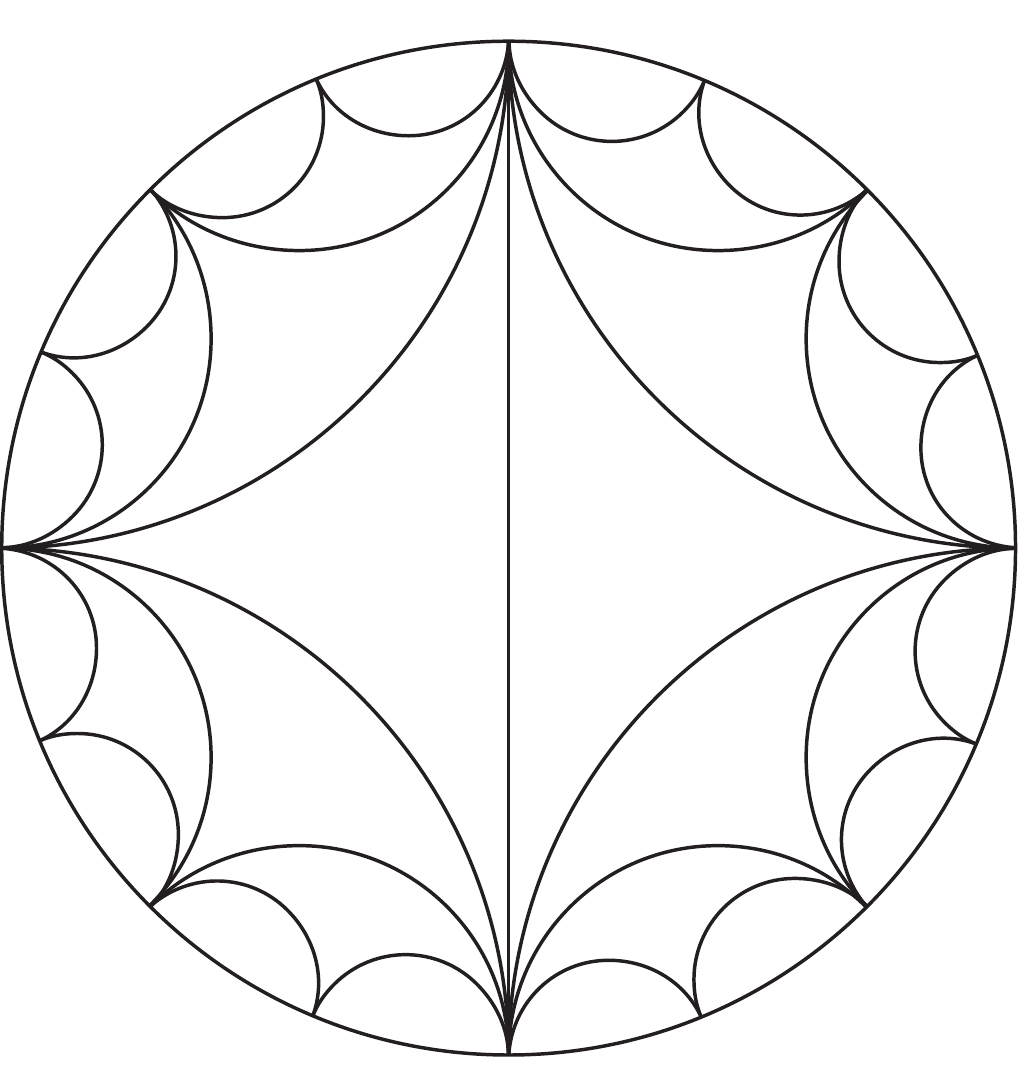}
\put(142, 2){$\infty$}
\put(145, 306){$0$}
\put(-15, 154){$-1$}
\put(297, 154){$1$}
\put(29, 43){$-2$}
\put(254, 43){$2$}
\put(19, 261){$-1/2$}
\put(254, 261){$1/2$}
\put(72, 296){$-1/3$}
\put(200, 296){$1/3$}
\put(-13, 212){$-2/3$}
\put(285, 212){$2/3$}
\put(-14, 96){$-3/2$}
\put(286, 96){$3/2$}
\put(76, 10){$-3$}
\put(203, 10){$3$}
\end{overpic}}
\caption{The Farey graph.}
\label{fareygraph}
\end{figure}
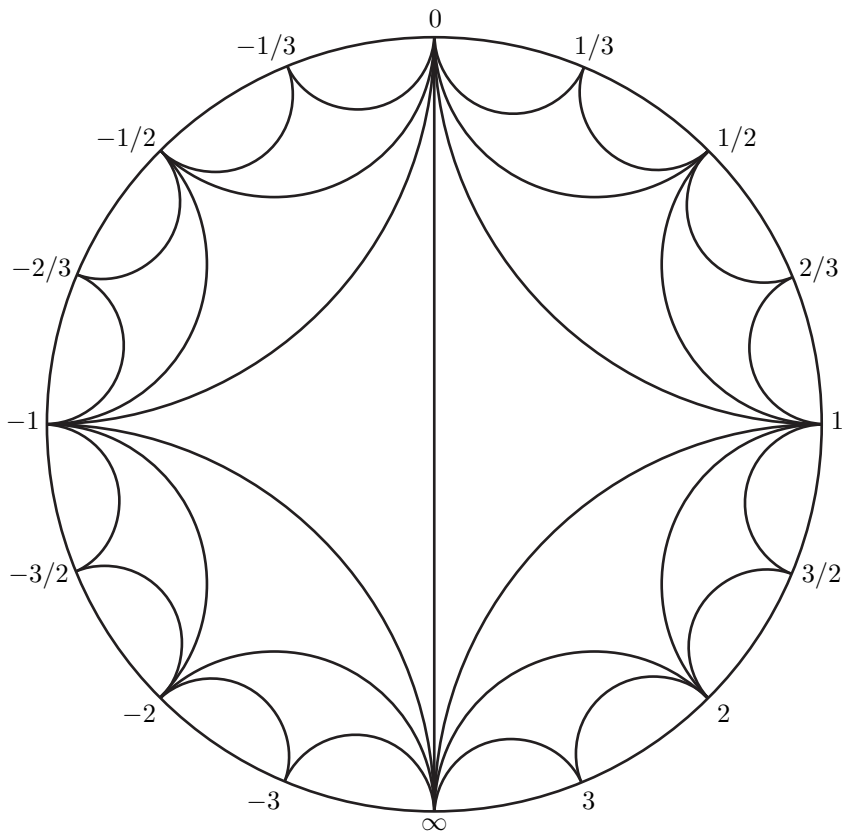
Recall that embedded essential curves on the torus are in one-to-one correspondence with the rational numbers (in lowest terms) union infinity. 

Given two rational numbers $a/b$ and $p/q$ in lowest terms we define their ``Farey sum" to be $a/b\oplus p/q= (a+p)/(b+q)$ we will also use $a/b\oplus n(p/q)$ to mean ``Farey summing" $p/q$ to $a/b$, $n$ times. 
Consider the unit disk in $\R^2$. Label the point $(0,1)$ by $0=0/1$ and $(0,-1)$ by $\infty=1/0$.  Now for a point on the unit circle with non-negative $x$-coordinate, if it is half-way between two labeled points, label it with the "Farey sum" of those two points and connect it to both of those by a hyperbolic geodesic (we can consider a hyperbolic metric on the unit disk). Iterate this process until all the positive rational numbers are a label on some point on the unit disk. Now do the same for the points with non-positive $x$-coordinate (except for $\infty$ use the fraction $-1/0$). See Figure~\ref{fareygraph}.

Notice that if we take two curves on $T^2$ of slopes $r$ and $s$, then they form a basis for $H_1(T^2)$ if and only if there is an edge in the Farey graph between them. 

{\em We note that the Farey graph can be constructed in many ways and not all papers use the same convention.}

We will use the notation $[a,b]$ to denote the region on the unit circle that is counterclockwise of $a$ and clockwise of $b$ with similar meanings for $(a,b]$, $[a,b)$, and $(a,b)$. 

We recall a well known lemma about continued fractions and the Farey graph. 
\begin{lemma}\label{rightandleft}
Given a continued fraction $x=[-a_1, \ldots, -a_n]$ with all $a_i\geq 2$, then 
$y=[-a_1,\ldots, -a_n+1]$ and $z=[-a_1,\ldots, -a_{n-1}]$ both have an edge in the Farey graph to $x$ and the former is clockwise of $x$ while the latter is counterclockwise of it. Moreover, there is an edge between the $y$ and $z$.
\end{lemma}

\subsection{Paths in the Farey graph}\label{pathinfg}
When considering a minimal path in the Farey graph we will be interested in continued fraction blocks. 
Any minimal path in the Farey graph will be a sequence of continued fraction blocks. So below we will define a continued fraction block and how to transition from one to another as one traverses a path. 

A \dfn{continued fraction block of length $k$} is a sequence of $k$ ``half-maximal jumps" in the Farey graph. See Figure~\ref{CFB}. 
\begin{figure}[htb]{\tiny
\begin{overpic}
{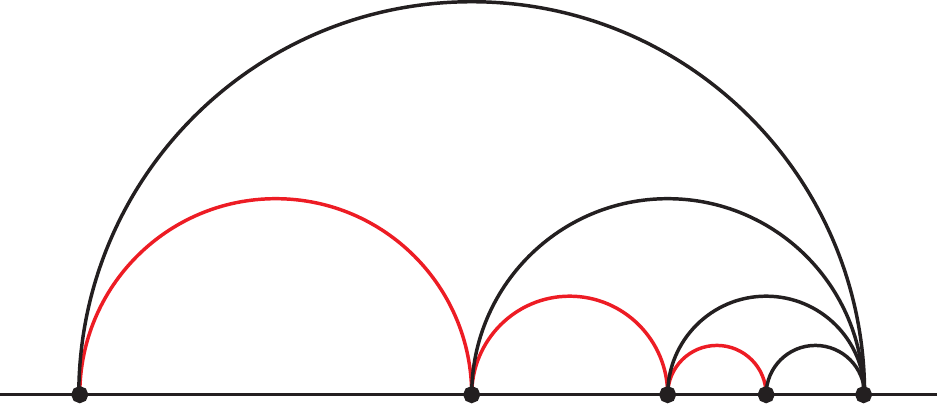}
\put(20, -2){$v_1$}
\put(133, -2){$v_2$}
\put(190,-2){$v_3$}
\put(220, -2){$v_4$}
\put(246, -2){$m$}
\end{overpic}}
\caption{A continued fraction block of length $3$.}
\label{CFB}
\end{figure}
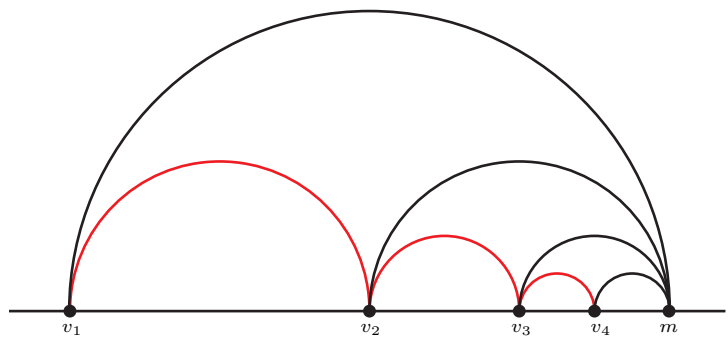
More precisely, we will consider paths going counterclockwise (in the figure this is represented as going left to right). If we have an edge from $v_1$ to $v_2$ then there is a unique vertex $m$ in the Farey graph that is outside of the region $[v_1,v_2]$ and has edges to both $v_1$ and $v_2$. Now inductively define $v_i$ to be the vertex in $[v_{i-1}, m)$ closest to $m$ with an edge to $v_{i-1}$, for $i\leq k+1$. The union of the $k$ edges between these vertices is the continued fraction block of length $k$. If $m<v_1<0$, then the $v_i$ are simply $v_1\oplus (i-1)m$.
Notice that each jump in the continued fraction block is a half-maximal jump in the sense that the only other jump that is larger than the one made (and does not return to $v_1$) would be to $m$. So the jump from any $v_l$ to $m$ would be ``maximal", while the jump to $v_{l+1}$ is ``half-maximal".  One could define a continued fraction block going clockwise as well. Notice that the notions are symmetric, that is if $v_1,\ldots, v_k$ are the vertices in a continued fraction block moving counterclockwise, then $v_k,\ldots, v_1$ is a continued fraction block moving clockwise. 

Now if one is traversing a path and at vertex $v_k$ in the path, the next jump is not part of the previous continued fraction block, then it will be the beginning of a new continued fraction block. If $v'$ would have been the next jump in the continued fraction block of which $v_k$ is a part, then the possible jumps from $v_k$ are $w^i_1=v_k\oplus iv'$. If the next jump is to $w^i_1$ then we say we started the continued fraction block that is $i$ down from the previous one. See Figure~\ref{NewCFB}.
\begin{figure}[htb]{\tiny
\begin{overpic}
{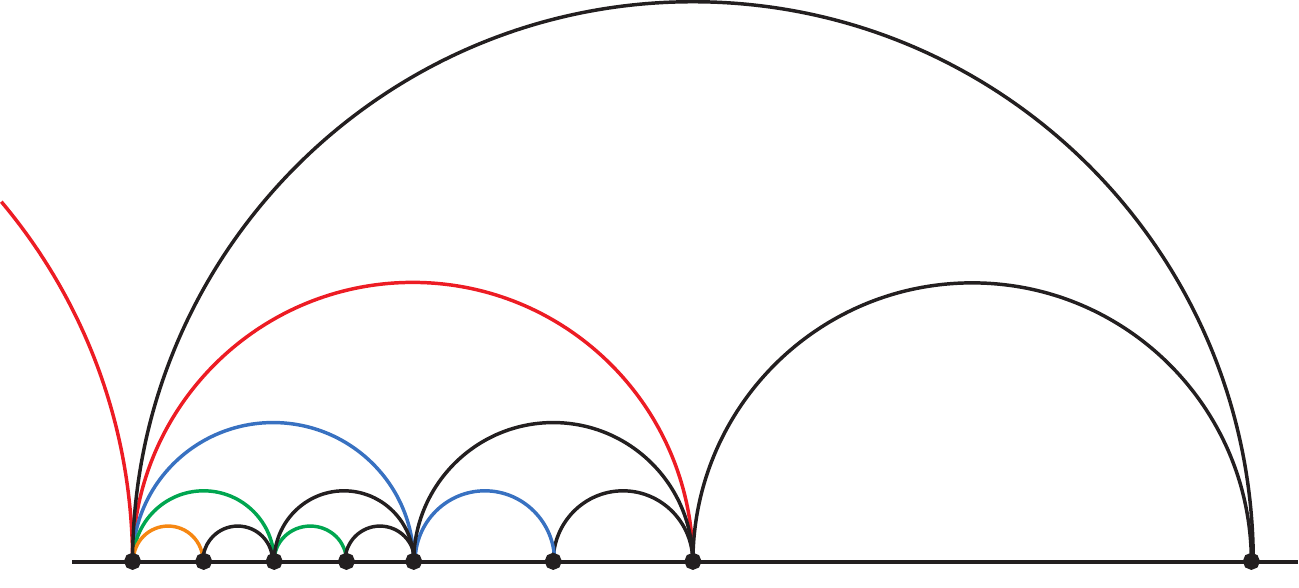}
\put(36, -2){$v_k$}
\put(198, -2){$v'$}
\put(118,-2){$w^1_1$}
\put(158, -2){$w^1_2$}
\put(77,-2){$w^2_1$}
\put(98, -2){$w^2_2$}
\put(56,-2){$w^3_1$}
\put(359, -2){$m$}
\end{overpic}}
\caption{Transitioning between continued fraction blocks. The blue path is the continued fraction block $1$ down from the red. The green path is $2$ down from the red, and the orange is $3$ down from the red.}
\label{NewCFB}
\end{figure}
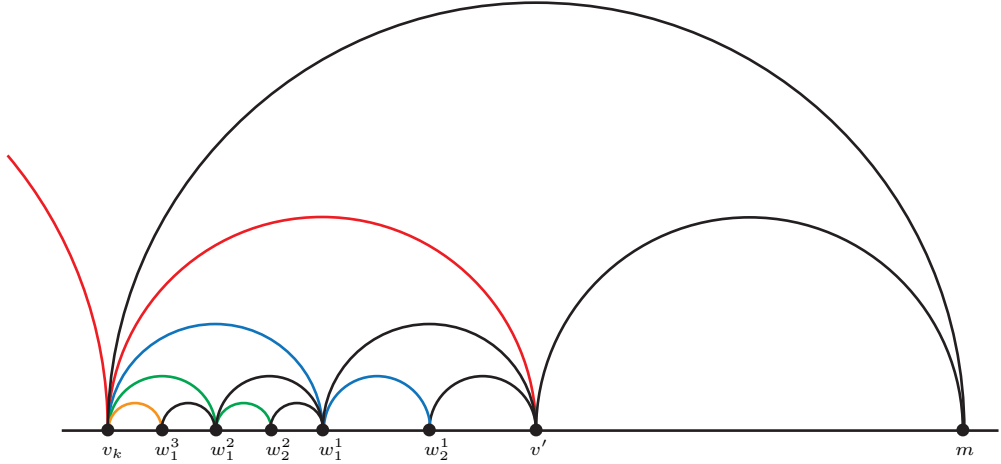
We make a simple observation.
\begin{observation}\label{exceptional}
If $[v_1,\ldots, v_k]$ is a continued fraction block that is followed by the continued fraction block $[w_1^i,\ldots, w^i_l]$ that is $i$ down from it, then the vertices with an edge to $v_k$ that are in $[w_1^i,v_{k-1}]$ are precisely $w^{i-1}_1, \ldots, w^1_1, w_1^0$, and $m$, where $w_1^0=v'$. We call these the \dfn{exceptional vertices} of the transition between continued fraction blocks. 
\end{observation}

It is clear that any minimal path in the Farey graph is a continued fraction block, followed by one that is $i$ down from it, and so on. 

We now explain the reason for the terminology ``continued fraction block". Suppose $-p/q<-1$ has the continued fraction decomposition $[-a_1,\ldots, -a_n]$ where each $a_i\geq 2$. By Lemma~\ref{rightandleft} we see that $[-a_1,\ldots, -a_n+1]$ is clockwise of $[-a_1,\ldots, -a_n]$ and has an edge to it. It is a maximal jump towards $0$ that has an edge back to $[-a_1,\ldots, -a_n]$. If we continue to add one to the last element on the continued fraction, subject to the relation that $[-a_1,\ldots, -a_k, -1]=[-a_1,\ldots, -a_k+1]$, we will get a sequence of vertices realizing a minimal path from $-p/q$, clockwise, to $0$. In particular, notice that there are $|a_n-1||a_1-1|\prod_{i=2}^{n-1} |a_i-2|$ jumps in this path, if $n>1$ and $a_1$ jumps if $n=1$. This should be clear since the jumps associated to $[-a_1,\ldots, -a_k]$ go from $[-a_1,\ldots, -a_k+1]$ to $[-a_1,\ldots, -1]$ if $k \not= 1$ or $n$. For $k=1$ or $n$ we have a similar computation. 

\begin{construction}[Building paths from continued fractions]\label{cftopaths}
Turning this around we can get from $0$ to $-p/q$ by moving counterclockwise using the following algorithm: First assume that none of the $a_i$ are $2$ and $n>1$. Starting at $0$ make a jump to $-1$ and continue in a continued fraction block of length $a_0-1$, then move to the continued fractions block one down from this one and make $a_2-2$ jumps, continue in this manner until the last continued fraction block which will be of length $a_n-1$. Now suppose there are some $2$'s among the $a_i$. In particular suppose between $a_k$ and $a_{k+l}$ all the $a_i$ are $2$'s but $a_k$ and $a_{k+1}$ are not $2$. Then when transitioning from the continued fraction block associate to $a_k$ we will start a continued fraction block $l$ down from it corresponding to $a_{k+l}$. 

So we see from the above construction that a continued fraction block corresponds to the vertices $[-a_1,\ldots, -a_k], [-a_1,\ldots, -a_k+1], \ldots, [-a_1,\ldots, -1]$, indicating the reason for the name. 
\end{construction}

We also notice how a continued fraction $[-a_1,\ldots, -a_n]$ can describe a path in the Farey graph that does not start at $0$. Suppose we choose a seed vertex $v_s$ and target vertex $v_t$ with an edge to $v_s$. For convenience, we assume that the rational numbers corresponding to $v_s$ and $v_t$ are negative. We will now describe a minimal path in $[v_s,v_t]$ based on $[-a_1,\ldots, -a_n]$. Assuming none of the $a_i$ are $2$, we start the path with an edge from $v_s$ to the unique vertex in $(v_s,v_t]$ with an edge to both endpoints and then continue with a continued fraction block of length $a_1-1$ then move to the continued fractions block one down from this one and make $a_2-2$ jumps, continue in this manner until the last continued fraction block which will be of length $a_n-1$. If there are $2$'s then we simply proceed as described above (notice that the description above involving the $2$'s did not involve anything about the path starting at $0$). Denote this path by $P^{v_s}_{v_t}[-a_1,\ldots, -a_n]$. Notice that $P^0_\infty[-a_1,\ldots, -a_n]$ is just the shortest path from $0$ to $[-a_1,\ldots, a_n]$. 
\begin{observation}
There is a change of basis for $\Z^2$ taking $P^{v_s}_{v_t}[-a_1,\ldots, -a_n]$ to the minimal clockwise path from $0$ to $-p/q$.
\end{observation}
This should be clear as there is a change of basis taking $v_s$ to $0$, $v_t$ to $\infty$, and $v_s\oplus v_t$ to $-1$ and then both paths are described by the same number and types of jumps (which is of course preserved by the change of basis). 

\subsection{Breaking up a continued fraction}\label{breakingup}
Given a continued fraction $[-a_1,\ldots, -a_n]$ for $-p/q<-1$ with $a_i\geq 2$, we see that the lens space $L(p,q)$ is described by surgery on the linear chain $C$ of unknots $U_1,\ldots, U_n$ shown in Figure~\ref{chain}, where the unknot $U_i$ has framing $-a_i$. If we remove $U_k$ from the chain $C$ we will get two other chains $C_s$ and $C_e$ (one might be empty if the removed unknot was on the end of the chain). In this section we would like to understand the lens spaces corresponding to these chains in terms of the continued fraction expansion of $-p/q$. 

To this end we establish some notation. Given $T^2$ with some coordinates chosen on $H_1(T^2)$, we can consider $T^2\times [0,1]$ and collapse the curves of slope $a/b$ on $T^2\times \{0\}$ and $r/s$ on $T^2\times \{1\}$. This will give a lens space and we call $a/b$ the lower meridian of the lens space and $r/s$ the upper meridian. In particular, the lens space with upper meridian $0$ and lower meridian $-p/q$ is $L(p,q)$. 

\begin{lemma}\label{bucfb}
With the notation above, suppose that $a_k$ is not $2$ and $k\not=1$ or $n$. Let $v_1,\ldots, v_{k-1}$ be the vertices associated to the $a_k$ continued fraction block and $m$ is the maximal jump from these, see Figure~\ref{CFB}. 

The lens space with upper meridian $0$ and lower meridian $m$ is $L(p',q')$ where $-p'/q'$ is given by $[-a_1,\ldots, -a_{k-1}]$ and the lens space with upper meridian $m$ and lower meridian $-p/q$ is $L(p'',q'')$ where $-p''/q''= [-a_{k+1}, \ldots, -a_n]$. That is, the two lens spaces we get are the ones described by the two chains one gets by removing $U_k$ from $C$. 

If $k=n$ then the first lens space has the same description and the second one is $S^3$ and similarly for $k=1$. 
\end{lemma}
\begin{proof}
By Lemma~\ref{rightandleft} we see that $m=[-a_1,\ldots, -a_{k-1}]$ and so the claim about $L(p'.q')$ is clear. 

The change of basis sending $m$ to $0$ and $m\oplus v_{k-1}$ to $-\infty$ will send $v_{k-1}$ to $-1$ and thus we see that $P^m_{m\oplus v_{k-1}}[-a_{k+1}, \ldots, -a_n]$ is sent to a path from $0$ to $[-a_{k+1}, \ldots, -a_n]$ and so the second lens space is as claimed. 

When $k=n$ then notice that $[-a_1,\ldots, -a_n]$ has an edge to $m$ and so the second lens space is $S^3$. We can similarly argue the case $k=1$. 
\end{proof}

\begin{lemma}\label{bucfb2}
With the notation above, suppose that between $a_k$ and $a_{k+i}$ all the $a_j$ are $2$'s but $a_k$ and $a_{k+i}$ are not $2$. That is the continued fraction blocks corresponding to $a_{k+i}$ is $i$ down from the one corresponding to $a_k$. We consider the exceptional vertices $w^{i-1}_1, \ldots, w_1^0$ from Observation~\ref{exceptional} (notice we alraedy understand the exceptional vertex $m$ by Lemma~\ref{bucfb}). 

The lens space with upper meridian $0$ and lower meridian $w_1^j$ is $L(p',q')$ where $-p'/q'$ is given by $[-a_1,\ldots, -a_{k+j}]$ and the lens space with upper meridian $w_1^j$ and lower meridian $[-a_1,\ldots, -a_n]$  is $L(p'',q'')$, where$-p''/q''= [-a_{k+j+2}, \ldots, -a_n]$. That is, the two lens spaces we get are the ones described by the two chains one gets by removing $U_{k+j+1}$ from $C$. 

As in Lemma~\ref{bucfb} if $k=1$ or $k+i=n$ then one of the lens spaces is $S^3$ and the other is as described above. 
\end{lemma}

\begin{proof}
The identification of $-p'/q'$ is clear from the construction of a path in the Farey graph starting at $0$ from a continued fraction.

To identify $-p''/q''$ consider the change of basis sending $w^j_1$ to $0$ and $w^{j+1}_1$ to $-\infty$. This change of basis will clearly send $[-a_1, \ldots, -a_k]$ to $-1$ (since this point has an edge to both $w^j_1$ and  $w^{j+1}_1$ and in contained in $[w^j_1, w^{j+1}_1]$). We know $-p/q$ is described by $P^{w^j_1}_{w^{j+1}_1} [-a_{k+j+2}, \ldots, -a_n]$ and thus under the above coordinate change is sent to $[-a_{k+j+2}, \ldots, -a_n]$, as claimed. 
\end{proof}

\subsection{Contact structures on lens spaces}\label{koclassification}
Here we recall the classification of tight contact structures on lens spaces given by Giroux \cite{Giroux00} and Honda \cite{Honda00a}. We will use Honda's notation and also recall classification on $T^2\times [0,1]$ and $S^1\times D^2$. 

Given a minimal path in the Farey graph, we call it a decorated path if each edge has been given a $+$ or $-$ sign. Two decorations on a path are called equivalent if they each give the same number of $+$ signs to each continued fraction block. 
\begin{theorem}
Let $\Gamma_i$ be a pair of parallel homologically essential curves on $T^2\times\{i\}$  of slope $r_i$, $ i=0,1$. Let $P$ be a minimal path in the Farey graph starting at $r_0$ moving clockwise and ending at $r_1$. The minimally twisting tight contact structures on $T^2\times [0,1]$ with dividing curves $\Gamma_0\cup \Gamma_1$ are in one-to-one correspondence with equivalence classes of decorations on $P$. 
\end{theorem}
In particular, we notice that if $r_0$ and $r_1$ share an edge in the Farey graph, then there are exactly two minimally twisting contact structures on $T^2\times [0,1]$ with these boundary conditions. These are called \dfn{basic slices} and the correspondence in the theorem is given by stacking basic slices according to the path.

We can create a solid torus from $T^2\times [0,1]$ by collapsing curves on $T^2\times \{0\}$ of slope $r$. We can similarly construct one by collapsing curves of slope $r$ on $T^2\times \{1\}$. Say the first torus has a \dfn{lower meridian} of slope $r$, and the second torus has an \dfn{upper meridian} of slope $r$. 
\begin{theorem}
Let $\Gamma$ be a pair of homologically essential curves on $\partial S^1\times D^2$ with a lower meridian of slope $m$. Denote the slope of $\Gamma$ by $r$. Let $P$ be a minimal path in the Farey graph that starts at $m$ and goes clockwise to $r$. Then the tight contact structures on $S^1\times D^2$ with dividing curves $\Gamma$ are in one-to-one correspondence with equivalence classes of decorations on $P$, except the first edge (the one starting at $m$) is left blank. 

The same holds if $S^1\times D^2$ has an upper meridian except the path will go counterclockwise from $m$ to $r$. 
\end{theorem}

It is useful to know, see \cite{Honda00a},  that given a contact structure on a solid torus as in the theorem, one can always realize a $(a,b)$-curve on a torus parallel to the boundary as a Legendrian divide if $b/a\in [r,m)$ in the first case and $b/a\in (m,r]$ in the second case. 

We now state the classification of tight contact structures on lens spaces. The terminology used here was introduced in the previous section. 
\begin{theorem}\label{classifylpq}
Consider a lens space $L$ with upper meridian $r$ and lower meridian $s$. Let $P$ be a minimal path in the Farey graph from $s$ going clockwise to $r$. The tight contact structures on $L$ are in one-to-one correspondence with equivalence classes of decorations on $P$ with the first and last edge left blank. 

A contact structure on $L$ is universally tight if and only if the decoration only uses one sign.  
\end{theorem}
The lens space $L(p,q)$ is $(-{p}/{q})$--surgery on the unknot in $S^3$, where $-{p}/{q} < -1$. That is it has upper meridian $0$ and lower meridian $-\frac{p}{q}$. So the tight contact structures on $L(p,q)$ are in one-to-one correspondence with equivalence classes of decorations on a minimal path from $0$ counter-clockwise to $-p/q$ (with the first and last edge left blank). If $-p/q=[-a_1,\ldots, -a_n]$, then the count of jumps in the path from $0$ to $-p/q$ given in Section~\ref{pathinfg} shows there are $\prod_{i=1}^{n} |a_i-2|$ edges that need to be decorated (recall the first and last edge are not decorated) and each continued fraction block can be decorated in $a_i-1$ different (equivalence classes of) ways. Thus there are $\prod_{i=1}^{n} |a_i-1|$ tight contact structures on $L(p,q)$. 

\begin{remark}\label{pathsandknots}\label{classifylpq2}
All of the contact structures on $L(p,q)$ are realized by Legendrian surgery on a Legendrian realization of the chain of unknots $\{U_1,\ldots, U_n\}$ in Figure~\ref{chain} with the $i^{th}$ unknot having Thurston-Bennequin invariant $-a_i +1$. Notice that there are $a_i-1$ ways to stabilize $U_i$. Thus we see the stabilizations of $U_i$ correspond precisely to the signs in the decoration of the continued fraction block associated to $a_i$. (Actually there are choices one needs to make when assigning a $\pm$ to a basic slice: so we can make that choice so that the sign corresponds to the sign of the stabilization.)
\end{remark}

\subsection{Rational unknots in lens spaces}\label{legsimpofratunknots}
Given a lens space $L(p,q)$ we call a knot $K$ a \dfn{rational unknot} if it is the core of a Heegaard torus for $L(p,q)$. In \cite{BakerEtnyre12}, Baker and the first author showed there is either $1$, $2$, or $4$ such knots  when considered as oriented knots. In Figure~\ref{ratunknots} we see one or two such (unoriented) knots. 
\begin{figure}[htb]{\tiny
\begin{overpic}
{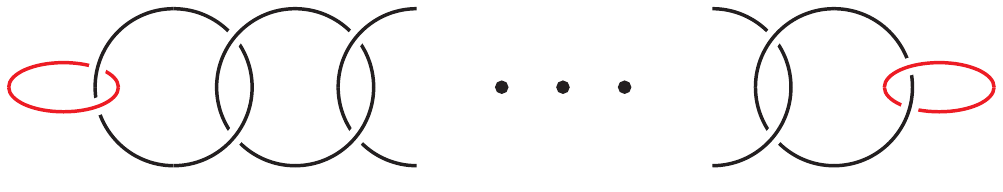}
\put(12, 13){$K_1$}
\put(43, 0){$-a_1$}
\put(78, 0){$-a_2$}
\put(237, 0){$-a_n$}
\put(270, 13){$K_2$}
\end{overpic}}
\caption{Rational unknots $K_1$ and $K_2$ in $L(p,q)$.}
\label{ratunknots}
\end{figure}
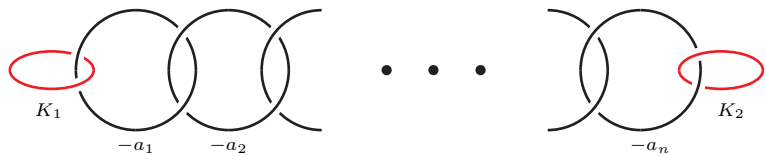
Recall that a knot type $K$ in a contact manifold $(M,\xi)$ is called weakly Legendrian simple, if Legendrian representatives of them are determined up to a contactomorphism, smoothly isotopic to the identity, by their Thurston-Bennequin invariant and rotation number (if the knots is only rationally null-homologous, then there are rational notions of these invariants \cite{BakerEtnyre12}). 
\begin{theorem}[Baker and Etnyre, 2012 \cite{BakerEtnyre12}]
Rational unknots in tight contact structures on lens spaces are weakly Legendrian simple and there is a unique such knot with maximal Thurston-Bennequin invariant. 
\end{theorem}
It is easy to see that Legendrian realizing $K_1$ or $K_2$ in Figure~\ref{ratunknots} by a Thurston-Bennequin invariant $-1$ unknot gives a maximal Thurston-Bennequin invariant rational unknot in $L(p,q)$ with any tight contact structure. 

In \cite{EtnyreHonda03} the first author and Honda showed how to relate the classification of Legendrian knots in the smooth knot type $K_1$ in $(M_1,\xi_1)$ and $K_2$ in $(M_2, \xi_2)$ to the classification of Legendrian knots in the smooth knot type $K_1\# K_2$ in $(M_1\# M_2, \xi_1\# \xi_2)$. We do not need the full force of the theorem, but an immediate corollary of that work is the following. 
\begin{lemma}\label{weaklysimple}
Suppose $K_i$ is a (weakly) Legendrian simple knot type in $(M_i,\xi_i)$ and there is a unique maximal Thurston-Bennequin invariant representative $L_i$ of $K_i$, for $i=1,2$. Then $K_1\# K_2$ is (weakly) Legendrian simple in $(M_1\#M_2, \xi_1\# \xi_2)$ and has a unique maximal Thurston-Bennequin invariant representative $L_1\# L_2$. 
\end{lemma}

Given a continued fraction $[-a_1,\ldots, -a_n]$ for $-p/q<-1$ with $a_i\geq 2$, we know that the lens space $L(p,q)$ is described by surgery on the linear chain $C$ of unknots $U_1,\ldots, U_n$ the unknot $U_i$ has framing $a_i$. If we remove $U_k$ from the chain $C$ we will get two other chains $C_s$ and $C_e$. Notice that in $L(C_s)\#L(C_e)$ the knot $U_k$ is simply the connected sum of rational unknots in the lens spaces. From the above discussion we have the following understanding of Legendrian realizations of $U_k$. 
\begin{corollary}\label{legsimpleratunknots}
Given any tight contact structure on $L(C_s)$ and $L(C_e)$ the knot $U_k$ is Legendrian simple in $L(C_s)\#L(C_e)$ and realizing $U_k$ as the unknot with Thurston-Bennequin invariant $-1$ in the obvious surgery diagram for $L(C_s)\#L(C_e)$ gives the maximal Thurston-Bennequin invariant representative of $U_k$. All other Legendrian representatives of $U_k$ are stabilizations of this one. 
\end{corollary}

We note that this corollary is not strictly necessary for our main result on fillings of lens spaces, as the classification there is only up to diffeomorphism. However, we believe this result is of independent interest and note that if one can upgrade it to to say that the $U_k$ are Legendrian simple, and not just weakly Legendrian simple, then the proof of Theorem~\ref{list} would also classify minimal symplectic fillings up to symplectomorphism and symplectic deformation equivalence in some cases. 
In particular, if the fillings  of the lens spaces corresponding to the $\mathcal{C}-(\mathcal{D}\cup\mathcal{M})$ in the theorem are known up to symplectomorphism and symplectic deformation equivalence, then one could get a classification of the lens space of $\mathcal{C}$ up the same equivalence. For example, all the symplectic fillings of virtually overtwisted contact structures in Theorem~\ref{cor5} would be known up to symplectomorphism and symplectic deformation equivalence by results of McDuff and Hind \cite{Hind03,McDuff90} and the proof of Theorem~\ref{list} if the $U_k$ could be classified up to Legendrian isotopy. Similarly, many of the fillings in Theorem~\ref{3component} would be known up to symplectomorphism and symplectic deformation equivalence. We strongly believe that Corollary~\ref{legsimpleratunknots} can be upgraded to a classification up to Legendrian isotopy, but leave it as a conjecture for now.

\subsection{Lisca's classification of symplectic fillings of universally tight contact structures on lens spaces}\label{lclass}
Recall each lens space $L(p,q)$ has either two universally tight contact structures if $q\not= p-1$ or one if $q=p-1$. If $p/q=[a_1,\ldots, a_n]$, then a universally tight contact structure is realized by Legendrian surgery on a Legendrian realization of the chain of unknots in Figure~\ref{chain} with the $i^{th}$ unknot having Thurston-Bennequin invariant $-a_i +1$ and rotation number $a_i-2$. We denote this contact structure by $\xi_{ut}$. The other universally tight structure is $-\xi_{ut}$ and has a similar picture but the rotation numbers are $-a_i+2$ (when $q=p-1$ this will be isotopic to $\xi_{ut}$ but in the other cases it will not be isotopic; it will however be contactomorphic).

{\bf Null sequences:} We say the sequence of integers $(n_1,\ldots, n_s+1, 1, n_{s+1}+1,\ldots, n_k)$ is a \dfn{blowup} of the sequence $(n_1,\ldots, n_k)$ (with the obvious modification for $s=0$ or $k$), and the reverse operation is a \dfn{blowdown}. A blowup is \dfn{strict} if $s>0$. We say a sequence of integers is \dfn{null} if it is obtained by a sequence of strict blowups of $(0)$. For example, the only length two null sequence is $(1, 1)$ and the length three null sequences are $(1, 2, 1)$ and $(2, 1, 2)$. 

If $(n_1, \ldots, n_k)$ is a null sequence of integers, then the chain of framed unknots shown in Figure~\ref{Wn} has boundary $S^1\times S^2$. Later we will need the fact that the chain can be blown down so that it is the first unknot in the chain that is left. This is because all our blowups were strict. But if $\n=(n_1,\ldots, n_k)$ is a null sequence, then so is $\overline \n=(n_k,\ldots, n_1)$, see Section~2 of \cite{Lisca08}, so you can also blowdown the chain so that the right most unknot is the one that remains.
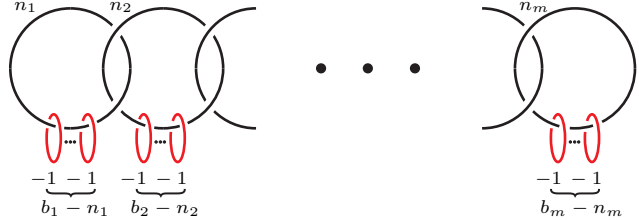
\begin{figure}[htb]{\tiny
\begin{overpic}
{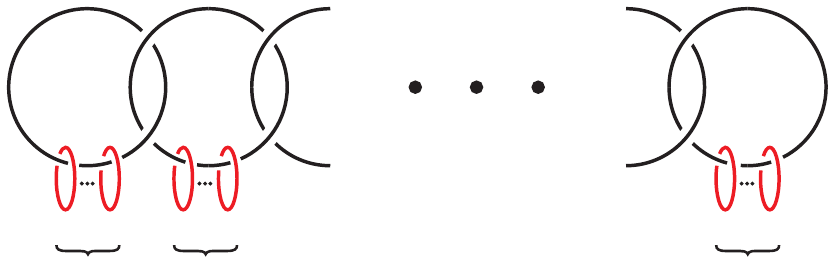}
\put(14, -2){$b_1-n_1$}
\put(48, -2){$b_2-n_2$}
\put(203,-2){$b_m-n_m$}
\put(10, 9){$-1 \,-1$}
\put(44, 9){$-1 \,-1$}
\put(201, 9){$-1 \,-1$}
\put(4, 75){$n_1$}
\put(40, 75){$n_2$}
\put(195, 75){$n_m$}
\end{overpic}}
\caption{Surgery on the black chain of unknots is gives $S^1\times S^2$. The red unknots form a link $L_{\n}$ in $S^1\times S^2$. }
\label{Wn}
\end{figure}

Let $[b_1,\ldots, b_m]$ be the continued fractions expansion of $p/(p-q)$.  Let $\BZ_{p,q}$ be the set of null sequences $(n_1,\ldots, n_m)$ such that $n_i\leq b_i$ for all $i=1, \ldots, m$. Now for any $\n\in \BZ_{p,q}$ we have the link $L_{\n}$ in $S^1\times S^2$ that is the image of the red curves in Figure~\ref{Wn} once the chain of black unknots has been blowndown to give $S^1\times S^2$. Finally, let $W_\n$ be the result of attaching $2$--handles to $S^{1} \times D^{3}$ along $L_\n$ in $\partial (S^1\times D^3)$. 

We can now state Lisca's classification of fillings of universally tight contact structures on lens spaces. 
\begin{theorem}[Lisca, 2008 \cite{Lisca08}]\label{lisca}
Given $-p/q<-1$.
\begin{enumerate}
\item For all $\n\in \BZ_{p,q}, W_\n$ admit Stein structures filling $(L(p,q),\xi_{ut})$. 
\item Any Stein filling of $(L(p,q), \xi_{ut})$ is diffeomorphic to $W_\n$ for some $\n\in \BZ_{p,q}$.  
\item For $\n$ and $\n'$ in $\BZ_{p,q}$,  $W_\n$ and $W_{\n'}$ are diffeomorphic if and only if either $\n'=\n$ or  $q^2\equiv 1 \!\! \mod p$ and $\n'=\n$ or $\n'=\overline\n$. 
\end{enumerate}
\end{theorem}
Notice that if $(X,J)$ is a Stein filling of $(M,\xi)$ then $(X,-J)$ is a Stein filling of $(M,-\xi)$ so the theorem above deals with all fillings of all universally tight contact structures on $L(p,q)$. 
\begin{remark}
The result in \cite{Lisca08} was stated for minimal symplectic fillings but as all lens spaces are supported by planar open books \cite{Schoenenberger05} any minimal symplectic fillings of a lens space is deformation equivalent to a Stein domain \cite{Wendl10}. So the statement above is equivalent to the one given in \cite{Lisca08}.
\end{remark}

\subsection{Alternative descriptions of Lisca's fillings of lens spaces}\label{alt}
Here we will describe a way to build Stein fillings of lens spaces and see that by Lisca's results in the previous section, this construction will give all fillings of universally tight contact structures on lens spaces. 

We can think of $S^1\times S^2$ as $T^2\times [0,1]$ where we have collapsed the $\infty$--sloped curves on $T^2\times \{0\}$ and $T^2\times\{1\}$. That is, in the terminology established in Section~\ref{breakingup}, both the upper and lower meridians are $\infty$.  Now the contact structure $\ker(\sin(\pi t)\, d\theta+\cos(\pi t)\, d\phi)$ on $T^2\times [0,1]$ (with coordinate $t$ on $[0,1]$ and cyclic coordinates $(\theta, \phi)$ on $T^2$, where $\theta$ corresponds to the $\infty$--slope) descends to $S^1\times S^2$. 

We can think of $T^2\times \{t\}, t\in (0,1),$ as the boundary of a neighborhood of $S^1\times \{N\}$, where $N$ is the north pole of $S^2$, and use the coordinates above to refer to curves on $T^2\times \{t\}$ by their slopes.

We have the following observation.
\begin{lemma}\label{newfills}
Each component of the link $L_\n$ in $S^1\times S^2$ that arises in Lisca's classification of fillings of lens spaces from Section~\ref{lclass} sits on a Heegaard torus for $S^1\times S^2$. That is, in the notation above, there are $t_1,\ldots, t_k\in [0,1]$ such that the components of $L_\n$ sit on $T^2\times \{t_1,\ldots, t_k\}$ (there can be more than one component on a given Heegaard torus, but we can also arrange that they are all on distinct tori). 

The framings on the components of $L_\n$ are one less than the framings induced by the tori on which they sit 
and the slopes of the components of $L_\n$ on the $T^2\times \{t_i\}$ are monotonic (that is, either decreasing or increasing, depending on the direction the chain is traversed). 
\end{lemma}
The knots that appear in the lemma are called torus knots in $S^1\times S^2$. 
\begin{proof}
Consider Figure~\ref{Wn}. Only consider the black curves that describe $S^1\times S^2$. Let $K$ be a meridian to the $k^\text{th}$ component in the chain. Slam dunk the $m^\text{th}$ unknot in the chain, and then the next component. Continue until the $(k+2)^\text{th}$ component has been slam dunked. Now similarly slam dunk the first unknot in the chain and continue until the $(k-1)^\text{th}$ unknot is slam dunked. We now have a chain with two components, each having rational surgery coefficient. See Figure~\ref{tknot}. 
\begin{figure}[htb]{\tiny
\begin{overpic}
{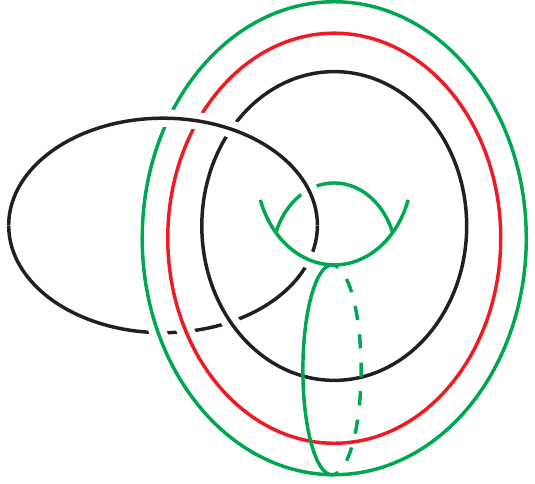}
\put(110, 20){$K$}
\put(118, 57){$r/s$}
\put(20,108){$a/b$}
\end{overpic}}
\caption{Surgery on the Hopf link that gives $S^1\times S^2$ and the knot $K$ and $K'$ sitting on the Heegaard torus.}
\label{tknot}
\end{figure}

There is a torus $T$ separating the two components of the chain. We know the chain describes $S^1\times S^2$ and $T$ clearly separates it into two solid tori. That is $T$ is a Heegaard torus for $S^1\times S^2$. Notice that $K$ sits on this Heegaard torus and the $-1$ framing on $K$ is one less than the framing coming from $T$. It is well-known that there is a unique genus one Heegaard splitting of $S^1\times S^2$ \cite{Waldhausen68}, so we see that any meridian to any component of the chain in Figure~\ref{Wn} is a torus knot. 

Let $s_i$ be the slope of the $i^{th}$ meridian on $T^2\times \{t_i\}$. It is clear that $s_1=-\infty$ and we claim that the slope of $s_i$ is $[n_1,\ldots, n_{i-1}]$. Given this, from our discussion in Section~\ref{pathinfg}, it should be clear that the slopes are monotonically increasing. To see the claim notice that the $i^{th}$ meridian when pushed inside the $i^{th}$ surgery torus will be a longitude, and when all the unknots are slam dunked to leave just the first unknot then it is easy to see that the boundary of the $i^{th}$ surgery torus is mapped into $S^3$ by the map $\begin{bmatrix} b_i^* &b_i\\ c_i^* & c_i\end{bmatrix}$ where $c_i/b_i= [n_1, \ldots, n_i]$. From Lemma~\ref{rightandleft} it is easy to see that $c^*_i/b^*_i=[n_1, \ldots, n_{i-1}]$ and this is the slope to which the longitude in the $i^{th}$ torus is mapped.
\end{proof}
From our description of $S^1\times S^2$ and the contact structure $\xi_{std}$ above, notice that each $T^2\times\{t\}$ is linearly foliated by Legendrian curves and the slope of these curve range from $-\infty$ up to $\infty$ as $t$ increases. 
\begin{corollary}\label{reinterpret}
Let $L$ be a collection of linear Legendrian curves on the Heegaard tori of $(S^1\times S^2, \xi_{std})$ as discussed in the previous sentence. Then Legendrian surgery on $L$ produces a Stein filling of a lens space. Moreover, all of Lisca's fillings are of this type. 
\end{corollary}
\begin{remark}
We notice that there are fillings you can construct from this corollary that do not obviously come form Lisca's construction, but by his classification theorem they will be diffeomorphic to one constructed by his algorithm. 
\end{remark}
\begin{proof}
Since the Heegaard tori are foliated by Legendrian curves, the contact framing and the framing induced by the torus are the same. So Legendrian surgery on one of these Legendrian curves corresponds to a $-1$--surgery on the knot with respect to the torus framing. It is well-known that this is equivalent to cutting the manifold along the torus and regluing by a right handed Dehn twist along the curve. Thus we obtain a lens space, even after surgery on several such curves. 

As $S^1\times S^2$ is Stein filled by $S^1\times D^3$ and Legendrian surgery corresponds to attaching a Stein $2$--handle along the Legendrian knot, it is clear that our lens spaces are Stein fillable. 

The above discussion shows that the link on which one surgers to obtain Lisca's fillings of lens spaces is a link as described in the corollary by Lemma~\ref{newfills}.
\end{proof}
We would now like to observe some symmetries in the Legendrian torus knots in $S^1\times S^2$. To this end we first notice that we have an obvious symmetry of $S^1\times D^3=(S^1\times [-1,1])\times D^2$ with the symplectic structure $d\theta\wedge dt + dx\wedge dy,$ where $\theta$ is the angular coordinate on $S^1$, $t$ the coordinate on $[-1,1]$, and $(x,y)$ are Euclidean coordinates on $D^2$. This is clearly an exact filling of $S^1\times S^2$ as exhibited by the Liouville vector field $t\frac{\partial}{\partial t} +\frac 12(x\frac{\partial}{\partial x} + y\frac{\partial}{\partial y})$. Now consider the symplectomorphism $F:S^1\times D^3\to S^1\times D^3$ given by $F(\theta, t, x,y)= (\theta, t, r_\theta(x,y)),$ where $r_\theta:D^2\to D^2$ is rotation by $\theta$. This also induces a contactomorphism of $S^1\times S^2$ that fixes $S^1\times \{N,S\}$, where $N$ and $S$ are the north and south poles of $S^2$, respectively. 

\begin{remark}\label{fixslope}
We notice that given a Legendrian torus knot $K$ of slope $r/s$ then $F(K)$ is a Legendrian torus knot of slope $(r+s)/s$. Moreover, attaching a Stein $2$--handle to $S^1\times D^3$ along $K$ gives a symplectic manifold that is symplectomorphic to the one obtained by attaching the handle to $F(K)$. Applying $F$ or $F^{-1}$ repeatedly we can assume that the knot $K$ has slope in $(-1,0]$. 

Notice that this explains why in Lisca's construction all the $2$--handles are attached along negative torus knots. {\em A priori} one could attach the handles along positive torus knots too but after applying $F^{-1}$ enough times, one can arrange that all the torus knots in the link used to construct the Stein filling have negative slope. 
\end{remark}

\subsection{Combining symplectic fillings of universally tight contact structures on lens spaces}\label{combiningchains}

We would now like to understand how fillings of lens spaces associated to sub-chains of $[a_1,\ldots, a_n]$ relate to fillings of the lens space associated to the Legendrian realization $\mathcal{C} = \{L_1, \ldots, L_n\}$ of the chain in Figure 1 corresponding the the continued fraction $[a_1,\ldots, a_n]$. For the rest of this section we will only consider fillings of the universally tight contact structure on $L(p,q)$, hence we drop the contact structure from the notation $\Fill(L(p,q),\xi)$. Recall from the introduction that there is a map 
\[
G_{L_s}: \Fill([a_1,\ldots, a_{s-1}])\times \Fill([a_{s+1}, \ldots, a_n])\to \Fill([a_1,\ldots, a_n]).
\]
(Here we use integers $[a_1,\ldots, a_n]$ to indicate the corresponding chain of unknots describing the lens spaces. In addition, for all the fillings to be of universally tight contact structures all the stabilizations of the Legendrian knots must be of the same type.) We will show this map is injective if we do not have $q^2\equiv 1\! \mod p$ and identify the lack of injectivity in this case. For this we need to understand $G_{L_s}$ better. We begin by observing an immediate consequence of Riemenschneider point diagrams discussed in Section~\ref{cf}.
\begin{lemma}
Given a continued fraction $p/q =[a_1,\ldots, a_n]$ with $a_i\geq 2$ and some $1< s< n$, let $p_1/q_1=[a_1,\ldots, a_{s-1}]$ and $p_2/q_2= [a_{s+1}, \ldots, a_n]$.
Let $[b_1,\ldots, b_k]$, respectively $[c_1,\ldots, c_l]$, be the continued fraction expansion of $p_1/(p_1-q_1)$, respectively $p_2/(p_2-q_2)$. Then the continued fraction expansion of $p/(p-q)$ is 
\[
[b_1,\ldots, b_{k-1}, b_k+c_1, c_2,\ldots, c_l]
\]
if $a_s=2$ and 
\[
[b_1,\ldots, b_{k-1}, b_k+1, 2, \ldots, 2, c_1+1, c_2,\ldots, c_l]
\]
otherwise, where there are $(a_s-3)$ $2$'s between $b_k+1$ and $c_1+1$. 

If $p'/q'=[a_2,\ldots, a_n]$ and $[b_1,\ldots, b_l]$ is the continued fractions expansion of $p'/(p'-q')$, then the expansion of $p/(p-q)$ is
\[
[2,\ldots, 2, b_1+1, b_2, \ldots, b_n]
\]
where there are $(b_1-2)$ $2$'s at the beginning. There is a similar formula if $a_n$ is left out. 
\hfill \qed
\end{lemma}

We now observe a fact about combining null sequences. 
\begin{lemma}\label{combnullchains}
If $(n_1,\ldots, n_k)$ and $(m_1,\ldots, m_l)$ are two null sequences then so are 
\begin{equation}\label{adding}
(n_1,\ldots, n_{k-1}, n_k+m_1, m_2, \ldots, m_l),
\end{equation}
and
\begin{equation}\label{adding2}
(n_1,\ldots, n_{k-1}, n_k+1, 2, \ldots, 2, m_1+1, m_2, \ldots, m_l).
\end{equation}
\end{lemma}
\begin{proof}
For the first claim notice that for $(m_1,\ldots, m_l)$ there is a sequence of strict blowups from $(1, 1)$ to $(m_1,\ldots, m_l)$. Now starting from $(n_1,\ldots, n_k)$ we can blowup at the last point to get $(n_1,\ldots, n_k+1, 1)$ which is the result of combining $(n_1,\ldots, n_k)$ and $(1,1)$ as in Equation~\ref{adding}. Now continue the sequence of blow up to go from $(1,1)$ to $(m_1,\ldots, m_l)$ on $(n_1,\ldots, n_k+1, 1)$ to get $(n_1,\ldots, n_{k-1}, n_k+m_1, m_2, \ldots, m_l)$. The last equation is obtained from applying the first equation to $(n_1,\ldots, n_k)$, $(1,2, \ldots, 2, 1)$, and $(m_1,\ldots, m_l)$, and since these are all null sequences so is the result. 
\end{proof}

Given $\n\in \BZ_{p_1,q_1}$ and $\m\in \BZ_{p_2,q_2}$ we call the combination in \eqref{adding} the \dfn{$2$--fusion} of $\n$ and $\m$ and denote it $\n*_{2} \m$. Similarly the combination in \eqref{adding2} is the \dfn{$a_s$--fusion} of $\n$ and $\m$ when there are $(|a_s| - 3)$ $2$'s, and denote it $\n*_{a_s} \m$. Clearly this gives a map
\[
F_{a_s}: \BZ_{p_1,q_1} \times \BZ_{p_2,q_2} \to \BZ_{p,q}.
\]
\begin{proposition}\label{altdesc}
Given fillings $W_\n\in \Fill(L(p_1,q_1))$ and $W_\m\in \Fill(L(p_2,q_2))$ then 
\[
G_{L_s}(W_\n, W_\m) = W_{F_{a_s}(\n, \m)}.
\]
\end{proposition}

For this proposition we need a preliminary result that identifies the rational unknots in $\partial W_\n$. 

\begin{lemma}\label{findrat}
Let $p/q=[a_1,\ldots, a_n]$ and $p/(p-q)= [b_1,\ldots, b_m]$. Denote the chain of unknots with framings $-a_i$, as in Figure~\ref{chain}, by $\mathcal{C}$, and the corresponding chain for the $b_i$ by $\mathcal{C}'$. The lens spaces $L(\mathcal{C})$ and $L(\mathcal{C}')$ are orientation-reversing diffeomorphic and the rational unknots $K_1$ and $K_2$ in $L(\mathcal{C})$ shown in Figure~\ref{ratunknots} and the corresponding ones $K_1'$ and $K_2'$ in $L(\mathcal{C}')$ are isotopic. Moreover, if $K''_1$ and $K''_2$ are the meridians to the $n_1$-framed and $n_m$-framed unknots in Figure~\ref{Wn}, then for any $\n\in \BZ_{p,q}$, $K_i''$ is isotopic to $K_i$ but the framing on $K_i$ is one less that the framing on $K_i''$. (Here all framings on the knots are with respect to the Seifert framing as knots in $S^3$.)
\end{lemma}
\begin{proof}
We first notice that given any $\n\in \BZ_{p,q}$ in Figure~\ref{Wn} we can blowdown all the $-1$-framed unknots to get $\mathcal{C}'$. Thus it is clear that the $K''_i$ are isotopic to the $K'_i$. So we are left to relate the $K_i''$ and the $K_i$. To this end, consider $\n=(1,2,\ldots, 2, 1)\in \BZ_{p,q}$. From above $L(\mathcal{C}')$ is clearly diffeomorphic to $\partial W_\n$ so that the rational unknots are preserved. We will now convert $\partial W_\n$ into $L(\mathcal{C})$ and see that the $K_i''$ go to the $K_i$. We start with the $-1$--framed meridians to the $n_m$-framed unknots in Figure~\ref{Wn}. Slide one $-1$--framed meridian over a second one and then that one over a third and so on until one sees the diagram at the top of Figure~\ref{ratsame}. Now blow down the $n_m$-framed unknot (recall $n_m=1$) to get the next diagram in Figure~\ref{ratsame}. Slide the $-2$--framed unknot over one of the $-1$--framed meridian to the $(n_{m-1}-1)$--framed unknot if there are any such unknots or do nothing if not. Then continue as before to slide the $-1$--framed meridians over the other such knots until one arrives at the bottom diagram in Figure~\ref{ratsame}. Continue this process until one arrives at Figure~\ref{ratfinalfig}. 
\begin{figure}[htb]{\tiny
\begin{overpic}
{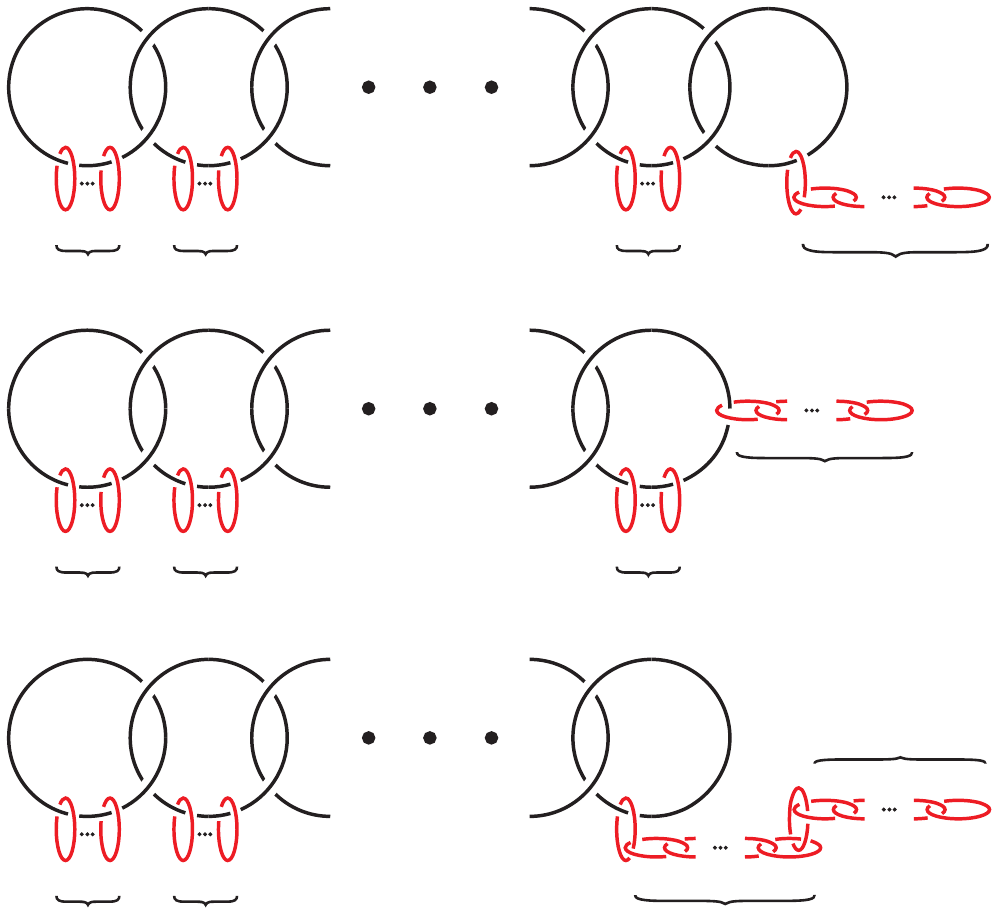}
\put(14, -2){$b_1-n_1$}
\put(48, -2){$b_2-n_2$}
\put(179,-2){$b_{m-1}-n_{m-1}-1$}
\put(10, 9){$-1 \,-1$}
\put(44, 9){$-1 \,-1$}
\put(165, 20){$-1$}
\put(180, 9){$-2 \quad \quad \quad \quad -2$}
\put(230, 38){$-2 \quad \quad \quad \quad-2$}
\put(237, 51){$b_m-n_m-1$}
\put(216, 30){$-3$}
\put(4, 75){$1$}
\put(43, 75){$2$}
\put(197, 75){$1$}
\put(4, 170){$1$}
\put(43, 170){$2$}
\put(198, 170){$1$}
\put(14, 92){$b_1-n_1$}
\put(48, 92){$b_2-n_2$}
\put(169,92){$b_{m-1}-n_{m-1}$}
\put(10, 104){$-1 \,-1$}
\put(44, 104){$-1 \,-1$}
\put(172, 104){$-1 \,-1$}
\put(211, 137){$-2 \quad \quad \quad \quad-2$}
\put(225, 126){$b_m-n_m$}
\put(4, 264){$1$}
\put(43, 264){$2$}
\put(198, 265){$2$}
\put(230, 264){$1$}
\put(14, 186){$b_1-n_1$}
\put(48, 186){$b_2-n_2$}
\put(169, 186){$b_{m-1}-n_{m-1}$}
\put(236, 186){$b_{m}-n_{m}-1$}
\put(10, 198){$-1 \,-1$}
\put(44, 198){$-1 \,-1$}
\put(172, 198){$-1 \,-1$}
\put(231, 198){$-2 \quad \quad \quad \quad-2$}
\put(216, 208){$-1$}
\end{overpic}}
\caption{Converting $\mathcal{C}'$ to $\mathcal{C}$.}
\label{ratsame}
\end{figure}
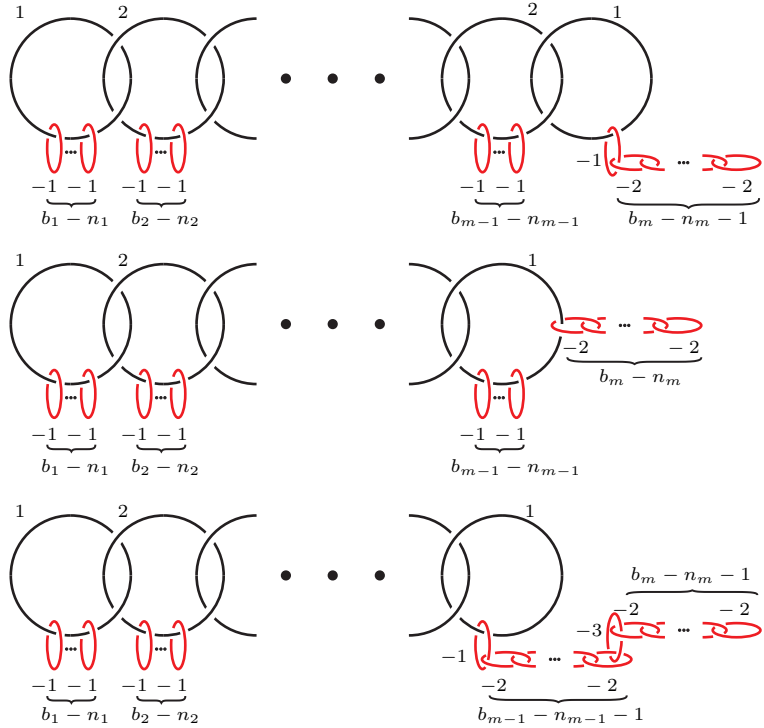
To that you get the framings claimed in Figure~\ref{ratfinalfig} notice that if the chain $(b_1,\ldots, b_m)$ is of the form $(2^{d_1}, c_1+3, 2^{d_2}, c_2+3, \ldots, 2^{d_k}, c_k+3, 2^{d_{k+1}})$ for $c_i, d_i\geq 0$, where $2^l$ means $2$ repeated $l$ times, then the above process gives a chain with framings $(-d_1-2, (-2)^{c_1}, -d_2-3, \ldots, -d_k-3, (-2)^{c_k}, -d_{k+1}-2)$ which is precisely $(-a_1,\ldots, -a_n)$, where $(a_1,\ldots,a_n)$ is obtained from $(b_1,\ldots, b_m)$ according to the Riemenschneider point diagram formula from Section~\ref{cf}.
\begin{figure}[htb]{\tiny
\begin{overpic}
{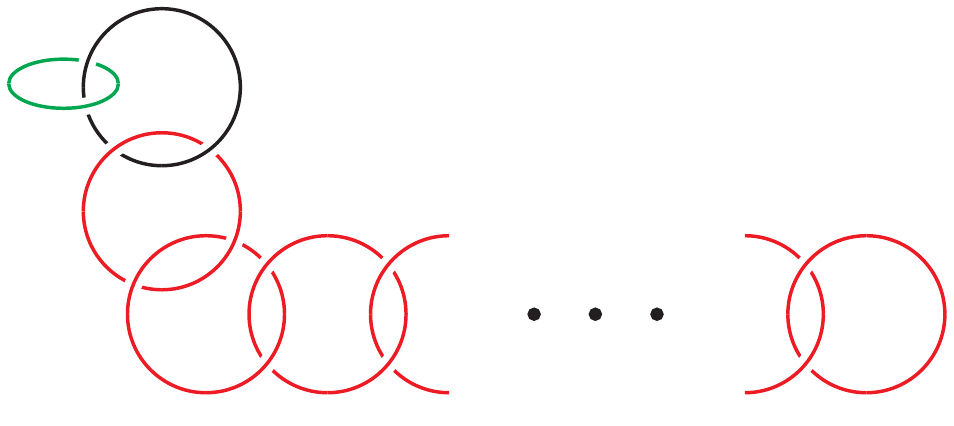}
\put(0, 109){$K_1''$}
\put(58, 121){$0$}
\put(11, 59){$-1$}
\put(50, 1){$-a_1$}
\put(86, 1){$-a_2$}
\put(242, 1){$-a_n$}
\end{overpic}}
\caption{Final step converting $\mathcal{C}'$ to $\mathcal{C}$.}
\label{ratfinalfig}
\end{figure}
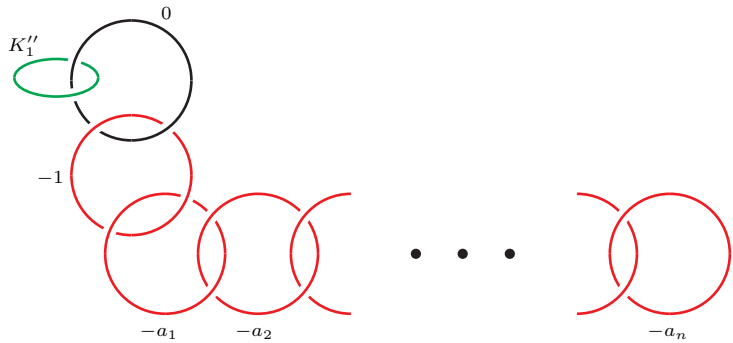
The curve $K_1''$ in Figure~\ref{ratfinalfig} can be slid over the  $-1$-framed unknots and then thinking of the $0$-framed unknot as a $1$--handle, it may be cancelled from the picture using the $-1$-framed unknot. The resulting picture is precisely the chain in Figure~\ref{ratunknots} with $K_1''$ becoming $K_1$ with framing one less. 

To see the claim for $K_2''$ one does the above argument but starting from the left most $-1$-framed unknot if Figure~\ref{Wn} and working to the right. 
\end{proof}

\begin{proof}[Proof of Proposition~\ref{altdesc}]
Let $p_1/q_1=[a_1,\ldots, a_{s-1}]$ and $p_2/q_2= [a_{s+1}, \ldots, a_n]$. Suppose $W_\n$ is a filling of $L(p_1,q_1)$ and $W_{\m}$ is a filling of $L(p_2,q_2)$. Recall that $W_\n$ is constructed by attaching a $1$--handle to the $4$--ball and then $2$--handles to a framed link $L_\n$ in $\partial (S^{1} \times D^{3})$. This is indicated on the upper left of Figure~\ref{fuse}. We obtain $L_\n$ by taking the red link in Figure~\ref{Wn} and blowing down the $n_i$--framed curves until only the rightmost one is left and it has framing $0$. Now replace this with $0$--framed unknot with a $1$--handle. We can similarly form $W_\m$ but this time we blowdown the $m_i$--framed knots leaving the leftmost one.  From this construction we see that the knots $K_2''$ in $\partial W_\n$ and $K_1''$ in $\partial W_\m$ are simply meridional circles to the $1$--handles. 

To construct $G_{L_s}(W_\n, W_\m)$ we attach a $1$--handle to $W_\n\cup W_\m$, obtaining the boundary connected sum, and then a $2$--handle to the connected sum of the rational unknots $K_2$ in $L(p_1,q_1) = \partial W_\n$ and $K_1$ in $L(p_2,q_2)=\partial W_\m$. Recall from Lemma~\ref{findrat} the $K_i$ are exactly the same knots as the $K_i''$ but the framing on the former is one less than that on the latter. So the resulting filling of $L(p,q)$ is shown at the top of Figure~\ref{fuse}; the surgery coefficient on $K_2''\# K_1''$ will be determined later.
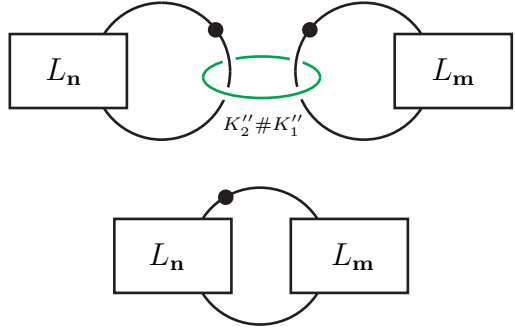
\begin{figure}[htb]{\tiny
\begin{overpic}
{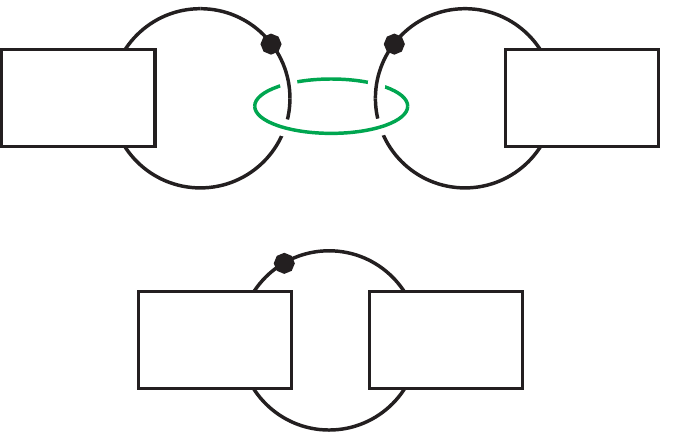}
\put(14, 93){\large$L_{\n}$}
\put(160, 93){\large$L_{\m}$}
\put(81, 74){$K''_2\#K''_1$}
\put(53, 23){\large$L_{\n}$}
\put(122, 23){\large$L_{\m}$}
\end{overpic}}
\caption{The results of attaching a round $1$--handle to $K''_2$ in $\partial W_\n$ and $K''_1$ in $\partial W_\m$.}
\label{fuse}
\end{figure}

Suppose $a_s=2$. Then $G(W_\n,W_\m)$ attaches a round $1$--handle to a Legendrian realization of $K_2''$ in $\partial W_\n$ and a realization of $K''_1$ in $\partial W_\m$. To get a filling of $L(p,q)$ we need to attach the $2$--handle portion of the round $1$--handle to $K_1\#K_2$ with framing $-2$. According to Lemma~\ref{findrat} this would be the same as attaching the $2$--handle to $K_2''\#K_1''$ with framing $0$. Now use the $0$--framed unknot to slide $L_\n$ off of the leftmost $1$--handle in Figure~\ref{fuse} and onto the rightmost handle. Then cancel the leftmost $1$--handle with the $0$--framed $2$--handle. This is shown on the bottom of Figure~\ref{fuse}.

Notice that this manifold is precisely the the manifold $W_{F_{2}(\n,\m)}$ associated to the $2$--fusion of $\n$ and $\m$ if one blows down the the curves in Figure~\ref{Wn} associated to $\n*_2\m$ leaving the one labeled $n_k+m_1$ as the $0$--framed unknot. Thus we have established the proposition when $a_2=2$. 
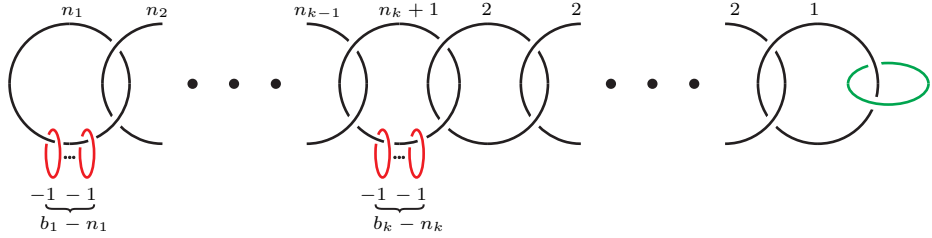
\begin{figure}[htb]{\tiny
\begin{overpic}
{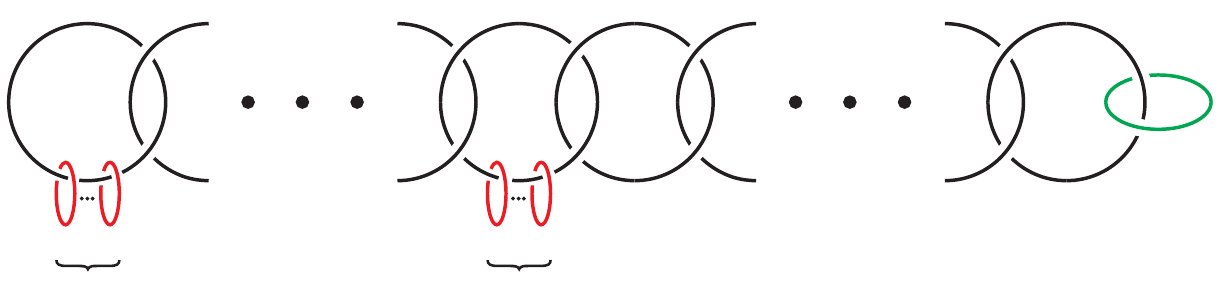}
\put(14, -2){$b_1-n_1$}
\put(140,-2){$b_k-n_k$}
\put(10, 9){$-1 \,-1$}
\put(135, 9){$-1 \,-1$}
\put(22, 79){$n_1$}
\put(54, 79){$n_2$}
\put(110, 79){$n_{k-1}$}
\put(142, 79){$n_k+1$}
\put(181, 79){$2$}
\put(215, 79){$2$}
\put(275, 79){$2$}
\put(305, 79){$1$}
\end{overpic}}
\caption{Another picture of $W_\n$. There are $a_s-3$ unknots with framing $2$ on the right.}
\label{fuselarger}
\end{figure}

Now we consider the case for $a_s>2$. For this notice that Figure~\ref{fuselarger} is a diagram that gives $W_\n$. One simply blows down the $1$--framed unknot on the right and continues to blow down the chain the unknots on the right until one is left with the original chain defining $W_\n$, we then blowdown $n_i$--framed unknots until we are left with the unknot that was originally framed $n_k+1$. It will now have a $0$--framing and can be replaced with a $1$--handle to get $W_\n$. Now notice the green curve is the rational unknot $K_2$, but the framing on it in $\partial W_\n$ is $a_s-1$ less than the framing on $K_2$. Now arguing as above we clearly see that $G_{L_s}(W_\n,W_\m)$ is the same as the manifold $W_{F_{a_s}(\n,\m)}$ obtained from the $a_s$--fusion of $\n$ and $\m$. 
\end{proof}

\begin{proposition}\label{combinestrings}
Given a continued fractions $p/q =[a_1,\ldots, a_n]$ with $a_i\geq 2$ and some $1< s< n$, let $p_1/q_1=[a_1,\ldots, a_{s-1}]$ and $p_2/q_2= [a_{s+1}, \ldots, a_n]$. The map 
\[
G_{L_s}: \Fill(L(p_1,q_1),  \xi_{ut})\times \Fill(L(p_2,q_2), \xi_{ut})\to \Fill(L(p,q), \xi_{ut})
\]
is injective unless $q^2\equiv 1 \!\mod p$. When $q^2\equiv 1\! \mod p$, the map will identify $(W_\n,W_\m)$ and $(W_{\n'}, W_{\m'})$ if and only if $\n'*_{a_s} \m' = \overline{\n*_{a_s}\m}$. (Note that we must choose the universally tight contact structures on $L(p_1,q_1)$ and $L(p_2,q_2)$ so that the image of the map is a universally tight contact structure.) 
\end{proposition}

\begin{proof}
Let $p_1/q_1=[a_1,\ldots, a_{s-1}]$ and $p_2/q_2= [a_{s+1}, \ldots, a_n]$. Suppose $W_\n$ and $W_{\n'}$ are fillings of $L(p_1,q_1)$ and $W_{\m}$ and $W_{\m'}$ are fillings of $L(p_2,q_2)$. From Lisca's classification, Theorem~\ref{lisca}, we know $W_{\n*_{a_s}\m}$ is the same as $W_{\n'*_{a_s}\m'}$ if and only if $\n*_{a_s}\m=\n'*_{a_s}\m'$ or $q^2\equiv 1\! \mod p$ and $\n'*_{a_s}\m'=\n*_{a_s}\m$ or $\overline{\n*_{a_s}\m}$. Thus the proposition follows from Proposition~\ref{altdesc}.
\end{proof}

\begin{remark}
We note that the map $G_{L_s}$ can also be surjective, but in general does not have to be. For example, notice that any filling in the image of the map must have Euler characteristic at least $2$ (since the boundary connected sum of  $W_\n$ and $W_{\n'}$ has Euler characteristic at least $1$ and we then add a $2$--handle). If we consider $-16/3=[-6,-2,-2]$, then $\Fill(L(16,3),\xi_{ut})$ contains two elements, the plumbing of three disk bundles over spheres and a rational homology ball. From our above observation we see that the rational homology ball filling will not be in the image of any $G_{L_s}$. 
\end{remark}

\subsection{Decomposing fillings along tori}\label{menkedecompose}
Let $X$ be a $4$--manifold with boundary. If $K_0$ and $K_1$ are two disjoint oriented knots in $\partial X$ then a round $1$--handle is attached to $X$ along $K_0\cup K_1$ by gluing $S^1\times ([0,1]\times D^2)$ to $X$ by identifying $S^1\times \{i\}\times D^2$ with a neighborhood of $K_i$, for $i=0,1$. Such a gluing is determined by a framing on $K_0$ and $K_1$. Denote the resulting manifold by $X'$. We notice that $\partial X'$ is obtained from $\partial X$ by removing neighborhoods of the $K_i$ and gluing the resulting boundary components together (so that the meridian on one of the tori goes to the meridian on the other, and the longitude determined by the framing on one of the tori goes to the longitude on the other). Conversely notice that there is a natural torus $T$ in $\partial X'$ so that $\partial X$ can be recovered from $\partial X'$ by removing a neighborhood of $T$ and gluing in two solid tori. (Not every gluing of solid tori will yield $\partial X$, but there will be one that does.) 

Notice that attaching a round $1$--handle can also be described by attaching a $1$--handle with attaching sphere a point on $K_0$ and a point on $K_1$ followed by attaching a $2$--handle to $K_0\# K_1$ (notice that after attaching the $1$--handle the parts of $K_i$ outside the attaching region of the $1$--handle can be joined to form the connected sum of the knots). 

If $X$ is a symplectic manifold with convex boundary and the $K_i$ are Legendrian knots in $\partial X$ then the $1$--handle and $2$--handle mentioned above can be attached as Weinstein handles and thus the resulting manifold $X'$ is also a symplectic manifold with convex boundary. 

Suppose $(M,\xi)$ is a contact manifold.  An embedded convex torus $T$ in $M$ is called a \dfn{mixed torus} if $T$ has a neighborhood $N=[-1,1]\times T^2$ with convex boundary such that $\{0\}\times T^2$ is $T$ and $[-1,0]\times T^2$ and $[0,1]\times T^2$ are basic slices of opposite sign. Let the slope of the dividing curves on $\{t\}\times T^2$ be denoted by $s_t$, for $t=-1,0, 1$. There will be an edge in the Farey graph from $s_{-1}$ to $s_0$ and one from $s_0$ to $s_{1}$. Let $E$ be the set of slopes in the interval $(s_1, s_{-1})$ with an edge to $s_0$. 
\begin{theorem}[Menke, 2018 \cite{menke18pre}]\label{menke}
If $(X,\omega)$ is a symplectic manifold with convex boundary and $T$ is a mixed torus in $\partial X$ then there is a symplectic manifold $(X',\omega')$ and Legendrian knots $K_0$ and $K_1$ in $\partial X'$ such that $(X,\omega)$ is recovered from $(X',\omega')$ by attaching a round $1$--handle along $K_0\cup K_1$. 

Moreover, $\partial X'$ is obtained from $\partial X$ by cutting along $T$ and gluing in two solid tori both with meridional slope $e$ for some $e\in E$ and extending the contact structure over these tori by the unique tight structure on them. These solid tori are neighborhoods of the Legendrian knots $K_i$. 
\end{theorem}
A contact structure $\xi$ on a lens space $L(p,q)$ is determined by a minimal path in the Farey graph from $0$ to $-p/q$ with all edges (but the first and last) decorated by a sign. When two adjacent edges have a different sign, the torus along which they come together is a mixed torus. So a mixed torus can occur within a continued fraction block (and the Legendrian knot corresponding to this continued fraction block will be stabilized both positively and negatively) or, if there are no mixed tori within a continued fraction block, then there can be mixed tori as one transitions from a continued fraction block to another along an inconsistent sub-chain.

\section{Decomposing fillings of virtually overtwisted lens spaces}\label{provelist}

We now analyze the result of applying Menke's theorem, Theorem~\ref{menke}, to fillings of lens spaces. We begin with a mixed torus within a continued fraction block. 
\begin{theorem}\label{firstglue}
Let $\xi_{\mathcal{C}}$ be a contact structure on $L(p,q)$ described by Legendrian surgery on the chain $\mathcal{C}=\{L_1,\ldots, L_n\}$ of Legendrian unknots. With the notation from Construction~\ref{constructionofG}, if $L_k$ is a Legendrian knot that has been stabilized both positively and negatively, then the map
\[
G_{L_k}:\Fill(L(\mathcal{C}_s))\times \Fill(L(\mathcal{C}_e)) \to \Fill(L(\mathcal{C})).
\]
is surjective.
\end{theorem}
This is essentially Theorem~1.3 in \cite{menke18pre} adapted to the notation used in this paper. 
\begin{proof}
Because $L_k$ has been stabilized both positively and negatively, we know in the continued fraction block associated to $L_k$ there is a mixed torus $T$. 

Given a Stein filling $X$ of $L(\mathcal{C})$. Theorem~\ref{menke} then says there is a filling $X'$ of the result of cutting $L(\mathcal{C})$ along $T$ and gluing in two solid tori. To determine the possible slope for these tori we notice that the basic slices adjacent to $T$ are two jumps in the continued fraction block, say from $v_j$ to $v_{j+1}$ and then to $v_{j+2}$, see Figure~\ref{CFB}. As such the only vertex in $(v_{j+2}, v_j)$ with an edge to $v_{j+1}$ is $m$ in Figure~\ref{CFB}. Thus this is the meridional slope for the two solid tori. That is $X'$ is a filling of the lens spaces with upper meridian $0$ and lower meridian $m=[-a_1, \ldots, -a_{k-1}]$ and the lens space with upper meridian $m$ and lower meridian $[-a_1,\ldots, -a_n]$, see Section~\ref{breakingup} for terminology. Lemma~\ref{bucfb} then says that topologically these lens spaces are $L(\mathcal{C}_s)$ and $L(\mathcal{C}_e)$, respectively. Moreover, the path in the Farey graph describing $L(\mathcal{C}_s)$ is a sub-path of the one describing $L(\mathcal{C})$ with one extra edge added (that is the jump from $[-a_1,\ldots, -a_{k-1}+1]$ to $m=[-a_1, \ldots, -a_{k-1}]$). Thus the contact structure on $L(\mathcal{C}_s)$ is the one given by Legendrian surgery on the sub-chain $\mathcal{C}_s$ and similarly for $L(\mathcal{C}_e)$. (See Remark~\ref{pathsandknots} for the relation between the paths in the Farey graph and Legendrian surgery on the chain.)

Since there cannot be a symplectic filling of a lens space with disconnected boundary (by \cite{Etnyre04a} fillings of contact structures supported by planar open books must have a single boundary component, and tight structures on lens spaces are supported by planar open books \cite{Schoenenberger05}) we see that $X'$ is a disjoint union of two symplectic manifolds $X_{\mathcal{C}_s}$ and $X_{\mathcal{C}_e}$. The filling $X$ is recovered from $X_{\mathcal{C}_s}\cup X_{\mathcal{C}_e}$ by attaching a round $1$--handle. Thus if either $X_{\mathcal{C}_s}$ or $X_{\mathcal{C}_e}$ were not minimal, then $X$ would not be either. But since it is, we must have both $X_{\mathcal{C}_s}$ and $X_{\mathcal{C}_e}$ be minimal symplectic fillings (and by \cite{Wendl10}, Stein fillings). When attaching a round $1$--handle to $X_{\mathcal{C}_s}\cup X_{\mathcal{C}_e}$ the associated $2$--handle is attached to $U_k$ and to get from $L(\mathcal{C}_s)\# L(\mathcal{C}_e)$ to $L(\mathcal{C})$ there is a unique possible Legendrian realization of $U_k$ to which a Stein $2$--handle can be attached by Corollary~\ref{legsimpleratunknots}. That is $X$ is $G_{L_k}(X_{\mathcal{C}_s}, X_{\mathcal{C}_e})$ and we see that $G_{L_k}$ is surjective. 
\end{proof}

We now consider mixed tori that are not contained in a continued fraction block. 
\begin{theorem}\label{mixedglue}
Let $\xi_{\mathcal{C}}$ be a contact structure on $L(p,q)$ described by Legendrian surgery on the chain $\mathcal{C}=\{L_1,\ldots, L_n\}$ of Legendrian unknots. Assume that $\mathcal{C}$ is nicely stabilized and $\mathcal{C}'=\{L_k,\ldots, L_{k+l}\}$ is an inconsistent sub-chain, then 
\[
\Fill(L(\mathcal{C})) = \bigcup_{i=k}^{k+l} Image (G_{L_i}).
\]
\end{theorem}
See Section~\ref{classsect} for the terminology used in this theorem. 
\begin{proof}
The argument is exactly the same as in the proof of Theorem~\ref{firstglue} except the mixed torus $T$ is at the juncture between the continued fraction block associated to $L_k$ and the one associated to $L_{k+l}$ (recall the other $L_i$ in $\mathcal{C}'$ all have Thurston-Bennequin invariant $-1$ and do not correspond to a continued fraction block but indicate how far down $L_{k+l}$'s continued fraction block is from $L_k$'s, see Section~\ref{pathinfg}). Thus given $X$ a filling of $L(\mathcal{C})$, Menke's theorem  (Theorem~\ref{menke}) gives a symplectic manifold $X'$ with $\partial X'$ obtained from $\partial X$ by cutting along $T$ and gluing in solid tori with meridional slopes identified in Observation~\ref{exceptional}. The lens spaces thus obtained are then identified in Lemma~\ref{bucfb2}.
\end{proof}
\noindent
This completes the proof of Theorem~\ref{list}, as follows. 
\begin{proof}[Proof of Theorem~\ref{list}]
Given a lens space, represent it by a chain of unknots $\mathcal{C}$, and call the lens space $L(\mathcal{C})$. Now consider $\mathcal{D}$, the components that have been stabilised both positively and negatively. By Theorem~\ref{firstglue}, we get a surjective map from the fillings of $\mathcal{C} - \mathcal{D}$ to the fillings of $\mathcal{C}$ (notice that the composition of $G_{\{L\}}$ with $L\in\mathcal{D}$ has the same image as $G_{\mathcal{D}}$). Then, Theorem~\ref{mixedglue} tells us that all the fillings of the sub-chains in $\mathcal{C}-\mathcal{D}$ are given by the gluing maps coming from a maximal collection $\mathcal{M}$ in $\mathcal{S}$ (see the introduction, before Example \ref{maximal}, for notation). 
\end{proof}

\section{Counting fillings of lens spaces}
In this section we will start by studying the gluing maps $G_{\mathcal{S}}$ that appear in our main theorem, Theorem~\ref{list}, and see the extent to which we need all of the ones used in the theorem. Then in Sections~\ref{general} and~\ref{specific} we establish all the corollaries of our main result discussed in the introduction. 

\subsection{Injectivity of the gluing map}\label{injectsect}
We begin by proving Theorem~\ref{inject} concerning the injectivity of the gluing maps in our main theorem. 
\begin{proof}[Proof of Theorem~\ref{inject}]
Proposition~\ref{combinestrings} says that the gluing map when one unknot is removed from a chain is injective on the fillings of the universally tight contact structure on $L(p,q)$ (up to the symmetries described in the proposition if $q^2\equiv 1 \!\!\mod p$). Since Theorem~\ref{maxcollection} (proven below) says that the fillings of any other contact structure are a subset of the fillings of the universally tight contact structure, we see that the gluing map for any contact structure is a restriction of the one for the universally tight contact structures and so will be at least as injective. 
\end{proof}

\begin{remark}
We note Menke \cite{menke18pre} showed that if $(M',\xi')$ is obtained from $(M,\xi)$ by Legendrian surgery on a Legendrian $L$ that has been stabilized positively and negatively, then any filling of $(M',\xi')$ comes from a filling of $(M,\xi)$ by attaching a Stein $2$--handle to $L$. Thus the gluing map corresponding to such an $L$ in a chain $\mathcal{C}$ is also surjective. Though we see that this is not true in general, see Example~\ref{firstex} below.   
\end{remark}

We now observe that in general there is no single gluing map in Theorem~\ref{list} that will produce all of the fillings of a lens space.
\begin{example}\label{firstex}
Consider the lens space in Figure~\ref{needall}. 
\begin{figure}[htb]{\tiny
\begin{overpic}
{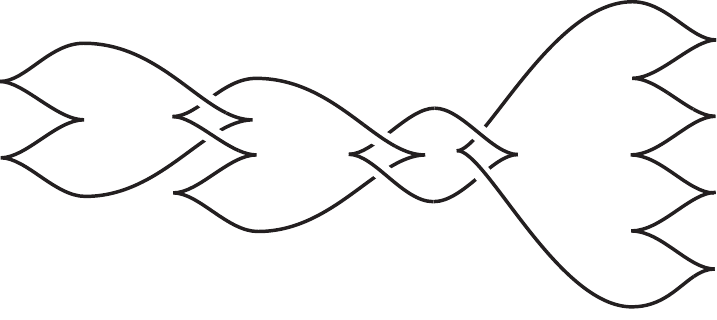}
\end{overpic}}
\caption{A  contact structure on $L(57,22)$ with three different fillings.}
\label{needall}
\end{figure}
There are no components in the chain that have been stabilized both positively and negatively. So by Theorem~\ref{list} we only need to consider the inconsistent sub-chain $\{L_2,L_3,L_4\}$ (we label the components from left to right). So Theorem~\ref{list} says that the fillings of this contact structure are in the image of the gluing maps $G_{L_2}$, $G_{L_3}$, and $G_{L_4}$. One may easily check that $Image(G_{L_3})=Image(G_{L_4})$ consists of two fillings corresponding to $(1,2,2,2,2,1)$ and $(2,1,3,2,2,1)$ in $\BZ_{57,22}$ while $Image(G_{L_2})$ consists of fillings corresponding to $(1,2,2,2,2,1)$ and $(1,2,4,1,2,2)$.
\end{example}

\begin{example}
In Figure~\ref{needall2} we see a Legendrian chain giving a contact structure on $L(222,61)$ with inconsistent sub-chain $\{L_2, L_3, L_4, L_5, L_6\}$. 
\begin{figure}[htb]{\tiny
\begin{overpic}
{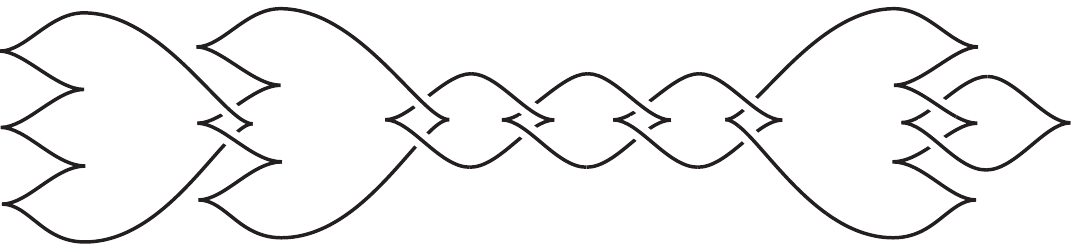}
\end{overpic}}
\caption{A  Legendrian chain for the lens space $L(222,61)$.}
\label{needall2}
\end{figure}
Applying Theorem~\ref{list} (together with Lemma~\ref{combnullchains}) we find the fillings of this contact structure as follows. 
The image of $G_{L_2}$ consists of the fillings corresponding to $(1,2,2,2,2,2,1), (1,2,2,2,3,1,2), (2,1,3,2,2,2,1),$ and $(2,1,3,2,3,1,2)$. The images of $G_{L_3}$ and $G_{L_4}$ agree and contain the fillings in the image of $G_{L_2}$ and $(1,2,3,1,3,2,1)$ and $(1,2,3,1,4,1,2)$. The image of $G_{L_5}$ contains the fillings in the image of $G_{L_3}$ and also the ones corresponding to $(2,2,2,1,5,2,1)$ and $(2,2,2,1,6,1,2)$. Finally, the image of $G_{L_6}$ contains the fillings corresponding to  $(1,2,2,2,2,2,1), (2,1,3,2,2,2,1), (1,2,3,1,3,2,1),$ and $(2,2,2,1,5,2,1)$. 
\end{example}

From the first example we see that  in Theorem~\ref{list} there is no single gluing map whose image will contain all of the fillings of a contact structure and that there is no choice for ``best" gluing map to use (that is whether using the end points of an inconsistent sub-chain or an interior part of the chain will give most or all of the fillings). From the second example we see that it is possible to get all of the fillings from a single gluing map but the filling map comes from one of the internal knots in the inconsistent sub-chain. 

\subsection{General results on fillings of lens spaces}\label{general}
We now see that the smooth manifolds that minimally symplectically fill any contact structure on $L(p,q)$ also fill the universally tight contact structure. 
\begin{proof}[Proof of Theorem~\ref{maxcollection}]
Let $\mathcal{C}$ be any chain of Legendrian unknots describing some contact structure $\xi$ on $L(p,q)$. Let $\mathcal{C}'$ be a chain that describes a universally tight contact structure on $L(p,q)$. Notice that $\mathcal{C}$ and $\mathcal{C}'$ are smoothly the same framed link, just some of the components are stabilized in a different way. Now let $\mathcal{D}$ be the components of $\mathcal{C}$ that have been stabilized both positively and negatively and let $\mathcal{S}$ be the unknots in $\mathcal{C}-\mathcal{D}$ that are in inconsistent sub-chains (see Section~\ref{classsect} for notation). Let $\mathcal{D}'$ be the subset of $\mathcal{C}'$ that corresponds to the components of $\mathcal{D}$. Similarly let $\mathcal{S}'$ be the subset of $\mathcal{C}'-\mathcal{D}'$ that corresponds to $\mathcal{S}$. Notice that for any maximal collection $\mathcal{M}$ in $\mathcal{S}$ and the corresponding collection $\mathcal{M}'$ in $\mathcal{S}'$, the chains in $\mathcal{C}-(\mathcal{D}\cup\mathcal{M})$ and $\mathcal{C}'-(\mathcal{D}'\cup\mathcal{M}')$ correspond to the same universally tight contact structures and so have the same elements. Thus $Image(G_{\mathcal{D}\cup\mathcal{M}})$ is a subset of $Image(G_{\mathcal{D}'\cup\mathcal{M}'})$. (Notice that they could be the same, but they do not have to be. For example, if $\mathcal{C}=\{L_1,L_2,L_3\}$ with $L_1$ stabilized positively and negatively while $L_2\cup L_3$ is the Hopf link with $tb(L_2)=tb(L_3)=-2$ and $r(L_2)=-r(L_3)=1$, then $Image(G_{\mathcal{D}'\cup\mathcal{M}'})$ contains two elements while $Image(G_{\mathcal{D}\cup\mathcal{M}})$ contains one.) Moreover the union of $Image(G_{\mathcal{D}'\cup\mathcal{M}'})$ over all maximal collections $\mathcal{M}'$ is contained in $\Fill(\mathcal{C}')$ we see that $\Fill(\mathcal{C})$ is contained in $\Fill(\mathcal{C}')$. (Recall, $\Fill(\mathcal{C})$ is describing the smooth types of the symplectic fillings of $L(\mathcal{C})$ and not the specific symplectic structures on the smooth manifolds.)
\end{proof}

We now prove Theorem~\ref{cor2} that says the difference between the number of fillings of a universally tight contact structure and a virtually overtwisted contact structure can be arbitrarily large. 
\begin{proof}[Proof of Theorem~\ref{cor2}]
From the proof of Corollary~1.2 in \cite{Lisca08} one can easily see that the fillings of the universally tight contact structure on $[-2, (-3)^n, (-2)^n]$ has at least $n-1$ fillings (recall $m^n$ in a continued fraction means repeat $m$, $n$ times). Now if we consider the chain of Legendrian unknots with Thurston-Bennequin invariant $-1$, $-2$ ($n${ times}), $-1$  ($n${ times}) where the $tb=-2$ unknots are alternatively stabilized positively and negatively gives an overtwisted contact structure $\xi$. From our main theorem, Theorem~\ref{list}, we see that it has a unique Stein filling. 
\end{proof}
Virtually overtwisted contact structures on lens space can have arbitrarily many fillings.
\begin{proof}[Proof of Theorem~\ref{cor3}]
Given any integer $k$, Lisca gives lens spaces with at least $k$ Stein fillings \cite{Lisca08}. Doing Legendrian surgery on a rational unknot in this lens space that has been stabilized both positively and negatively yields a virtually overtwisted contact structure that has at least $k$ Stein fillings by Theorem~\ref{list}.
\end{proof}
We now see that the only filling of a lens space with the same second Betti number as the plumbing that fills the lens space is indeed the plumbing.
\begin{proof}[Proof of Theorem~\ref{cor4}]
Suppose $p/(p-q)=[b_1, \ldots, b_k]$ where $b_i\geq 2$. The fillings of the universally tight contact structure are given by $\BZ_{p,q}$. The proof of Lemma~\ref{findrat} shows that the filling corresponding to $(1,2^{k-2}, 1)$ corresponds to the filling given by the plumbing diagram in Figure~\ref{chain}. 

We claim that any other filling, if it exists, must have fewer $2$--handles (and hence smaller second Betti number). Given $\n\in \BZ_{p,q}$ where $\n=(n_1,\ldots, n_k)$  the number of $2$--handles in $W_\n$ is 
\[
\sum_{i=1}^k b_i-n_i=\sum_{i=1}^k b_i - \sum_{i=1}^k n_i.
\]
and so the second Betti number is one less that this (since $W_\n$ is made with one $1$--handle that is (rationally) cancelled by one of the $2$--handles). The only way to get a null collection $(n_1,\ldots, n_k)$ of length $k$ is to (strictly) blowup such a collection of length $k-1$. A blowup that is on the interior of the chain, adds $3$ to $\sum n_i$, but when done at the far right it only adds $2$. The chain $(1, 2^{k-2}, 1)$ is obtained from $(1,1)$ by always blowing up on the far right and so $\sum n_i$ in this case is $2k-2$, but in all other cases is strictly larger than this. Thus all other fillings must have strictly smaller second Betti number than the plumbing. 
\end{proof}

We now prove there is a lower bound on the Euler characteristic of a filling of a contact structure on $L(p,q)$. 
\begin{proof}[Proof of Theorem~\ref{chilower}]
Let $\mathcal{C}$ be a chain defining a contact structure $\xi$ on the lens space $L(p,q)$ and let $\mathcal{D}$ be the unknots in $\mathcal{C}$ that are stabilized both positively and negatively. Let $\mathcal{S}$ be the unknots in $\mathcal{C}-\mathcal{D}$ that are in inconsistent sub-chains, and $\mathcal{M}$ will denote some maximal subset of $\mathcal{S}$ (we are using the notation from the discussion just before Example~\ref{maximal}). Now from Theorem~\ref{list} we know that any filling $X$ of $\xi$ will be in the image of $G_{\mathcal{D}\cup \mathcal{M}}$ for some choice of $\mathcal{M}$. That is, let $\mathcal{C}_1,\ldots,\mathcal{C}_n$ be the sub-chains of $\mathcal{C}$ left after removing $\mathcal{D}\cup \mathcal{M}$ and let $X_i$ be a Stein filling of $L(\mathcal{C}_i)$. Then we construct the filling of $\xi$ by taking the boundary connected sum of the $X_i$ and then attaching a $2$--handle for each knot in $\mathcal{D}\cup \mathcal{M}$. Each of the $X_i$ must have Euler characteristic at least $1$, so the Euler characteristic of the result of the boundary connected sum of the $X_i$ will be at least $1$ as well. We now add $|\mathcal{D}\cup \mathcal{M}|$ $2$--handles. So the Euler characteristic of $X$ is at least $1+|\mathcal{D}\cup \mathcal{M}|$. 

Recall we set $k=|\mathcal{D}|$ and $l$ is the number of inconsistent sub-chains in $\mathcal{C}-\mathcal{D}$. It is easy to see that any maximal $\mathcal{M}$ has at least $\lceil l/2\rceil$ elements. Thus from above we see that any filling of $\xi$ must have Euler characteristic bounded below by $1+k+\lceil l/2\rceil$.
\end{proof}

We now identify the rational homology balls that lens spaces bound. 
\begin{proof}[Proof of Lemma~\ref{ratballlemma}]
In \cite{Lisca08}, Lisca shows that the universally tight contact structure on $L(p,q)$ has a rational homology ball Stein filling if and only if $(p,q)=(m^2, mh-1)$ for some $h$ relatively prime to $m$. In \cite{EtnyreTosun20pre, GollaStarkston19pre} and Theorem~\ref{chilower} it was shown that no virtually overtwisted contact structure on a lens space bounds a rational homology ball, while the same result under extra hypothesis was also given in \cite{Fossati2019pre2}. So we are left to explicitly give a handle presentation of the filling of $L(m^2, mh-1)$. To this end notice that in Figure~\ref{ratballs} the Legendrian knot sits on a Heegaard torus for $S^1\times S^2$ as a $(n,-m)$--curve and the contact framing agrees with the framing induced by the Heegaard torus. So the framing corresponding to Legendrian surgery will be one less that this. In Figure~\ref{ratballs} this will be the $-nm-1$ framing (indeed, drawing Figure~\ref{ratballs} using dotted circle notation for the $1$--handle we see the $2$--handle is attached to the $(n,-m)$--torus knot and it is well-known that the difference between the torus framing and Seifert framing of a $(n,-m)$--torus knot is $-nm$).
Thus smoothly the result of Legendrian surgery is the same as cutting $S^1\times S^2$ along the Heegaard torus and regluing by a Dehn twist on the $(n,-m)$--curve. So we see that the meridian of one of the Heegaard tori is glued to the $(1-mn,m^2)$--curve on the other. That is the same as doing $-m^2/(mn-1)$ surgery on the unknot in $S^3$. Thus we get the lens space $L(m^2,mn-1)$. So the Stein diagrams in Figure~\ref{ratballs} fill the claimed lens spaces. And as such they must come from Lisca's construction. 

We now turn to the uniqueness of rational homology ball fillings of lens spaces. To this end we notice that this follows from the following lemma because Lisca \cite{Lisca08} showed that all Stein fillings are described up to diffeomorphism by compatible null sequences. We note that this was previously observed in \cite{GollaStarkston19pre}.
\end{proof}
\begin{lemma}\label{uniqueratball}
If $L(p,q)$ bounds a Stein rational homology ball, then there exists a unique null sequence $(n_{1}, \ldots , n_{k})$ such that $n_{i} = 1$, no other $n_j$ is $1$, and $p/(p-q) = [n_{1}, \ldots , n_{i-1}, 2 , n_{i+1} , \ldots, n_{k}]$.
\end{lemma}

\begin{proof}
A null sequence $(n_{1}, \ldots , n_{k})$ with only one $n_{i} = 1$ must be obtained by strict blowups from the sequence $(2, 1, 2)$. At any of the intermediate stages, one can either do a {\em blowup on the left}, i.e., $(m_{1}, \ldots , m_{i-1}, 1 , m_{i+1} , \ldots, m_{l}) \to (m_{1}, \ldots , m_{i-1}+1, 1 , 2, m_{i+1} , \ldots, m_{l})$, or a {\em blowup on the right}, i.e., $(m_{1}, \ldots , m_{i-1}, 1 , m_{i+1} , \ldots, m_{l}) \to (m_{1}, \ldots , m_{i-1}, 2 , 1, m_{i+1}+1 , \ldots, m_{l}).$ In particular, a null sequence as described in the statement is obtained by a series of blowups alternately on the left and on the right, and this path of blowups obviously determines the null sequence. A general null sequence obtained from $(2, 1, 2)$ by $t_1$ blowups on the left, followed by $t_2$ blowups on the right, followed by $t_3$ blowups on the left, and so on, ending with $t_{2r}$ blowups on the right, will look like $(t_{1}+2,2^{t_{2}-1},t_{3}+2, \dots, 2^{t_{2r-2}-1},t_{2r-1}+2,2^{t_{2r}},1,t_{2r}+2, 2^{t_{2r-1}-1}, \ldots, t_{2}+2,2^{t_{1}})$. 

Now, suppose $p/(p-q) = [b_{1},\ldots,b_{k}] = [n_{1}, \ldots , n_{i-1}, 2 , n_{i+1} , \ldots, n_{k}]$, where $(n_{1}, \ldots , n_{k})$ is a null sequence with only $n_{i} = 1$. By the above discussion, it follows that $(n_{1}, \ldots , n_{k})$ is obtained from $(2,1,2)$ by
$t_{1} = b_{1} - 2$ blowups to the left, followed by,
$t_{2} = b_{k - t_{1}} -2$ blowups to the right, followed by,
$t_{3} = b_{t_{2}+1}-2$ blowups to the left, followed by,
$t_{4} = b_{k - t_{1} - 1 - t_{3} + 1} - 2$ blowups to the right, and so on, until
$b_{i+1} - 2$ blowups to the right. This shows that the null sequence is uniquely determined by the continued fraction expansion of $p/(p-q)$.
\end{proof}

We now turn to fillings of lens spaces with second Betti number equal to $1$. 
\begin{proof}[Proof of Lemma~\ref{b21}]
The first statement is clear from Corollary~\ref{reinterpret}. For the second, notice that the fundamental group can be computed from the handle decomposition as $\langle x| x^s, x^b\rangle$. This is equivalent to $\langle x| x^{\gcd(s,b)}\rangle$. The result follows. 
\end{proof}

We can now prove our result about the lens spaces one obtains from Legendrian surgery on some knot in $(S^3,\xi_{std})$.

\begin{proof}[Proof of Theorem~\ref{berge}]
Since any filling of a lens space is smoothly the same as a filling of a universally tight contact structure on a lens space we consider such fillings. If a filling of a lens space comes from $B^4$ by attaching a $2$--handle to a Legendrian knot then the filling has second Betti number equal to one. Thus from Lisca's classification, in Section~\ref{lclass}, this will correspond to a lens space obtained from attaching two $2$--handles to $S^1\times D^3$ and the attaching curves of the $2$--handles are specified from a null sequence by two $-1$--framed meridians to the null sequence chain.  Notice that to get from $S^1\times D^3$ with two $2$--handles attached, to $B^4$ with one $2$--handle attached, one must have one of the $2$--handles cancelling the $1$--handle in the natural handle structure on $S^1\times D^3$. That is, one of the $2$--handles (possibly after handle slides) must link the $0$--framed unknot geometrically once. 

From Corollary~\ref{reinterpret} we know that the $2$--handles are attached on two torus knots in $S^1\times S^2$. The simplest way to cancel the $1$--handle is for one of the torus knots to be a curve of slope $n$ (that is, the curve intersects $\{p\}\times S^2$ once for each $p\in S^1$). It is clear that when one cancels the $1$--handle with such a curve then the other torus knot in $S^1\times S^2$ becomes a torus knot in $S^3$ and its framing is one less than the framing given by the torus. The torus knot must be a negative torus knot as we know from the Bennequin inequality, the positive torus knots cannot be Legendrian realized with the required contact framing. Since all negative torus knots can be realized with the required framing they must all come from this construction via Lisca's theorem, Theorem~\ref{lisca}.

The next simplest way that the $1$--handle can be cancelled is if after a single handle slide one of the $2$--handles becomes a torus knot of slope $n$. For this to happen both of the $2$--handles must have slopes with an edge in the Farey graph to $n$ and an edge to each other (if there were not an edge between the slopes then the handle slide would not give a torus knot, and if there is an edge the result of the handle slide is a Farey addition or subtraction of the slopes and so there must be edges to $n$). Using Remark~\ref{fixslope} we can assume that the $2$--handles are attached to curves of slope $\pm 1/n$ and $\pm 1/(n\pm1),$ where $n$ is a positive integer. 

We first consider the case when the slopes are $-1/n$ and $-1/(n-1)$. This $4$--manifold is shown on the left of Figure~\ref{makeb}. We then slide the $(-n-2)$--framed handle over the other to get the middle diagram in the figure. 
\begin{figure}[htb]{\tiny
\begin{overpic}
{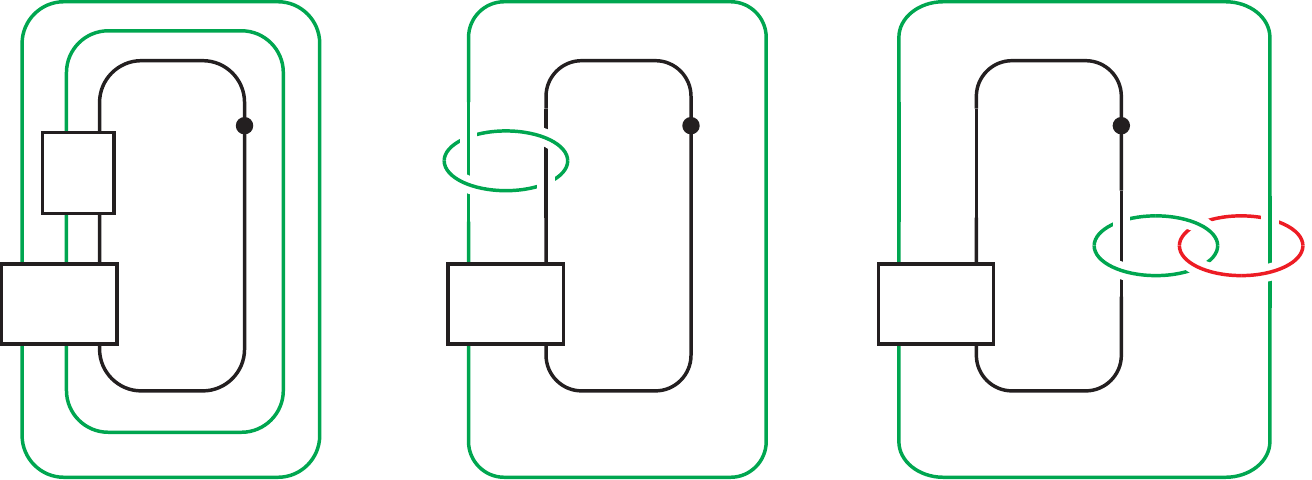}
\put(11, 48){$-n$}
\put(17, 86){$-1$}
\put(35, 16){$-n-2$}
\put(35, 4){$-n-1$}
\put(138, 48){$-n$}
\put(115, 90){$-3$}
\put(164, 5){$-n-1$}
\put(264, 48){$-n$}
\put(296, 5){$-n$}
\put(301, 65){$-2$}
\put(384, 65){$1$}
\end{overpic}}
\caption{On the left is the $-1/n$ and $-1/(n-1)$ torus knots in $S^1\times S^2$ with framings one less than that given by the tori on which they sit. Sliding the $(-n-2)$-framed $2$--handle over the other $2$--handle gives the next diagram. The last diagram is obtained by blowing up a $1$--framed unknot to unlink the two $2$--handles.}
\label{makeb}
\end{figure}
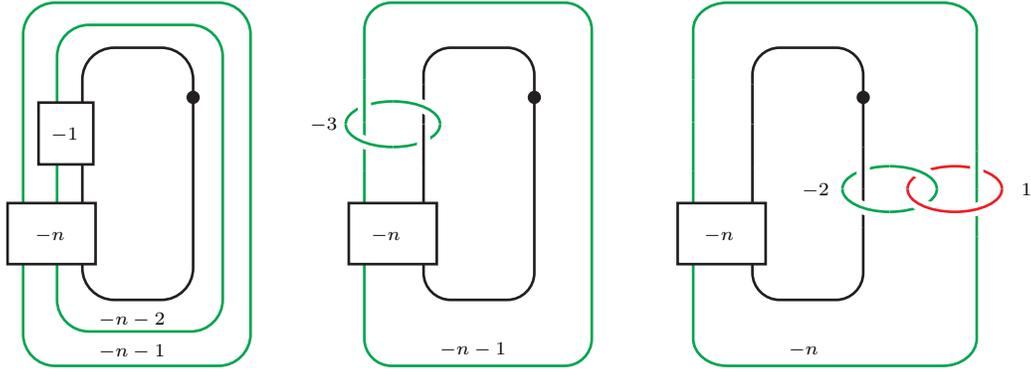
Then blowup a $1$--framed unknot to get the final diagram in the figure. Finally, sliding the $-n$--framed handle $n$ times over the $-2$--framed handle, canceling the $1$--handle with the $-2$--framed $2$--handle, and blowing down the $1$--framed handle will give the diagram in Figure~\ref{bknot}. Thus we know there are lens spaces filled by surgery on a Legendrian version of the knot in that figure. To identify these lens spaces we note that the first diagram also comes from applying Lisca's algorithm in Section~\ref{lclass} to the continued fractions $(3n^2+3n+1)/3n^2=[2,\ldots 2, 3,2,n+1]$ (where there are $n-1$, $2$s at the start) with the null sequence $(2,\ldots, 2, 1, n+1)$ (where there are $n$, $2$s at the start). 

Moving to the case of $1/n$ and $1/(n+1)$, one can perform the same computation as above and obtain a negative torus knot. 

Now, the continued fractions expansion of $-\frac{3n^{2}+3n+1}{3n+1}$
is $[-n-1,-4,-2, \ldots, -2]$ (where there are $n-1$, $-2$s at the end). Thus the lens space is represented by a surgery diagram on a link with $n+1$ unknots, call them $U_{1}, U_{2}, \ldots U_{n+1}$. A virtually overtwisted structure is represented by a chain where either $U_{1}$ and/or $U_{2}$ is stabilized on both sides, and/or they are stabilized in opposite directions. The last $n$ components will always be max tb unknots and form a consistent chain. By Theorem~\ref{list}, a filling of a virtually overtwisted structure will come either by:
\begin{itemize}
    \item attaching a 2-handle along the connect sum of rational unknots (corresponding to $U_{2}$) in the boundary of a filling of $L(n+1,1)$ (lens space corresponding to the subchain $U_{1}$) connect summed with a filling of $L(n,n-1)$ (lens space corresponding to the subchain $U_{3},\ldots,U_{n+1})$, or
    \item attaching a 2-handle along the rational unknot (corresponding to $U_{1}$) in the boundary of a filling of the lens space corresponding to the chain of unknots $U_{2},\ldots,U_{n+1}$, i.e., $L(p,q)$ where $-\frac{p}{q} = [-4,-2,\ldots,-2]$.
\end{itemize}
It is clear that in both cases, the filling cannot have $b_{2}=1$, and in particular, the virtually overtwisted contact structure cannot arise from Legendrian surgery on a Legendrian knot in $(S^{3},\xi_{std})$.
\end{proof}
\begin{remark}\label{onlyex}
From the proof above, it seems there should be no other knots in $S^3$ that come from Lisca's construction. If the two torus knots in $S^1\times S^2$ are not as in the proof, but do give a simply connected filling, then one can slide the $2$--handles so that one of them algebraically links the $1$--handle one time. However, it does not appear that one can arrange the geometric linking to be one. We were not able to prove this in any examples, but experimenting with the possibilities strongly indicates this. This is the main evidence for Conjecture~\ref{bergec}. 
\end{remark}

We will now put some restrictions on the fundamental group of a filling of a lens space, improving a result of Fossati \cite{Fossati2019pre2}.

\begin{proof}[Proof of Theorem~\ref{pi1}]
If $X$ is a Stein filling of $L(p,q)$, it is built by adding $0$--, $1$--, and $2$--handles to $B^{4}$. Turning it upside down, it is a smooth manifold built by adding 2, 3, and 4 handles to $L(p,q) \times I$. Thus, $\pi_{1}(X)$ must be a quotient of $\mathbb{Z}_{p}$. If $\pi_{1}(X) = \mathbb{Z}_{p}$, we consider the universal cover $X'$ of $X$. Then $X'$ is a Stein filling of $S^{3}$ by lifting the Stein structure through the covering map. Eliashberg tells us that $X'\cong B^4$ \cite{Eliashberg90b}.  Thus the deck transformations are fixed point free maps from $B^{4}$ to itself. But there are no such maps, by Brouwer's fixed point theorem.
\end{proof}

We now turn to the proof of Theorem~\ref{morepi1} about the fillings and covers of contact structures on lens spaces defined by Legendrian surgery on a chain of Legendrian unknots where the first or last unknot in the chain has been stabilized both positively and negatively. 
\begin{proof}[Proof of Theorem~\ref{morepi1}]
Notice that the meridian to either the first or last unknot in the chain in Figure~\ref{chain} is a core of a Heegaard torus for $L(p,q)$ and as such generates the fundamental group of $L(p,q)$. Thus in any cover of $L(p,q)$ this meridian will have to unwrap; that is, a component of the preimage of one of these curves $m$ under the covering map will have to be a non-trivial connected cover of $m$. 

It is well-known, see \cite[Proposition~5.1]{Gompf98}, that if a Legendrian knot $L$ has been stabilized both positively and negatively, then the Whitehead double of the meridian can be Legendrian realized so that it is the boundary of an immersed overtwisted disk. So if the meridian unwraps in a cover of Legendrian surgery on $L$, then that cover will be overtwisted. Thus we see that $\xi_{\mathcal{C}}$ becomes overtwisted in any cover of $L(p,q)$. 

Now if $X$ is a filling of $(L(p,q), \xi_{\mathcal{C}})$ and $X$ is not simply connected, then let $X'$ be the universal cover of $X$. This will be a symplectic filling of a cover of $(L(p,q), \xi_{\mathcal{C}})$, contradicting the above observation. 
\end{proof}
\begin{remark}\label{middlestab}
Notice that it is important that in the hypothesis of Theorem~\ref{morepi1} it is the first or last unknot in the chain that has been stabilized both positively and negatively. A family of examples showing this was first given by Marco Golla and reported in \cite{Fossati2019pre2}. One of those examples is the lens space $L(56,15)$ which is surgery on the chain $(-4,-4,-4)$, where the first and last unknots in the chain are only stabilized positively, but the middle unknot has been stabilized both positively and negatively. A filling of this contact structure can be made by taking two copies of the rational homology ball filling of surgery on the twice positively stabilized unknot, adding a $1$--handle and then adding a Stein $2$--handle to the positively and negatively stabilized middle unknot. One may readily check that this filling has $\pi_1\cong \Z/2\Z$. 

It is easy to see where the above proof breaks down in this example. The meridian to the middle unknot has order $4$ in the fundamental group of $L(56,16)$ and so in the double cover lifts to two curves and does not unwrap. 

In general, if a middle unknot in a chain is stabilized positively and negatively, then one can determine its order in the fundamental group and determine covers that must become overtwisted. 
\end{remark}

\subsection{Specific results on fillings of lens spaces}\label{specific}
We now classify symplectic fillings of lens spaces obtained by surgery on a one or two component chain. 
\begin{proof}[Proof of Theorem~\ref{cor5}]
The lens spaces obtained by negative surgery on the unknot (one component chain) are $L(p,1)$.
From McDuff \cite{McDuff90} we know that the only minimal symplectic fillings of the universally tight contact structures on $L(p,1)$ are given by the plumbing given in Figure~\ref{chain} except for $L(4,1)$ which is also filled by the manifold shown in the upper left of Figure~\ref{excpetionalfills}. Plamenevskaya and van Horn-Morris \cite{PlamenevskayaVanHorn-Morris2010} showed all the virtually overtwisted contact structures on $L(p,1)$ were only filled by the plumbing.

We now turn to the universally tight contact structures on $L(p,q)$ obtained by negative surgery on a two component chain (and we assume the surgery coefficients are less than $-1$ since otherwise we could just blowdown to get a surgery on a knot). Suppose $p/q=[n,m]$ where $n,m\geq 2$. Using  Riemenschneider point diagrams from Section~\ref{cf} we see that $p/(p-q)=[2^{n-2}, 3, 2^{m-2}]$. The only elements in $\BZ_{p,q}$ are $(1, 2^{m+n-5}, 1)$ and $(0)$ if $m+n-5=-1$, $(1,1)$ if $m+n-5=0$, $(2,1,2)$ if $m+n-4=1$, $(2,1,3,1), (2,3,1,2), (3,1,2,2),$ and $(2,2,1,3)$ if $n+m-5=2$ and $(2,1,3, 2^{n+m-7}, 1)$ and $(1, 2^{n+m-7}, 3, 1, 2)$ if $m+n-4\geq 3$. 

The first case corresponds to $L(3,2)$ and gives the filling by the plumbing. The second case will give fillings $L(p,q)$ when $p/(p-1)$ is $[2,3]$ and $[3,2]$. That is they give fillings of $L(5,2)\cong L(5,3)$ and they are simply the plumbing. For a chain of length three we get the lens spaces $L(7,4)\cong L(7,2)$ and $L(8, 3)$. Applying Lisca's algorithm to the former gives the upper left diagram in Figure~\ref{excpetionalfills} with a $2$--handle attached to a $(-1)$--framed unknot added to a meridian of the $1$--handles. Cancelling the $1$--handle with this $2$--handle yields the $(-3,2)$--torus knot with framing $-7$. The filling of $L(8,3)$ gotten from $(2,1,2)$ is shown in the upper right of Figure~\ref{excpetionalfills}.

For chains of length four we get the lens spaces $L(9,5)\cong L(9,2)$ and $L(11,4)\cong L(11, 3)$. The extra filling of the former is shown on the lower left of Figure~\ref{excpetionalfills} while the latter is surgery on the $(-5,2)$--torus knot with framing $-11$. Finally, for longer chains, we get the lens spaces $L(4k+7, k+2)\cong L(4k+7, 4)$ and get the filling given by the upper left diagram in Figure~\ref{excpetionalfills} with a $2$--handle attached to a $(-n+1)$--framed meridian to the $1$--handle. Cancelling the handles yields the $(-2n+1, 2)$--torus knot with framing $-4n+1$. 

Now turning to the virtually overtwisted contact structures we see that either one of the Legendrian knots in the plumbing diagram has to be stabilized both positively and negatively or one of the unknots is stabilized positively and the other is stabilized negatively. In both cases Theorem~\ref{list} says that any filling of this contact structure comes from a filling of the chain where one of the unknots is removed with a $2$--handle added. So the only way we can get more than one filling is if one of the components is framed $-4$ and has only been stabilized positively or negatively. In this case we are again attaching a $2$--handle to a meridian for the $1$--handle in the upper left of Figures~\ref{excpetionalfills} with framing $-n$ (or $-m$). Once again this yields a $(-2n+1,2)$--torus knot with framing $-4n+1$. The Euler classes of these resulting contact structures can easily be worked out as the rotation numbers of such torus knots, which have been given in \cite{EtnyreHonda01b}.
\end{proof}

\begin{proof}[Proof of Theorem~\ref{3component}]
\noindent
If $L(p,q)$ is obtained from a 3-component chain of unknots, then one may use the Riemenschneider point diagram from Section~\ref{cf} to see that the continued fraction expansion of $p/(p-q)$ is either $[2^{k},3,2^{l},3,2^{m}]$, where $k, l, m \geq 0$ or $[2^{k},4,2^{m}]$, if the 2nd component has framing $-2$. By Lisca's theorem, Theorem~\ref{lisca}, the possible fillings of the universally tight contact structures will 
come from null sequences that are all 1s and 2s, with the exception of either at most two 3s, or a single 4. The possible null sequences of length 1 are $(0)$, of length 2 are $(1, 1)$, of  length 3 are
$(2,1,2), (1,2,1),$
 of length 4 are
\[
 (1,2,2,1), (2,1,3,1), (1,3,1,2), (3,1,2,2), (2,2,1,3),
 \]
of length 5 are
\begin{align*}
 (1,2,2&,2,1), (3,1,2,3,1), (1,3,2,1,3), (2,3,1,2,3), (3,2,1,3,2), (2,2,2,1,4), (4,1,2,2,2),\\
 & (2,2,1,4,1), (1,4,1,2,2), (2,1,4,1,2), (1,3,1,3,1), (2,1,3,2,1), (1,2,3,1,2),
 \end{align*}
of length 6 are
\begin{align*}
 &(1,2,2,2,2,1), (2,2,1,4,2,1), (3,1,2,3,2,1), (1,2,4,1,2,2), (1,2,3,2,1,3), \\
 &(2,1,3,3,1,2), (1,3,1,3,2,1), (1,2,3,1,3,1), (1,2,2,3,1,2), (2,1,3,2,2,1),
  \end{align*}
  and  of length greater than 6 are
\begin{align*}
(1, 2^{k}&, 1), 
(1,2^{k},3,1,3,2^{j},1), 
(2,1,3,2^{k},3,1,2), (2,1,3,2^{k},1),  (1, 2^{k},3,1,2),\\
&  (3,1,2,3,2^{k},1), (1, 2^k, 3, 2, 1, 3), (2,2,1,4,2^{k},1), (1,2^{k},4,1,2,2).
  \end{align*}
Given the above, in the {universally tight case}, it is easy to now look at $L(p,q)$ obtained from a 3 component link, and for the $p/(p-q)$ expansion, find the admissible sequences from the above list, and then write down the corresponding fillings. We shall list the lens spaces by the lengths of the expansions of $p/(p-q)$. We shall use the following convention, $L(p,q) = L([p/(p-q)])$, where $[p/(p-q)]$ refers to its continued fraction expansion. In particular, we list all the lens spaces that have fillings other than the ones coming from the plumbing, i.e., ones coming from null sequences other than $(1,2, \ldots,2,1)$. For the virtually overtwisted case, we will state the results in terms of the rotation number configuration $(r_{1}, r_{2}, r_{3})$ (notice that the Thurston-Bennequin numbers of the unknots in the chain are determined by the continued fraction of $-p/q$). The structures not mentioned will admit one filling. The results for the virtually overtwisted structures follow from Theorem~\ref{list} and Theorem~\ref{cor5}. 

\smallskip

\noindent
{\bf {Size 3 chains:}} The universally tight contact structures on 
\begin{align*}
L(10,7)= L([2,2, 4]) & \cong L(10,3)=L([4, 2,2 ]), \\
L(13,5) = L([2,3,3]) &\cong L(13,8) = L([3,3,2]), \\
L(12,7) = L([3,2,3]) &  \text{ and }  L(12,5) = L([2,4,2]),\\ 
\end{align*}
 all have 2 fillings, corresponding to the null sequences $(1,2,1)$ and $(2,1,2)$. All the virtually overtwisted structures have just one filling.
 
 \noindent
{\bf {Size 4 chains:}} The universally tight contact structures on 
\begin{align*}
L(13,9) = L([4,2,2,2]) &\cong L(13,3) = L([2,2,2,4]), \\ 
L(17,7) = L([2,4,2,2]) & \cong L(17,5) = L([2,2,4,2]), \\
L(21,8)=L([2,3,3,2]),& \text{ and }  L(16,9)=L([3,2,2,3])
\end{align*} 
have one extra filling, for the first two lens spaces  these are coming from $(3,1,2,2)$, and $(1,3,1,2)$, respectively. (Here, and below, we list the null sequences for the first named lens space.) For $L(21,8)$ there are two extra null sequences $(2,1,3,1)$ and $(1,3,1,2)$, but $8^2\equiv 1 \!\!\mod 21$, so by Lisca's theorem, Theorem~\ref{lisca}, we know the fillings corresponding to $\n$ and $\overline\n$ are diffeomorphic. Similarly for $L(16,9)$ and the null sequences $(3,1,2,2)$ and $(2,2,1,3)$. 
The universally tight contact structures on 
\begin{align*}
L(18,5) = L([2,2,3,3]) &\cong L(18,11) = L([3,3,2,2]), \text{ and} \\ 
L(19,7) = L([2,3,2,3]) & \cong L(19,11) = L([3,2,3,2]) \\
\end{align*}
have two extra fillings coming from $(2,1,3,1)$ and $(2,2,1,3)$; and $(1,3,1,2)$ and $(2,2,1,3)$, respectively. 

There are two fillings of the following virtually overtwisted contact structures: 
\begin{align*}
L(17,7) &\text{ with configurations  } (\pm 1, 0, \mp 2), \\
L(18,5) &\text{ with rotation number configurations  } (\pm2, \mp1, 0), \\
L(19,7) &\text{ with configurations } (\pm 1, \mp 2, 0),  \text{ and } \\
L(21,8) &\text{ with configurations  } (\pm1, \pm1,\mp1) \text{ and } (\pm 1,\mp 1, \mp 1).\\
\end{align*}

\noindent
{\bf {Size 5 chains:}}  The universally tight contact structures on 
\begin{align*}
L(16,11) & = L([4,2,2,2,2]) \cong L(16,3) = L([2,2,2,2,4]), \\ 
L(22,9) & = L([2,4,2,2,2])  \cong L(22,5) = L([2,2,2,4,2]), \\
& \text{ and }  L(30,11)=L([2,3,2,3,2])
\end{align*} 
have one extra filling coming from $(4,1,2,2,2)$, $(1,4,1,2,2)$, and  $(1,3,1,3,1)$, respectively. 
The universally tight contact structures on 
\begin{align*}
L(25,14) & = L([3,2,2,3,2]) \cong L(25,9) = L([2,3,2,2,3]), \\ 
L(26,15) & = L([3,2,3,2,2])  \cong L(26,7) = L([2,2,3,2,3]), \\
L(29,11) & = L([2,3,3,2,2])  \cong L(29,8) = L([2,2,3,3,2]), \\
& \text{ and }  L(24,7)=L([2,2,4,2,2])
\end{align*} 
have two extra fillings, the first three coming from $(3,1,2,3,1)$ and $(3,2,1,3,2)$; $(2,1,3,2,1)$ and $(1,2,3,1,2)$; and $(2,1,3,2,1)$ and $(1,2,3,1,2)$. For $L(24,7)$ there are three extra null sequences $(2,1,4,1,2)$, $(2,1,3,2,1)$, and $(1,2,3,1,2)$, but since $7^2\equiv 1 \!\!\mod 24$ the last two give the diffeomorphic fillings. 

There are two fillings of the following virtually overtwisted contact structures: \\ 
\begin{align*}
L(22,9) &\text{ with rotation number configurations } (\pm 1, 0, \mp 3),\\ 
L(24,7) &\text{ with configurations } (0,0,\pm 2), \text{ and }  (\pm 2, 0,0),\\ 
L(25,14) &\text{ with configurations } (0,\pm 3, \mp 1),\\ 
L(26,15) &\text{ with configurations } (0,0,\pm 2) \text{ and }  (0, \pm 2, 0),\\ 
L(29,11) &\text{ with configurations } (\pm 1, \pm 1, 0), (\pm 1, \mp 1, \pm 2), \text{ and } (\pm 1, \mp 1, \mp 2), \text{ and}\\ 
L(30,11) &\text{ with configurations } (\pm 1, \mp 2, \pm 1), (\pm 1, \mp 2, \mp 1), \text{ and } (\mp 1, \mp 2, \pm 1).
\end{align*}

There are three fillings of the following virtually overtwisted contact structures: 
\begin{align*}
L(24,7) &\text{ with the configurations } (\pm 2, 0, \mp 2) \text{ and}\\
L(26,15) &\text{ with configurations } (0,\pm 2, \mp 2),\\ 
L(29,11) &\text{ with the configurations } (\pm 1, \pm 1, \mp 2).
\end{align*}

\noindent
{\bf {Size 6 chains:}} The universally tight contact structures on 
\begin{align*}
L(33,19) = L([3,2,3,2,2,2]) &\cong L(33,7) = L([2,2,2,3,2,3]), \text{ and} \\ 
L(37,14) = L([2,3,3,2,2,2]) & \cong L(37,8) = L([2,2,2,3,3,2]) \\
\end{align*}
have one extra filling both coming from $(2,1,3,2,2,1)$. 
The universally tight contact structures on 
\begin{align*}
L(31,9) & = L([2,2,4,2,2,2]) \cong L(31, 7) = L([2,2,2,4,2,2]), \\ 
L(34,19) & = L([3,2,2,3,2,2])  \cong L(34,9) = L([2,2,3,2,2,3]), \\
L(41,15) & = L([2,3,2,3,2,2])  \cong L(41,11) = L([2,2,3,2,3,2]), \\
& \text{ and }  L(40,11)=L([2,2,3,3,2,2])
\end{align*} 
have two extra fillings. For the first three lens spaces  these are  coming from $(2,1,3,2,2,1)$ and $(1,2,4,1,2,2)$; $(3,1,2,3,2,1)$ and $(1,2,2,3,1,2)$; and $(1,3,1,3,2,1)$ and $(1,2,2,3,1,2)$, respectively. For the lens space $L(40,11)$ there are three extra null sequences $(2,1,3,3,1,2)$, $(2,1,3,2,2,1)$, and $(1,2,2,3,1,2)$, but the latter two give diffeomorphic fillings since $11^2\equiv 1 \!\!\mod 40$.

There are two fillings of the following virtually overtwisted contact structures: \\ 
\begin{align*}
L(31,9) &\text{ with all rotation number configurations except } (0,0,\pm 1)  \text{ and }  (\pm 2, 0, \mp 3),\\ 
L(33,19) &\text{ with all configurations except } (0, 0, \mp 1) \text{ and }  (0,0,\pm 3),\\ 
L(34,19) &\text{ with all configurations except } (0,\pm 1, 0) \text{ and }  (0,\pm 3,\mp 2),\\ 
L(37,14) &\text{ with all configurations except } (\pm 1, \mp 1, \mp 3) \text{ and } (\pm 1, \mp 1, \pm 3),\\ 
L(40,11) &\text{ with configurations } (0,\pm 1, \pm 2), (0,\pm 1, \mp 2), (\pm 2, \pm 1, 0), \text{ and } (\pm 2, \mp 1, 0) \text{ and}\\  
L(41,15) &\text{ with either only one of } r_2 \text{ or } r_3 \text{ not-zero, or both non-zero and of the same sign}.\\ 
\end{align*}
There are three fillings of the following virtually overtwisted contact structures: \\ 
\begin{align*}
L(31,9) &\text{ with rotation number configurations } (\pm 2, 0, \mp 3), \\
L(34,19) &\text{ with configurations } (0, \pm 3, \mp 2), \\  
L(40,11) &\text{ with configurations } (\pm 2, \pm 1, \mp 2), (\pm 2, \mp 1, \pm 2), \text{ and } (\pm 2, \mp 1, \mp 2),  \text{ and}\\  
L(41,15) &\text{ with configurations} (\pm 1, \pm 2, \mp 2) \text{ and } (\pm 1, \mp 2, \pm 2).
\end{align*}

\noindent 
{\bf {Size 7 and above chains:}} 
We will consider two cases depending on whether $p/(q-p)$ is $[2^k, 3, 2^l, 3, 2^m]$ or $[2^k, 4, 2^l]$.

\noindent
{\bf Case 1:} Lens spaces $L(p,q)$ with $p/(q-p)=[2^k, 3, 2^l, 3, 2^m]$. 

If $k=2$ or $m=2$, then there is an extra filling coming from the null sequences $(2,1,3,2^l,1)$ and $(1,2^l, 3, 1, 2)$, respectively, and if both are $2$ we also have $(2,1,3,2^l,3,1,2)$.

If $l\geq 3$ or $l=0$,  then the only possible extra fillings come from those mentioned above.

If $l=2$, then when $k=0$ or $m=0$, then there are extra fillings coming from $(3,1,2,3, 2^m, 1)$ and $(1,2^k,3,2,1,3)$ in addition to the ones mentioned above. 

If $l=1$, then there are extra fillings coming from 
\[
(1,2^k, 3,1,3, 2^m, 1), (2,1,3,2,3,1,2), (2,1,3,2,2,2,2), \text{ and } (2,2,2,2,3,1,2).
\]

Thus we see the universally tight contact structure on the lens spaces 
\[
L(4(l+3)(m+2) -m-6, (l+3)(m+2)-1)=L([2, 2, 3, 2^l, 3, 2^m])
\]
\[
 \cong L(4(l+3)(m+2)-m-6, 4(l+3)-1)=L([2^m, 3, 2^l, 3,2,2])
\] 
each have one extra filling when $l\geq 3$ or $l=0$ and the continued fraction has length at least 7. The virtually overtwisted contact structures on these lens spaces will also have an extra filling if $r_1\not=0$ (as above when giving rotation number configurations for the virtually overtwisted contact structures, we will always use the chain of unknots that describes the first listed lens space). The universally tight contact structure on 
\[
L(16(l+3) -8, 4(l+3) -1)= L([2,2,3, 2^l, 3,2,2])
\]
has two extra fillings (there are three extra null sequences but since $(4(l+3)-1)^2\equiv 1 \!\!\mod (16(l-3)-8)$ two of them give diffeomorphic fillings). The virtually overtwisted contact structure on these lens spaces will have one extra filling if one of $r_1$ or $r_3$ is zero and the other is not and two extra fillings if both $r_1$ and $r_3$ are non-zero. 

The universally tight contact structure on the lens spaces 
\[
L(9(m+2)-2, 5(m+2)-1)=L([3,2,2,3,2^m])\cong L(9(m+2)-2, 9)=L([2^m,3,2,2,3]) \text{ and} 
\]
\[
L(19(m+2)-4,5(m+2)-1)=L([2,2,3,2,2,3,2^m])\cong L(19(m+2)-4,19)=L([2^m,3,2,2,3,2,2]) 
\]
each have one extra filling (we need $m\geq 1$ so the continued fraction has length at least 7). The lens space $L(72,19)$ has two extra fillings since there are three extra null sequences but since $19^2\equiv 1 \!\! \mod 72$ two of them are diffeomorphic. For the first family the virtually overtwisted contact structures with $r_2=\pm 3$ will have two fillings and for the second family the virtually overtwisted contact structures with $r_1\not =0$ will have two fillings. Virtually overtwisted contact structures on $L(72, 19)$ will have three fillings if $|r_1|=|r_3|= 2$, they will have two fillings if one of $r_1$ or $r_3$ is $\pm 2$ and the other is $0$, all others will have just one filling. 

The universally tight contact structure lens spaces 
\[
L(4(k+2)(m+2)-k-m-4, 4(m+2)-1)=L([2^k,3,2,3,2^m])
\]
have one extra filling if $m+k\geq 4$ and $k$ and $m$ are not equal to $2$. The virtually overtwisted contact structure with $r_2\not = 0$ will have two fillings. When one of $k$ or $m$ is $2$, then the universally tight contact structure on the lens spaces
\begin{align*}
L(1&5(m+2)-4, 4(m+2)-1)=L([2,2,3,2,3,2^m]) \\ & \cong L(15(m+2)-4, 15)=L([2^m,3,2,3,2,2])
\end{align*}
have two extra fillings if $m\geq 3$. The virtually overtwisted contact structures with one of $r_1$ or $r_2$ not zero but the other zero will have two filings while the ones with both $r_1$ and $r_2$ non-zero will have three fillings. When $k=m=2$ we get the lens space $L(56,15)$ and the universally tight contact structure has three extra fillings. There are actually four extra null sequences, but two of them give diffeomorphic fillings since $15^2\equiv 1 \!\! \mod 56$. The virtually overtwisted contact structures will have one extra filling if two of the $r_i$ are $0$, two extra fillings if $r_2=0$ but the others are not, and three extra fillings if all the $r_i$ are non-zero. 

\noindent
{\bf Case 2:} Lens spaces $L(p,q)$ with $p/(q-p)=[2^k, 4, 2^l]$. 

If $k=2$ or $m=2$, then there is an extra filling coming from the null sequences $(2,1,3,2^2,1)$ and $(1,2^l, 3, 1, 2)$, respectively. 

If $k=3$ or $m=3$, there there is an extra filling coming from the null sequence $(2,2,1,4,2^l,1)$ and $(1,2^l, 4,1,2,2)$, respectively. 

Thus the universally tight contact structures on the lens spaces 
\[
L(7l+10,2l+3)=L([2,2,4,2^l])\cong L(7l+10, 7)=L([2^l,4,2,2]) \text{ and}
\]
\[
L(9l+13, 2l+3)=L([2,2,2,4,2^l])\cong L(9l+13, 9)=L([2^l,4,2,2,2])
\]
have one extra filling. The virtually overtwisted contact structures on the first family will have two fillings if $r_1\not=0$ and on the second family if $r_1=\pm 3$.
\end{proof}

\section{Stein cobordisms}
We begin this section by constraining Stein cobordisms between lens spaces based on the length of the continued fractions describing the lens spaces. 

\begin{proof}[Proof of Theorem~\ref{length}]
We show that if there is a Stein cobordism from a tight contact structure on $L(p,q)$ to any contact structure on $L(p',q')$ then we must have that $l(p'/q')\geq l(p/q)$, where $l(r/s)$ is the length of the continued fractions expansion of $r/s$ as discussed in the introduction. To this end let $X$ be a Stein cobordism from a contact structure on $L(p,q)$ to any contact structure on $L(p',q')$. If $l(p'/q')< l(p/q)$ then we could use $X$ to build a Stein filling of $L(p',q')$ with second homology larger than is allowed by our main results Theorem~\ref{list}. To see this we simply take the filling $Y$ of $L(p,q)$ that has maximal second Betti number (which is given by the plumbing and has $b_2=l(p/q)$) and glue it to $X$. As $X$ and $Y$ are both Stein, so is the resulting filling of $L(p',q')$; which, moreover, has second Betti number at least that of $Y$.

Now suppose that $l(p/q)=l(p'/q')$ and we have the Stein cobordism above. From above we see that the second Betti number of $X$ is trivial.  Moreover, if we choose $Y$ to be the filling of $L(p,q)$ with largest second Betti number, then $X\cup Y$ will be the analogous filling of $L(p',q')$. Thus by Theorem~\ref{cor4} both fillings are simply connected and have the same second homology. 

We now claim that $H_2(X,L(p,q))=0$. To see this we first notice that  $H_2(X,L(p,q))$ has no torsion because $X$ is built from $[0,1]\times L(p,q)$ by attaching $1$-- and $2$--handles. Now by excision notice that $H_k(X,L(p,q))\cong H_k(X\cup Y, Y)$ and the long exact sequence of the pair $(X\cup Y, Y)$ shows that $H_k(X,L(p,q))$ is trivial except possibly when $k=2$ where it might be torsion. But since we have already seen that there is no torsion, it is trivial there as well. 
By duality, we see that all the homology groups $H_i(X,L(p',q'))$ also vanish.
It is known that two lens spaces are homology cobordant if and only if they are diffeomorphic, see \cite{DoigWehrli15pre} for a modern proof of this fact and a discussion of previous proofs.

Since these fillings are both simply connected we can use the van Kampen theorem to see that $\pi_1(L(p,q))$ must surject onto $\pi_1(X)$, and thus they are both abelian and isomorphic to $H_1(L(p,q))$ and $H_1(X)$, respectively. Now since the relative homologies $H_k(X,L(p,q))$ all vanish, we see that the inclusion map $i:L(p,q)\to X$ induces an isomorphism from $\pi_1(L(p,q))$ to $\pi_1(X)$.

We now claim that $i$ is a homotopy equivalence. To see this notice that $i$ lifts to the universal cover of $L(p,q)$ and $X$ to give a map $j: S^3\to \widetilde X$. Or more precisely, if $p:S^3\to L(p,q)$ is the covering map then $i\circ p$ lifts to $\widetilde X$. Notice that the boundary component $L(p,q)$ in $X$ is covered by an $S^3$ boundary component of $\widetilde X$ because $\pi_1(L(p,q))=\pi_1(X)$. Thus $j$ is an embedding of $S^3$ into a boundary component of $\widetilde X$. Now since $X$ is built from $[0,1]\times L(p,q)$ by attaching the same number of $1$-- and $2$--handles, we also know $\widetilde X$ is built from $[0,1]\times S^3$ by attaching the same number of $1$-- and $2$--handles. Since we know $\widetilde X$ is simply connected, we must also have that its first homology is trivial; and then by using cellular homology we can compute that its second homology is also trivial (the chain groups are generated by the handles, and the boundary map form the $2$--chain group must be onto the $1$--chain group). Of course, since there are no higher index handles, all the higher relative homology groups also vanish. Thus we see that $j$ induces an isomorphism on all the homology groups $H_k(S^3)\to H_k(\widetilde X)$ and Whitehead's theorem implies $j$ is a homotopy equivalence and hence induces an isomorphism on all the homotopy groups. Of course the projections $p$ and $p'$ also induce isomorphisms on all the higher homotopy groups, and hence so does $i$. But we already knew that $i$ induces an isomorphism on $\pi_1$. Thus $i$ is a homotopy equivalence and $X$ is an $h$-cobordism (since we can make the same arguments for the inclusion of $L(p',q')$ into $X$).

To complete the proof suppose we are given a Stein cobordism $X$ from a contact structure $\xi$ on $L(p,q)$ to $\xi'$ on the same lens space, from above we know that $X$ is a homology cobordism and thus there is no second cohomology. In \cite{Gompf98}, Gompf introduced invariants of homotopy classes of oriented plane fields. One of these was a refinement of the Chern class of the plane field. He also gave formulas to see how they would change under a Stein cobordism. Since this is a homology cobordism, the invariants must be the same. But from the classification of contact structures on lens spaces, two tight contact structures are contactomorphic if and only if their invariants are the same. Thus $\xi$ is the same as $\xi'$. 
\end{proof}

\section{Cobordisms between lens spaces}\label{setion:cobordism}
In this section we  more thoroughly discuss Construction~3 of Stein cobordisms between lens spaces from the end of the introduction. All the relevant details about the other constructions were discussed in the introduction. Construction~3 is based on surgeries on torus knots. Much of the discussion below comes from the upcoming paper \cite{BakerEtnyreMinOnaran21pre} where more on Legendrian torus knots in lens spaces can be found. 

Suppose $\xi$ is a contact structure on $T^2\times [0,1]$ determined by a path in the Farey graph that is a continued fraction block of length $2n$. We will call $\xi$ a \dfn{balanced continued fraction block} if it has the same number of positive signs as negative signs. The slope of the $n^{th}$ vertex in the continued fraction block is called the \dfn{central slope}. In \cite{ChakrabortyEtnyreMin20pre} the following result was proven.
\begin{theorem}[Chakraborty, Etnyre, and Min, 2020 \cite{ChakrabortyEtnyreMin20pre}]
Let $\xi$ be a minimally twisting tight contact structure on $T^2\times [0,1]$. There is a Legendrian knot $L$ smoothly isotopic to a $(p,q)$--curve on $T^2\times \{pt\}$ with contact twisting $n$ larger than the framing coming from the torus if and only if the path in the Farey graph describing $\xi$ contains a balanced continued fraction block of length $2n$ with central slope $q/p$. 
\end{theorem}
We will call an $L$ as in the theorem with $m>0$ a \dfn{Legendrian large torus knot}. Notice that according to this theorem, only some virtually overtwisted contact structures on a Lens space can have Legendrian large torus knots; specifically, the ones coming from a path in the Farey graph with a balanced continued fraction block. In particular, if $\xi$ on $L(p,q)$ is determined by surgery on a chain $\mathcal{C}$ of Legendrian unknots, then there will be a Legendrian torus knot with contact framing $1$ larger, respectively $2$ larger, than the torus framing if and only if one of the Legendrian knots in $\mathcal{C}$ has been positively and negatively stabilized at least $1$, respectively $2$, times. 

Recall from Section~\ref{koclassification}, that we can describe a lens space as $T^2\times [0,1]$ with slope $s/r$ curves collapsed on $T^2\times \{1\}$ and slope $p/q$ curves collapsed on $T^2\times \{0\}$ and we called $r/s$ the upper meridian and $p/q$ the lower meridian. The standard way to represent $L(p,q)$ is with an upper meridian $0$ and lower meridian $p/q$. We can use the same slope convention to talk about torus knots in $L(p,q)$. Suppose $a/b$ is slope of a torus knot that has a Legendrian representative $L$ with contact framing $1$ larger than the torus framing. Then as discussed at the end of the introduction we know that Legendrian surgery on $L$ will be the manifold $L(a,b)\#L(c,d)$ where $c/d$ is determined as follows: change the basis of $T^2$ so that $a/b$ becomes zero and $p/q$ becomes number less than or equal to $-1$, this latter number will be $c/d$. 

To pin down the contact structure on the summands, let ${P}$ be the decorated path in the Farey graph corresponding to $\xi$ by Theorem~\ref{classifylpq}. There will be a  length $2$ continued fraction block $P'$ with central slope $a/b$ corresponding to $L$ in $P$. Let $P_1$ be the part of the path from $p/q$ to the start of $P'$ and $P_2$ the part of $P$ from the end of $P'$ to $0$. The contact structure on $L(a,b)$ will correspond to the path $P_2$ extended by one jump to $a/b$. The contact structure on $L(c,d)$ will correspond to the path with one jump from $a/b$ to the end of $P_1$ followed by $P_1$. One can see this since the surgery on $L$ takes place entirely in the $T^2\times I$ corresponding to $P'$ and hence the contact structure on the complementary tori of $T^2\times I$ are unchanged; moreover, the surgery on $L$ inside of $T^2\times I$ yields the connected sum of two tight contact structures on solid tori and considering the dividing slopes on the boundary and meridian slopes, we see the contact structures on the solid tori are unique. 

Now suppose $b/a$ is the slope of a torus knot that has a Legendrian representative $L$ with contact framing $2$ larger than the torus framing. Then as discussed at the end of the introduction Legendrian surgery on $L$ gives the lens space obtained from $L(p,q)$ by cutting along a Heegaard torus and re-gluing by a negative Dehn twist along the $(a,b)$--curve. One may check that the resulting lens space still has upper meridian $0$ but the lower meridian will be $d/c$ where $d/c$ is determined by 
\[
\begin{bmatrix}c\\ d\end{bmatrix}=
\begin{bmatrix}1-ab & a^2\\ -b^2& 1+ab\end{bmatrix}\begin{bmatrix}q\\ -p\end{bmatrix}.
\]
To determine the contact structure on $L(d,c)$ we observe the following lemma.
\begin{lemma}
If $(T^2\times [0,1],\xi)$ is determined by a length $4$, balanced continued fraction block with central slope $r/s$ and $L$ is the Legendrian $(s,r)$--torus knot with contact twisting $2$ larger than the torus framing, then Legendrian surgery on $L$ will result in an $[0,1]$-invariant contact structure on $T^2\times [0,1]$. 
\end{lemma}
\begin{proof}
We can suppose (by choosing an appropriate basis for the first homology of $T^2$) that the dividing slope $s_0$ on $T^2\times \{0\}$ is $-4$ and the slope $s_1$ on $T^2\times \{1\}$ is $0$. Then we will have that $L$ is a $(1,-2)$--torus knot and Legendrian surgery on $L$ will produce a contact structure on $T^2\times [0,1]$ with dividing slopes $s'_0=0$ and $s_1'=s_1=0$. Since the original $\xi$ can be embedded in the tight contact structure on $S^3$ (consider the complement of Legendrian unknot with $tb=-5$ and rotation number $0$) we know that the resulting contact structure must be tight and minimally twisting. Thus it must be $[0,1]$-invariant. 
\end{proof}
Arguing as we did for the case above we see the contact structure on $L(d,c)$ will correspond to the path in the Farey graph obtained by removing the length $4$ continued fraction block corresponding to $L$, keeping the part of the path between $0$ and it the same, and applying the above transform to the vertices of the path between the continued fraction block and $p/q$ (and keeping the signs of the corresponding jumps the same). 

\begin{example}
Consider the contact structure $\xi$ on $L(6,1)$ obtained by Legendrian surgery on the Legendrian unknot with $tb=-5$ and rotation number $0$. The path in the Farey graph corresponding to $\xi$ has vertices $-6,-5,-4,-3,-2, -1, $ and $0$. The first and last edges are undecorated, and there are two $+$ signs and two $-$ signs on the other jumps (since the path is a continued fraction block, the signs can be put on any of the edges). Thus we see there are Legendrian knots $L_1$ and $L_3$ in the knot type of the $(1,-2)$-- and $(1,-4)$--torus knot, respectively, that have contact framing $1$ more than the torus framing, and a Legendrian knot $L_2$ in the knot type of the $(1,-3)$--torus knot with contact framing $2$ larger than the torus framing.  From the discussion above, Legendrian surgery on $L_1$ and $L_3$ will produce $(L(1,2), \xi_1)\# (L(4,1), \xi_2)$ where $\xi_1$ is the unique tight contact structure on $L(2,1)$ and $\xi_2$ is the unique virtually overtwisted contact structure on $L(4,1)$. Legendrian surgery on $L_2$ will result in the unique contact structure on $L(3,2)$. One may easily draw Legendrian surgery diagrams for $L_1$ and $L_2$ from our discussion at the end of the introduction. It is an interesting exercise to find a surgery diagram for $L_3$. We also leave it as an interesting exercise to determine that Legendrian surgery on $L_2$ after a single stabilization will result in $(L(3,1),\xi_3)\#(L(3,1),\xi_4)$ where $\xi_3$ is universally tight and $\xi_4=-\xi_3$. 
\end{example}

We now show that one can always find a Stein cobordism from a contact structure on a lens space to a contact structure on a lens space given by surgery on a nicely stabilized chain. 
\begin{proof}[Proof of Theorm~\ref{makenice}]
Let $\mathcal{C}=\{L_1, \ldots, L_n\}$ be a chain of Legendrian unknots. Suppose that $L_i$ has been stabilized both positively and negatively. Then as in Figure~\ref{reduciblesurgery} we can find a Legendrian knot $L$ in $(L(\mathcal{C}), \xi_\mathcal{C})$ to which one can attach a $2$--handle to get a cobordism to the connected sum of the lens spaces $(L(\mathcal{C}'), \xi_{\mathcal{C}'})$ and $(L(\mathcal{C}''), \xi_{\mathcal{C}''})$, where $\mathcal{C}'=\{L_1,\ldots, L_{i-1},U\}$ where $U$ is a maximal Thurston-Bennequin invariant unknot and $\mathcal{C}''=\{L_i', L_{i+1},\ldots, L_n\}$ where $L'_i$ is obtained from $L_i$ by destabilizing positively and negatively. We can now attach a $2$--handle to another unknot to get a cobordism to $(L(\mathcal{C}'''), \xi_{\mathcal{C}'''})$ where $\mathcal{C}'''=\{L_1,\ldots, L_{i-1},U, U, L'_i, L_{i+1}, \ldots, L_n\}$. We can continue in this way until we have a nicely stabilized chain. 
\end{proof}

We end by discussing surgeries on torus knots with framing one less than the torus framing. Suppose $b/a$ is the slope of a torus knot that has a Legendrian representative $L$ with contact framing equal to the torus framing. Then as discussed at the end of the introduction, Legendrian surgery on $L$ gives the lens space obtained from $L(p,q)$ by cutting along a Heegaard torus and re-gluing by a positive Dehn twist along the $(a,b)$--curve. One may check that the resulting lens space still has upper meridian $0$ but the lower meridian will be $d/c$ where $d/c$ is determined by 
\[
\begin{bmatrix}c\\ d\end{bmatrix}=
\begin{bmatrix}1+ab & -b^2\\ a^2& 1-ab\end{bmatrix}\begin{bmatrix}q\\ -p\end{bmatrix}.
\]
That is, one will obtain the $L(d,c)$ lens space. One can determine the contact structure on $L(d,c)$ as follows. The contact structure on $L(p,q)$ is determined by a path in the Farey graph. The slope $b/a$ might not be an end point of one of the jumps in the path, but one can lengthen the path so that it is (a basic slice can be split into two basic slices of the same sign). Now arguing as above, the path determining the contact structure on $L(d,c)$ agrees with the given path between $0$ and $b/a$ and the above transformation is applied to all the vertices between $b/a$ and $-p/q$ (and the signs of the jumps are kept the same). 

\begin{example}
Consider the universally tight contact structure on $L(3,2)$. One can realize the $(4,-5)$-curve by a Legendrian knot with contact framing agreeing with the torus framing. Legendrian surgery on this knot will produce a universally tight contact structure on $L(13,10)$ (one can arrange to get either universally tight contact structure by choosing the appropriate Legendrian realization of the $(4,-5)$-curve). 
\end{example}

\def\cprime{$'$} \def\cprime{$'$}


\begin{thebibliography}{10}

\bibitem{AcetoMcCoyPark2020pre}
Paolo Aceto, Duncan McCoy, and JungHwan Park.
\newblock Non-simply connected symplectic fillings of lens spaces, 2020.

\bibitem{Avdek13}
Russell Avdek.
\newblock Contact surgery and supporting open books.
\newblock {\em Algebr. Geom. Topol.}, 13(3):1613--1660, 2013.

\bibitem{BakerEtnyre12}
Kenneth Baker and John Etnyre.
\newblock Rational linking and contact geometry.
\newblock In {\em Perspectives in analysis, geometry, and topology}, volume 296
  of {\em Progr. Math.}, pages 19--37. Birkh\"auser/Springer, New York, 2012.

\bibitem{BakerEtnyreMinOnaran21pre}
Kenneth~L. Baker, John~B. Etnyre, Hyunki Min, and Sinem Onaran.
\newblock Torus knots in lens spaces.
\newblock In preparation.

\bibitem{ChakrabortyEtnyreMin20pre}
Apratim Chakraborty, John~B. Etnyre, and Hyunki Min.
\newblock Cabling legendrian and transverse knots, 2020.

\bibitem{ChristianLi20Pre}
Austin Christian and Youlin Li.
\newblock Some applications of menke's {JSJ} decomposition for symplectic
  fillings.
\newblock Preprint 2020.

\bibitem{DingGeiges09}
Fan Ding and Hansj{\"o}rg Geiges.
\newblock Handle moves in contact surgery diagrams.
\newblock {\em J. Topol.}, 2(1):105--122, 2009.

\bibitem{DoigWehrli15pre}
Margaret Doig and Stephan Wehrli.
\newblock A combinatorial proof of the homology cobordism classification of
  lens spaces, 2015.

\bibitem{Eliashberg90b}
Yakov Eliashberg.
\newblock Filling by holomorphic discs and its applications.
\newblock In {\em Geometry of low-dimensional manifolds, 2 (Durham, 1989)},
  volume 151 of {\em London Math. Soc. Lecture Note Ser.}, pages 45--67.
  Cambridge Univ. Press, Cambridge, 1990.

\bibitem{Eliashberg04}
Yakov Eliashberg.
\newblock A few remarks about symplectic filling.
\newblock {\em Geom. Topol.}, 8:277--293 (electronic), 2004.

\bibitem{Etnyre04a}
John~B. Etnyre.
\newblock On symplectic fillings.
\newblock {\em Algebr. Geom. Topol.}, 4:73--80 (electronic), 2004.

\bibitem{EtnyreHonda01b}
John~B. Etnyre and Ko~Honda.
\newblock Knots and contact geometry. {I}. {T}orus knots and the figure eight
  knot.
\newblock {\em J. Symplectic Geom.}, 1(1):63--120, 2001.

\bibitem{EtnyreHonda02a}
John~B. Etnyre and Ko~Honda.
\newblock On symplectic cobordisms.
\newblock {\em Math. Ann.}, 323(1):31--39, 2002.

\bibitem{EtnyreHonda03}
John~B. Etnyre and Ko~Honda.
\newblock On connected sums and {L}egendrian knots.
\newblock {\em Adv. Math.}, 179(1):59--74, 2003.

\bibitem{EtnyreTosun20pre}
John~B. Etnyre and B{\"u}lent Tosun.
\newblock Homology spheres bounding acyclic smooth manifolds and symplectic
  fillings, 2020.

\bibitem{Fossati19pre}
Edoardo Fossati.
\newblock Contact surgery on the hopf link: classification of fillings, 2019.

\bibitem{Fossati2019pre2}
Edoardo Fossati.
\newblock Topological constraints for stein fillings of tight structures on
  lens spaces, 2019.

\bibitem{Gay02a}
David~T. Gay.
\newblock Symplectic 2-handles and transverse links.
\newblock {\em Trans. Amer. Math. Soc.}, 354(3):1027--1047 (electronic), 2002.

\bibitem{GeigesOnaran18}
Hansj\"{o}rg Geiges and Sinem Onaran.
\newblock Legendrian lens space surgeries.
\newblock {\em Michigan Math. J.}, 67(2):405--422, 2018.

\bibitem{Giroux00}
Emmanuel Giroux.
\newblock Structures de contact en dimension trois et bifurcations des
  feuilletages de surfaces.
\newblock {\em Invent. Math.}, 141(3):615--689, 2000.

\bibitem{GollaStarkston19pre}
Marco Golla and Laura Starkston.
\newblock The symplectic isotopy problem for rational cuspidal curves, 2019.

\bibitem{Gompf98}
Robert~E. Gompf.
\newblock Handlebody construction of {S}tein surfaces.
\newblock {\em Ann. of Math. (2)}, 148(2):619--693, 1998.

\bibitem{Hind03}
R.~Hind.
\newblock Stein fillings of lens spaces.
\newblock {\em Commun. Contemp. Math.}, 5(6):967--982, 2003.

\bibitem{Honda00a}
Ko~Honda.
\newblock On the classification of tight contact structures. {I}.
\newblock {\em Geom. Topol.}, 4:309--368 (electronic), 2000.

\bibitem{Kaloti13pre}
Amey Kaloti.
\newblock Stein fillings of planar open books, 2013.

\bibitem{LekiliMaydanskiy14}
Yank\i Lekili and Maksim Maydanskiy.
\newblock The symplectic topology of some rational homology balls.
\newblock {\em Comment. Math. Helv.}, 89(3):571--596, 2014.

\bibitem{Lisca08}
Paolo Lisca.
\newblock On symplectic fillings of lens spaces.
\newblock {\em Trans. Amer. Math. Soc.}, 360(2):765--799, 2008.

\bibitem{McDuff90}
Dusa McDuff.
\newblock The structure of rational and ruled symplectic {$4$}-manifolds.
\newblock {\em J. Amer. Math. Soc.}, 3(3):679--712, 1990.

\bibitem{menke18pre}
Michael Menke.
\newblock A {JSJ}-type decomposition theorem for symplectic fillings, 2018.

\bibitem{Plamenevskaya12}
Olga Plamenevskaya.
\newblock On {L}egendrian surgeries between lens spaces.
\newblock {\em J. Symplectic Geom.}, 10(2):165--181, 2012.

\bibitem{PlamenevskayaVanHorn-Morris2010}
Olga Plamenevskaya and Jeremy Van Horn-Morris.
\newblock Planar open books, monodromy factorizations and symplectic fillings.
\newblock {\em Geom. Topol.}, 14(4):2077--2101, 2010.

\bibitem{Riemenschneider74}
Oswald Riemenschneider.
\newblock Deformationen von {Q}uotientensingularit\"{a}ten (nach zyklischen
  {G}ruppen).
\newblock {\em Math. Ann.}, 209:211--248, 1974.

\bibitem{Rolfsen76}
Dale Rolfsen.
\newblock {\em Knots and links}.
\newblock Publish or Perish Inc., Berkeley, Calif., 1976.
\newblock Mathematics Lecture Series, No. 7.

\bibitem{Schoenenberger05}
Stephan Sch\"onenberger.
\newblock {\em Planar open books and symplectic fillings}.
\newblock PhD thesis, University of Pennsylvania, 2005.

\bibitem{Simone16pre}
Jonathan Simone.
\newblock Symplectically replacing plumbings with euler characteristic 2
  4-manifolds, 2016.

\bibitem{Waldhausen68}
Friedhelm Waldhausen.
\newblock Heegaard-{Z}erlegungen der {$3$}-{S}ph\"{a}re.
\newblock {\em Topology}, 7:195--203, 1968.

\bibitem{Wendl10}
Chris Wendl.
\newblock Strongly fillable contact manifolds and {$J$}-holomorphic foliations.
\newblock {\em Duke Math. J.}, 151(3):337--384, 2010.

\end{thebibliography}

\end{document}